%% file: Engel_quotients_main.tex
\documentclass[reqno]{amsart}

\input{preamble.tex}

\title[Failure of the $\MCP$ via quotients in higher-step sub-Riemannian structures]{Failure of the measure contraction property via quotients in higher-step sub-Riemannian structures}

\author{Samuël Borza}
\author{Luca Rizzi}

\address{University of Vienna, Faculty of Mathematics,
  Oskar-Morgenstern-Platz 1,
1090 Vienna, Austria}
\email{\href{mailto:samuel.borza@univie.ac.at}{samuel.borza@univie.ac.at}}

\address{SISSA, via Bonomea 265, 34136 Trieste, Italy}
\email{\href{mailto:lrizzi@sissa.it}{lrizzi@sissa.it}}

\date{\today}

\begin{document}

\begin{abstract}

  We prove that the synthetic Ricci curvature lower bound known as the measure contraction property ($\MCP$) can fail in sub-Riemannian geometry. This may happen beyond step two, if the distance function is not Lipschitz in charts, and it already occurs in fundamental examples such as the Martinet and Engel structures.

  Central to our analysis are new results, of independent interest, on the stability of the local $\MCP$ under quotients by isometric group actions for general metric measure spaces, developed under a weaker variant of the essential non-branching condition which, in contrast with the classical one, is implied by the minimizing Sard property in sub-Riemannian geometry.

  As an application, we find sub-Riemannian structures with pre-medium-fat distribution that do not satisfy the $\MCP$, answering a question raised in \cite{R-subdiff}.
  Finally, and quite unexpectedly, we show that ideal sub-Riemannian structures can fail the $\MCP$, and this actually happens generically for rank greater than $3$ and high dimension.
\end{abstract}

\maketitle

\setcounter{tocdepth}{1}
\tableofcontents

\input{introduction.tex}    	
\input{preliminaries.tex}   	
\input{deltaenb.tex}       		
\input{discreteactions.tex}  	
\input{compactactions.tex}   	
\input{carnot.tex}          	
\input{cornucopia.tex}       	
\input{generic.tex}				
\input{martinet-engel.tex}   	
\input{variants.tex}			

\appendix
\crefalias{section}{appendix}
\input{srgeo.tex}       		
\input{goh.tex}             	
\input{lemmaproba.tex}       	

\bibliographystyle{alphaabbr}

\bibliography{Engel_quotients_biblio}
\vspace{0.5cm}

\end{document}

%% file: preamble.tex
\usepackage{amsmath,amssymb}
\usepackage{mathrsfs}
\usepackage{amsthm}

\usepackage{bbm}
\usepackage[bb=libus]{mathalpha}

\usepackage[english]{babel}
\usepackage[utf8]{inputenc}
\usepackage[T1]{fontenc}
\usepackage{mathtools}
\usepackage{lmodern}

\usepackage[top=3cm,bottom=3cm,left=3.5cm,right=3.5cm,marginparwidth=2.5cm]{geometry}

\usepackage[dvipsnames]{xcolor}
\usepackage[colorinlistoftodos]{todonotes}
\usepackage[hypertexnames=false]{hyperref}
\hypersetup{
  colorlinks,
  linkcolor={red!50!black},
  citecolor={blue!50!black},
  urlcolor={blue!80!black},
}
\usepackage[stretch=10]{microtype}

\usepackage{tikz-cd}
\usetikzlibrary{decorations.markings}
\tikzset{->-/.style={decoration={
    markings,
    mark=at position #1 with {\arrow{>}}},postaction={decorate}}}
\tikzset{-<-/.style={decoration={
    markings,
    mark=at position #1 with {\arrow{<}}},postaction={decorate}}}

\usepackage[shortlabels]{enumitem}
\setlist[1]{itemsep=2pt}

\usepackage[capitalise]{cleveref}

\usepackage{mathtools}

\usepackage{autonum} 

\newcommand{\R}{\mathbb{R}}

\newcommand{\N}{\mathbb{N}}

\newcommand{\e}{\mathfrak{e}}	
\newcommand{\g}{\mathfrak{g}}	
\newcommand{\h}{\mathfrak{h}}	
\renewcommand{\k}{\mathfrak{k}}	
\newcommand{\Heis}{\mathbb{Heis}}	
\newcommand{\Gru}{\mathbb{Gru}}	
\newcommand{\Eng}{\mathbb{Eng}}	
\newcommand{\G}{\mathbb{G}}		
\newcommand{\M}{\mathbb{M}}		
\renewcommand{\H}{\mathbb{H}}	
\renewcommand{\S}{\mathbb{S}}	
\newcommand{\K}{\mathbb{K}}		
\newcommand{\Mar}{\mathbb{Mar}}	
\newcommand{\MCP}{\mathrm{MCP}}	
\newcommand{\CD}{\mathrm{CD}}	
\newcommand{\loc}{\mathrm{loc}}	
\newcommand{\distr}{\mathscr{D}}
\newcommand{\di}{\mathrm{d}}	
\newcommand{\Cut}{\mathrm{Cut}}
\newcommand{\p}{\mathsf{p}}
\newcommand{\q}{\mathsf{q}}
\newcommand{\Htilde}{{\tilde{\H}}}

\renewcommand{\subset}{\subseteq} 

\newcommand{\sfd}{\mathsf{d}}	
\newcommand{\mm}{\mathsf{m}}	

\newtheorem{theorem}{Theorem}[section]
\newtheorem{proposition}[theorem]{Proposition}
\newtheorem{corollary}[theorem]{Corollary}
\newtheorem{lemma}[theorem]{Lemma}

\newtheorem*{conjecture}{Conjecture}

\theoremstyle{definition}
\newtheorem{definition}[theorem]{Definition}

\theoremstyle{remark}
\newtheorem{remark}[theorem]{Remark}

\crefalias{issue}{enumi}
\crefname{issue}{issue}{issues}
\Crefname{issue}{Issue}{Issues}

\AddToHook{env/lemma/begin}{\crefalias{theorem}{lemma}}
\AddToHook{env/proposition/begin}{\crefalias{theorem}{proposition}}
\AddToHook{env/corollary/begin}{\crefalias{theorem}{corollary}}
\AddToHook{env/definition/begin}{\crefalias{theorem}{definition}}
\AddToHook{env/remark/begin}{\crefalias{theorem}{remark}}

\newcommand{\red}{\color{red}}
\newcommand{\green}{\color{PineGreen}}

\DeclareMathOperator{\Ad}{Ad}
\DeclareMathOperator{\ad}{ad}

\DeclareMathOperator{\Geo}{Geo}

\DeclareMathOperator{\supp}{supp}
\DeclareMathOperator{\OptGeo}{OptGeo}
\DeclareMathOperator{\Cpl}{Cpl}
\DeclareMathOperator{\Opt}{Opt}
\DeclareMathOperator*{\essup}{ess\,sup}
\DeclareMathOperator{\diam}{diam}
\DeclareMathOperator{\End}{\mathrm{End}}

\renewcommand{\P}{\mathcal{P}}

\DeclareMathOperator{\sech}{sech}
\DeclareMathOperator{\EllipticK}{K}
\DeclareMathOperator{\JacobiK}{K}

\DeclareMathOperator{\JacobiE}{E}
\DeclareMathOperator{\JacobiSN}{sn}
\DeclareMathOperator{\JacobiCN}{cn}
\DeclareMathOperator{\JacobiDN}{dn}
\newcommand{\J}{\mathcal{J}}

%% file: introduction.tex
\section{Introduction}

In this work, we investigate the validity of synthetic Ricci curvature lower bounds in sub-Riemannian geometry. One such condition that has drawn considerable interest in the last two decades is the curvature-dimension inequality, introduced by Lott–Villani and Sturm in \cite{LV-Ricci,S-ActaI,S-ActaII}. It is now well-understood that no sub-Riemannian manifold (that is not Riemannian) can satisfy the curvature-dimension condition \cite{J09,J-proceeding,AS-heatandentropy,J21,MR-ARnoCD,RS-failure}. Quite surprisingly, however, some sub-Riemannian structures are known to satisfy a weaker condition known as the measure contraction property -- $\MCP$ for short.

The $\MCP$ was independently introduced by Ohta in \cite{O-MCP} and Sturm in \cite{S-ActaII}. It depends on two parameters, $K \in \R$ and $N \in [1,\infty)$, representing respectively a Ricci curvature lower bound and a dimension upper bound. While the curvature-dimension is a condition on the optimal transport between any pair of absolutely continuous probability measures on a metric measure space, the $\MCP$ is a condition on the optimal transport between Dirac masses and uniform distributions. The $\MCP$ shares many quantitative geometric consequences with the curvature-dimension condition, and thus it may replace the latter in sub-Riemannian geometry.

Since the $\MCP$ is inherited by metric measure tangent cones, particular attention has been devoted to this case. The (metric measure) tangent cones of sub-Riemannian metric measure spaces are known as \emph{Carnot homogeneous spaces}, and they are \emph{Carnot groups} at regular points. Following the pioneering work of Juillet in \cite{J09}, who established the $\MCP$ in the Heisenberg groups with sharp values of the parameters, a lot of effort has been devoted to proving the $\MCP$ for other sub-Riemannian structures. The $\MCP(K,N)$ -- for optimal values of the parameters $K$ and $N$ -- was obtained for ideal, corank $1$ and generalized $H$-type Carnot groups, in \cite{R-Ricci,R-MCPCarnot,BR-Htype}, respectively, and for the Grushin plane in \cite{BR-Interpolation}, which is the simplest Carnot homogeneous space that is not a Carnot group. The $\MCP$ also holds for sub-Riemannian structures that are not Carnot homogeneous spaces, such as Sasakian, 3-Sasakian or $H$-type structures, provided that they satisfy suitable curvature bounds, see \cite{AAPL,LLZ-Sasakian,LL-Bishop,RS-3Sasakian,BGKT-sublaplacian,BR-BakryEmery,BGRV-Htypecomparison,BMR-Unification} and references therein.

More recently, Badreddine and Rifford proved in \cite{BR-MCP} that real-analytic sub-Riemannian structures whose distance is locally Lipschitz in charts away from the diagonal satisfy the $\MCP$. This is the case for all Carnot homogeneous spaces of step
$\leq 2$.  Note though that the general results obtained in \cite{BR-MCP} do not provide explicit ranges for the parameters $K$ and $N$ for which the $\MCP(K,N)$ holds. While it is clear (by the existence of metric dilations) that $K\leq 0$ for Carnot homogeneous spaces, various quantitative lower bounds have been obtained for the best (i.e.\ the lowest) value of $N$ such that $\MCP(0,N)$ holds, see \cite{R-Ricci,R-MCPCarnot,NicolussiZhang,Zhang}.

All this evidence led to the belief that the $\MCP$ should be satisfied -- at least locally -- by all sub-Riemannian spaces and therefore by any Carnot homogeneous space. However, the analysis has so far  been limited to \emph{Lipschitz structures}, namely those whose distance function is locally Lipschitz in charts outside of the diagonal. In this paper, we prove the rather surprising fact that as soon as the distance function is \emph{not} Lipschitz in charts, the $\MCP$ may fail;  this occurs already in fundamental examples such as the Martinet and Engel structures.

In the rest of this introduction, we present and discuss our results.

\subsection{Failure of the measure contraction property for Martinet}

The Martinet structure is of fundamental importance in sub-Riemannian geometry. It is a Carnot homogeneous space -- the simplest one of step $3$ -- and the most elementary non-Lipschitz structure \cite{martinetagrachev}. We refer to \cref{sec:EngMart} for precise definitions. The first result we report corresponds to \cref{thm:Martinet-noMCP}.

\begin{theorem}[Failure of the $\MCP$ for the Martinet structure]\label{thm:Martinet-noMCP-intro}
  The Martinet structure does not satisfy the $\MCP(K, N)$ for any $K \in \mathbb{R}$ and any $N \in [1,\infty)$.
\end{theorem}

The proof of \cref{thm:Martinet-noMCP-intro} relies upon the optimal synthesis of the Martinet structure, which was obtained in \cite{martinetagrachev}.
\cref{thm:Martinet-noMCP-intro} is actually the last result we prove in this paper, but it is the starting point and motivation for the rest of our work.

\subsection{Measure contractions for quotients}

A powerful method to build new examples of structures satisfying curvature bounds consists in taking quotients by isometric group actions. One well-known result concerns lower bounds for the sectional curvature for quotients by proper, free, and isometric actions on Riemannian manifolds, having the structure of a Riemannian submersion. In this setting, a lower bound for the sectional curvature of the total space is inherited by the base space thanks to O'Neill's formulas \cite{Besse}. The passage to the quotient is more delicate for Ricci curvature lower bounds, as the measure also plays a role \cite{Lott-BE,ProWilhelm}.

On general metric measure spaces, the stability of the curvature-dimension conditions under isometric group actions was studied in \cite{GKMS-quotients}. It is mentioned in \cite[Rmk.\ 3.8]{GKMS-quotients} that also the $\MCP$ descends to suitable quotients, provided that the metric measure space is \emph{essentially non-branching}, see \cref{def:enb}. We note though that a complete proof of this fact, following the arguments of \cite{GKMS-quotients}, requires additional non-trivial arguments from \cite{CM-OptMaps}. We state here the corresponding result, anticipating that, for our purposes, a more general version of \cref{thm:GKMS-remark} will be needed.

\begin{theorem}[\cite{GKMS-quotients,CM-OptMaps}]\label{thm:GKMS-remark}
  Let $(X,\sfd,\mm)$ be an essentially non-branching metric measure space satisfying the $\MCP(K,N)$ for some $K\in \R$ and $N \in [1,\infty)$. Let $G$ be a compact Lie group acting freely by metric measure isometries on $(X,\sfd,\mm)$. Then the metric measure quotient $(X^*,\sfd^*,\mm^*)$ satisfies the $\MCP(K,N)$.
\end{theorem}

For applications in sub-Riemannian geometry, the typical situation we want to deal with is that of a quotient by an isometric group action
\begin{equation}\label{eq:projintro}
  \pi :\M\to\tilde{\M},
\end{equation}
where both $\M$ and $\tilde{\M}$ are \emph{Carnot homogeneous spaces},  the metric measure tangent cones of sub-Riemannian manifolds, see \cite{Bellaiche}. A Carnot homogeneous space is itself a quotient $\M = \G\slash\H$, where $\G$ is a Carnot group and $\H<\G$ is a stratified subgroup acting isometrically on $\G$. We refer to \cref{sec:Carnothomogeneousspaces} for the precise construction and to \cref{prop:quotientofCarnothomogeneous} for the formalization of the quotient \eqref{eq:projintro}. We only stress here that Carnot homogeneous spaces are not homogeneous in the metric sense: although $\G$ acts transitively on $\M$, this action is generally not by isometries.

Fundamental examples of \eqref{eq:projintro} are the quotient of the Heisenberg group to the Grushin plane $\pi:\Heis\to\Gru$, and the quotient of the Engel group to the Martinet structure $\pi:\Eng\to\Mar$.

Our objective is to prove the stability of the $\MCP$ under quotients as \eqref{eq:projintro}. In particular, by establishing the $\MCP$ for a Carnot group $\G$, one would then be able to deduce the $\MCP$ for \emph{all} Carnot homogeneous spaces $\G\slash\H$. Conversely, the failure of the $\MCP$ on $\tilde{\M}$ would imply that it fails on $\M$. This would yield the failure of the $\MCP$ for any sub-Riemannian metric measure space having $\M$ as tangent at some point, motivating our focus on Carnot homogeneous spaces.

However there are non-trivial challenges preventing the application of \cref{thm:GKMS-remark} to \eqref{eq:projintro}.
\begin{enumerate}[(i)]
  \item \label[issue]{i:challenge1} It is not known whether general sub-Riemannian manifolds, or  even simply Carnot groups, are essentially non-branching. Sub-Riemannian branching geodesics do exist  \cite{MR-Branching}, and this may happen if they contain non-trivial abnormal segments. An important condition in sub-Riemannian geometry -- the minimizing Sard property -- stipulates that the set of endpoints of abnormal geodesics starting from a given point has measure zero.  Over the last decades, a great amount of work has been devoted to establishing its validity for several classes of structures (see \cite[Sec.\ 1.3]{LRT-Sard} for  an up-to-date bibliography), and this remains one of the main open problems in sub-Riemannian geometry \cite[Prob.\ 3]{A-openproblems}, \cite[Conj.\ 1]{RT-MorseSard}.
        While it is reasonable to expect that the minimizing Sard property implies some form of non-branching, it is not clear whether it implies the essentially non-branching condition in its standard form.
  \item\label[issue]{i:challenge2} For technical reasons the argument used to prove \cref{thm:GKMS-remark} works for \emph{compact} group actions. In our setting, the acting group is nilpotent and simply connected, and thus it is never compact. As suggested in \cite[Cor.\ 7.25]{GKMS-quotients}, and using the notation of \cref{thm:GKMS-remark}, this technical nuisance can be handled if there exists a co-compact lattice $\Gamma \triangleleft G$ (i.e. a discrete normal subgroup such that the quotient group $G\slash\Gamma$ is compact). In this case, the quotient map $X\to X^* = X\slash G$ admits a factorization
        \begin{equation}\label{eq:factor}
          \begin{tikzcd}
            X \arrow[r, "\q"] &  X\slash\Gamma \arrow[r, "\p"] & (X\slash\Gamma) \slash( G\slash \Gamma) \simeq  X^*.
          \end{tikzcd}
        \end{equation}
        On the one hand, note that $\p$ is now a quotient by a \emph{compact} isometric group action, to which \cref{thm:GKMS-remark} can be applied. On the other hand, $\q$ is a local isometry, so that the $\MCP(K,N)$ for $X$  descends to a weaker $\MCP_{\loc}(K,N)$ for the intermediate factor $X\slash\Gamma$. In \cite{GKMS-quotients}, one uses the local-to-global property of the curvature-dimension, namely the fact that $\CD_{\loc}(K,N) \Rightarrow \CD(K,N)$ under the essential non-branching condition \cite{CM-globalization,Li-globalization}.
        However, the $\MCP$ does not have the local-to-global property \cite[Rmk.\ 5.6]{S-ActaII}. In our specific setting, where both $X,X^*$ are Carnot homogeneous spaces so that they possess metric dilations, $\MCP_{\loc}(K,N)$ and $\MCP(K,N)$ are indeed equivalent, but this is not the case for the intermediate factor $X\slash\Gamma$. Thus, we ultimately need a local version of \cref{thm:GKMS-remark}.
  \item\label[issue]{i:challenge3} A further challenge comes from the fact that the acting group $G$ might not admit a co-compact lattice. The Mal'cev criterion \cite[Ch.\ 5]{CG-NilpotentGroupsBook} establishes that a simply connected nilpotent group has a co-compact lattice if and only if the Lie algebra of $G$ admits a basis with rational structure constants. There are examples where this is not the case even for nilpotent groups of step $2$, see \cite[Ex.\ 5.1.13]{CG-NilpotentGroupsBook}. 
\end{enumerate}

\subsection{Dealing with \texorpdfstring{\Cref{i:challenge1}}{Issue (i)}: \texorpdfstring{$\delta$}{δ}-essentially non-branching}
To address the first of these issues, we introduce a variant of the classical  essentially non-branching condition, that we call \emph{$\delta$-essentially non-branching}, see \cref{sec:deltaess}. It requires non-branching only for $W_2$-geodesics between absolutely continuous measures and finite sums of Dirac masses. This condition is better suited for sub-Riemannian geometry, as shown by the next result (corresponding to \cref{thm:Sardimpliesdeltaessnb}).

\begin{theorem}\label{thm:Sardimpliesdeltaessnb-intro}
  Let $(M,\sfd,\mm)$ be a sub-Rieman\-nian metric measure space satisfying the $*$-minimizing Sard property. Then $(M,\sfd,\mm)$ is $\delta$-essentially non-branching.
\end{theorem}

The $*$-minimizing Sard property is a mild reinforcement of the Sard property. It was introduced in \cite{BMR-Unification}. For real-analytic sub-Riemannian structures, such as Carnot homogeneous spaces, the $*$-minimizing Sard property is equivalent to the minimizing Sard property. See \cref{sec:deltaess}.

The $\delta$-essentially non-branching condition and the \emph{local} $\MCP$ together are preserved by local metric measure isometries. We record the following straightforward result concerning discrete group actions (corresponding to \cref{thm:quotient1}).

\begin{theorem}[Quotients by discrete group actions]\label{thm:quotient1-intro}
  Let $(X,\sfd,\mm)$ be a metric measure space satisfying the $\MCP_{\loc}(K,N)$ for some $K\in \R$ and $N \in [1,\infty)$. Let $G$ be a discrete group of metric measure isometries, acting properly and freely. Then, the metric measure quotient $(X^*,\sfd^*,\mm^*)$ satisfies the $\MCP_{\loc}(K,N)$. If $(X,\sfd,\mm)$ is $\delta$-essentially non-branching, then $(X^*,\sfd^*,\mm^*)$ is $\delta$-essentially non-branching.
\end{theorem}

\subsection{Dealing with \texorpdfstring{\Cref{i:challenge2}}{Issue (ii)}: local measure contractions for quotients}

We now need a version of \cref{thm:GKMS-remark} for the \emph{local} $\MCP$.  To this aim, we combine the arguments from \cite{GKMS-quotients} with ideas from \cite{CM-OptMaps}, while working under the weaker $\delta$-essentially non-branching condition. The key technical step is \cref{thm:MCPineqdeltaess}, which allows us to improve the measure contraction for general second marginals, provided that they can be approximated ``uniformly enough'' by Dirac masses. We do not present \cref{thm:MCPineqdeltaess} in this introduction, but only its main consequence, corresponding to \cref{thm:quotient2}.

\begin{theorem}[Quotients by compact group actions]\label{thm:quotient2-intro}
  Let $(X,\sfd,\mm)$ be a $\delta$-essentially non-branching metric measure space satisfying the $\MCP_{\loc}(K,N)$ for some $K\in \R$ and $N \in [1,\infty)$. Let $G$ be a connected, compact, abelian Lie group of metric measure isometries, acting freely. Then, the metric measure quotient $(X^*,\sfd^*,\mm^*)$ satisfies the $\MCP_{\loc}(K,N)$.
\end{theorem}

The rather restrictive assumption on $G$ -- that it must be an abelian torus -- is due to the fact that we can only assume $\MCP_{\loc}(K,N)$. More precisely, the hypothesis on $G$ is needed to apply \cref{thm:MCPineqdeltaess} via \cref{lem:goodapprox} (see footnote in its proof). If one replaces $\MCP_{\loc}(K,N)$ with  $\MCP(K,N)$, then \cref{thm:quotient2-intro} remains valid for a general compact topological group $G$. We refer to \cref{rmk:MCPglob,rmk:differencesfromCM,rmk:globalMCP,rmk:nonbranchingquot} for more comments and comparison with \cite{GKMS-quotients,CM-OptMaps}.

\subsection{Dealing with \texorpdfstring{\Cref{i:challenge3}}{Issue (iii)}: the factorization argument}

With \cref{thm:quotient1-intro,thm:quotient2-intro}, we can address \Cref{i:challenge1,i:challenge2}. However, we still have to deal with \Cref{i:challenge3}, namely, the fact that stratified groups do not generally admit lattices. Furthermore, we must ensure that the acting group for the $\p$ factor in \Cref{i:challenge2} is a standard torus.

While, all results presented so far hold for general metric measure spaces, what follows is specific for quotients between Carnot homogeneous spaces. In this case, we can exploit their stratified nature to replace the naive two-step factorization outlined in \Cref{i:challenge2} with a more elaborate construction. Namely, we factor \eqref{eq:projintro} via multiple pairs of local isometries and quotients by compact abelian groups such that, at each factor, both \cref{thm:quotient1-intro,,thm:quotient2-intro} can be applied. This is the content of the next result (corresponding to \cref{thm:tower}, to which we refer to for additional comments and remarks).

\begin{theorem}[Factorization of quotients between Carnot homogeneous spaces]\label{thm:tower-intro}
  Let $\G$ be a Carnot group of step $s$.  Let $\M=\G\slash \H$ be a Carnot homogeneous space admitting a quotient to $\tilde{\M}=\G\slash \Htilde$ in the sense of \cref{prop:quotientofCarnothomogeneous}. Then, there exist
  \begin{enumerate}[(a)]
    \item\label{i:towera-intro} closed, dilation-invariant subgroups $\H_i<\G$ for $i=1,\dots,s$, with $\H_0=\H$ and $\H_s = \Htilde$, satisfying
          \begin{equation}
            \H_0\triangleleft\H_1\triangleleft\dots\triangleleft\H_{s-1}\triangleleft \H_{s}   < \G;
          \end{equation}
    \item\label{i:towerb-intro} Carnot homogeneous spaces $\M_i = \G\slash \H_i$ for $i=0,\dots,s$, with $\M_0=\M$ and $\M_s=\tilde{\M}$,  each equipped with their (left-invariant) Carnot-Carathéodory metrics and a choice of right-invariant measure;
    \item\label{i:towerc-intro}  for $i=0,\dots,s-1$, each $\M_{i}$ admits a quotient $\pi_{i+1}:\M_{i}\to\M_{i+1}$ in the sense of \cref{prop:quotientofCarnothomogeneous}. More precisely, the quotient group $R_{i}:=\H_{i+1}\slash \H_{i}$ is an abelian group acting freely and properly on $\M_i$ by metric isometries such that
          \begin{equation}
            \M_{i+1} = \M_{i}\slash  R_{i},
          \end{equation}
          as metric spaces;
    \item\label{i:towerd-intro} discrete normal subgroups $\Gamma_i \triangleleft R_i$, for $i=0,\dots,s-1$, such that
          \begin{itemize}
            \item[--] $\Gamma_i$ acts freely and properly on $\M_i$ by metric measure isometries. We denote by $\M_i\slash\Gamma_i$ the corresponding quotient metric measure space and $\q_i:\M_i\to \M_i\slash\Gamma_i$ the projection map (see \cref{sec:quotients});
            \item[--] $R_i\slash\Gamma_i$ is a compact abelian Lie group, acting freely by metric measure isometries on $\M_i\slash\Gamma_i$. We denote by $(\M_i\slash\Gamma_i)\slash(R_i\slash\Gamma_i)$ the corresponding quotient metric measure space and $\p_i : \M_i\slash\Gamma_i \to (\M_i\slash\Gamma_i)\slash(R_i\slash\Gamma_i) $ the projection map  (see \cref{sec:quotients});
          \end{itemize}
          such that, for all $i=0,\dots,s-1$, the quotient $\pi_{i+1}: \M_{i} \to \M_{i+1}$ can be factorized as
          \begin{equation}\label{eq:factorization-intro}
            \begin{tikzcd}
              \pi_{i+1}: \M_i \arrow[r,"\q_{i +1}"] & \M_i\slash\Gamma_i \arrow[r,"\p_{i +1}"] & (\M_i\slash\Gamma_i)\slash(R_i\slash\Gamma_i) =\M_{i+1}.
            \end{tikzcd}
          \end{equation}
  \end{enumerate}
  In particular, the quotient $\pi:\M\to\tilde{\M}$ can be factorized as a finite number of compositions of local metric measure isometries (the $\q_i$'s) and quotients by compact groups of metric measure isometries (the $\p_i$'s), as described by the following diagram:
  \begin{equation}
    \begin{tikzcd}[sep=scriptsize]
      \M =\G\slash\H \arrow[rrrrrr,"\pi"] 
      & & &  & & & \tilde{\M} =\G\slash\tilde{\H} \\
      \M_0   \arrow[equal,u] \arrow[r,"\q_1"] & \M_0\slash\Gamma_0 \arrow[r,"\p_1"] & \M_{1} \arrow[r] & \dots  \arrow[r]  & \M_{s-1} \arrow[r,"\q_{s}"] &  \M_{s-1}\slash\Gamma_{s-1} \arrow[r,"\p_{s}"] & \M_{s} \arrow[equal,u]
    \end{tikzcd}
  \end{equation}
\end{theorem}

\subsection{Measure contraction properties beyond step 2}

Putting together the results presented so far, we can prove that the $\MCP$ descends to suitable quotients of Carnot homogeneous spaces. The next result corresponds to \cref{thm:MCPforQuotientsofCarnotHomo}. 

\begin{theorem}[MCP for quotients of Carnot homogeneous spaces]\label{thm:MCPforQuotientsofCarnotHomo-intro}
  Suppose that a Carnot homogeneous space $\M=\G\slash\H$ admits a quotient to a Carnot homogeneous space $\tilde{\M}=\G\slash\tilde{\H}$. Assume that the minimizing Sard property holds for all factors $\M,\M_1,\dots,\M_{s-1}$ of the map $\pi:\M\to\tilde{\M}$ in \cref{thm:tower} (this is the case, for example, if the minimizing Sard property holds for all Carnot homogeneous spaces of step $\leq s$, where $s$ is the step of $\M$). Then, if $\M$ satisfies the $\MCP(K,N)$ for some $K \in \R$ and $N \in [1,\infty)$, the space $\tilde{\M}$ also satisfies the $\MCP(K,N)$.
\end{theorem}

\begin{remark}[A new proof for the $\MCP$ for the Grushin plane]
  The simplest Carnot homogeneous space that is not a group is the Grushin plane. In \cite{BR-Interpolation}, via direct and non-trivial estimates, it was proved that it satisfies the $\MCP(K,N)$ for all $K\leq 0$ and $N\geq 5$, matching the same sharp constants of the Heisenberg group found in \cite{J09}. \Cref{thm:MCPforQuotientsofCarnotHomo-intro} provides a new, computation-free proof of this fact, since the Grushin plane is a quotient of the Heisenberg group.
\end{remark}

Combining \cref{thm:MCPforQuotientsofCarnotHomo-intro} and the fact that the $\MCP$ fails for the Martinet structure (i.e. \cref{thm:Martinet-noMCP-intro}), we obtain the following result, corresponding to \cref{thm:noMCP}.

\begin{theorem}[No $\MCP$ with Martinet quotients]\label{thm:noMCP-intro}
  Let $\G$ be a Carnot group of step $s\geq 3$, admitting a quotient to the Martinet structure. Assume that the minimizing Sard property holds for all Carnot groups of step $\leq s$. Then $\G$ does not satisfy the $\MCP(K,N)$ for all $K\in\R$ and $N \in [1,\infty)$.
\end{theorem}
The minimizing Sard property is known to hold in several cases, discussed in \cref{sec:failure}, and we obtain the following unconditional result (corresponding to \cref{cor:noMCPSardok}).

\begin{corollary}[Examples of Carnot groups failing the $\MCP$]\label{cor:noMCPSardok-intro}
  The following Carnot groups do not satisfy the $\MCP(K,N)$ for all $K\in\R$ and $N \in [1,\infty)$:
  \begin{itemize}
    \item Carnot groups of step $3$ that admit a quotient to the Martinet structure (in particular, the Engel group and all the free Carnot groups of step $3$);
    \item Filiform Carnot groups of step $s\geq 3$;
    \item Carnot groups of step $4$ and rank $2$.
  \end{itemize}
\end{corollary}

It is not trivial to determine whether a Carnot group admits a quotient to the Martinet structure. Quite unexpectedly, we prove in the next result that admitting a quotient to Martinet structure is equivalent to admit a quotient to the Engel group, and that both conditions are purely algebraic depending only on the underlying stratified group. This corresponds to \cref{thm:quotienttoMartinet}.

\begin{theorem}[Stratified groups admitting Martinet or Engel quotients]\label{thm:quotienttoMartinet-intro}
  Let $\G$ be a Carnot group with Lie algebra $\g$ and step $s \geq 3$. The following conditions are equivalent:
  \begin{enumerate}[label=(\roman*)]
    \item \label{item:quot-martinet-intro} there exists a closed, dilation-invariant subgroup $\K<\G$ such that $ \G \slash \K$ is a Carnot homogeneous space smoothly isometric to $\Mar$;
    \item \label{item:quot-engel-intro} there exists a closed, and dilation-invariant normal subgroup $\H \triangleleft \G$ such that $\G \slash \H$ is a Carnot group smoothly isometric to $\Eng$;
    \item \label{item:quot-algebraicondition-intro} there exists $\h_3\subset \g_3$ with codimension $1$ such that the subspace
          \begin{equation}\label{eq:subspace-intro}
            \h_2:=\{Y \in \g_2 \mid [\g_1,Y] \in \h_3\}\subset \g_2
          \end{equation}
          has codimension $1$.
  \end{enumerate}
\end{theorem}
A characterization also appeared in \cite[Prop.\ 5.1]{M-regularity}, in terms of a special basis of the Lie algebra. In contrast, condition \eqref{eq:subspace-intro} in \cref{thm:quotienttoMartinet-intro} is intrinsic, and it is key for the next applications. See \cref{rmk:almostfree} for a sufficient condition in terms of the dimension of $\G$.

Different choices of scalar product on the first stratum of a stratified group may correspond to non-isometric Carnot groups (see \cref{sec:Carnothomogeneousspaces}). \Cref{thm:quotienttoMartinet-intro} and \cref{thm:noMCP-intro} yield the failure of the $\MCP$ for \emph{any} Carnot structure on the same stratified group satisfying one of the equivalent conditions of \cref{thm:quotienttoMartinet-intro}. Such structures are all bi-Lipschitz equivalent, but the $\MCP$ is not preserved under bi-Lipschitz equivalence. For example, the $\ell^p$ Heisenberg group for $p>2$ does not satisfy any $\MCP$ \cite{MCPlpsubFH}. We are not able to produce sub-Riemannian Carnot structures on the same stratified group, one satisfying the $\MCP$ and the other failing it.

\subsection{Applications to Carnot groups in low dimensions}

Recall that a stratified group is \emph{indecomposable} if its Lie algebra is not the direct sum of two non-trivial Lie algebras. A classification in dimension $\leq 7$ is found in \cite{LDT-cornucopia,Gong}.

On one hand, a criterion for the validity of the $\MCP$ on a Carnot group can be given in terms of \emph{Goh-Legendre geodesics}, see \cref{sec:validity} for details.


\begin{corollary}[\cite{AAPL-OT,BR-MCP}]\label{cor:goh-ideal-intro}
  Any Carnot group with no non-trivial Goh-Legendre geodesics satisfies the $\MCP(K,N)$ for all $K\leq 0$ and some $N\in [1,\infty)$.
\end{corollary}

On the other hand, by combining \cref{thm:noMCP-intro} and the algebraic condition of \cref{thm:quotienttoMartinet-intro}, we obtain the following result (corresponding to \cref{thm:cornucopia}).
\begin{theorem}[A cornucopia of Carnot groups failing the $\MCP$]\label{thm:cornucopia-intro}
  Among the indecomposable  stratified groups of dimension $\leq 7$, listed in \cref{tab:cornucopia}, equipped with any choice of Carnot-Carathéodory metric and invariant measure turning it into a Carnot group:
  \begin{enumerate}[(i)]
    \item\label{i:cornucopiared-intro} the ones marked in red do not satisfy the $\MCP(K,N)$ for any $K\in \R$, $N\in [1,\infty)$ (an asterisk indicates this is conditional on the validity of the Sard property for that group);
    \item\label{i:cornucopiagreen-intro} the ones marked in green do satisfy the $\MCP(K,N)$ for all $K\leq 0$ and some $N\in [1,\infty)$.
  \end{enumerate}
\end{theorem}
\input{table.tex}

\begin{remark}[Black Carnot groups]\label{rmk:black}
  We do not know whether the $\MCP$ is satisfied by the remaining groups in \cref{tab:cornucopia}, when equipped with a Carnot structure. They do not admit a quotient  to the Martinet structure, so \cref{thm:noMCP-intro} cannot be applied. They do have non-trivial Goh–Legendre geodesics, so we cannot use \cref{cor:goh-ideal-intro} either. Among these groups, it seems that $247H$, $247H_1$, $247K$, $247N$, and $247P_1$ admit a quotient to $N_{6,2,6}$ or $N_{6,3,1a}$. Proving that $\MCP$ fails for these two groups would imply the failure of the $\MCP$ for the others.
\end{remark}
\begin{remark}\label{rmk:nalon}
  Preliminary computations of the second author with L. Nalon, building upon techniques developed in \cite{BV-Dynamical,BNV-Sardfiliform}, show that all the Carnot groups of step $4$ and $5$ appearing in \cref{tab:cornucopia} do satisfy the minimizing Sard property.
\end{remark}
\begin{remark}[Failure of the $\MCP$ in the pre-medium-fat case]
  The notion of \emph{medium-fat} distribution has been introduced in \cite{AS-minvssubanal}, and sub-Riemannian structures with medium-fat distribution have no non-trivial horizontal curves satisfying the Goh condition. In particular, by \cref{cor:goh-ideal-intro}, any Carnot group with medium-fat distribution satisfies the $\MCP$. In \cite{R-subdiff}, the author introduces the concept of \emph{pre-medium-fat distribution}, a relaxation of the medium-fat condition, allowing for non-trivial Goh-Legendre geodesics. One of the main results in \cite{R-subdiff} is that sub-Riemannian structures whose distribution is pre-medium-fat satisfy the minimizing Sard property, see \cite[Cor.\ 1.3]{R-subdiff}. The author asks in \cite[Sec.\ 5.4]{R-subdiff} whether these structures do satisfy the $\MCP$. As an answer to that question, one can verify that the following stratified groups in \cref{tab:cornucopia} have pre-medium-fat distribution and do not satisfy any $\MCP$: $\Eng$, $N_{5,2,3}$, $N_{6, 3, 3}$, $N_{6,3,4}$, $357A$, $257B$, $247D$, $247E$, $247G$, and $247J$.
\end{remark}

To justify the reduction to the indecomposable case, we end with a straightforward result for decomposable groups  following from the tensorization properties of the $\MCP$ (see \cref{thm:decomposable}).

\begin{corollary}[MCP for decomposable Carnot groups]\label{thm:decomposable-intro}
  Let $\G$ be a Carnot group that admits a decomposition as a direct product of Carnot groups
  \begin{equation}
    \G = \G_1\times\dots\times\G_n,
  \end{equation}
  with product metric $\sfd_{\G}^2 = \sum_{j=1}^n \sfd_{\G_j}^2$, and product measure $\mm_{\G} = \otimes_{i=j}^n \mm_{\G_j}$. Then
  \begin{enumerate}[(i)]
    \item \label{i:yes-intro} if each factor $\G_i$ satisfies the $\MCP(K_i,N_i)$ for some $K_i\in \R$ and $N_i\in [1,\infty)$, then $\G$ satisfies the $\MCP(K,N)$ with $K=\min_i K_i$ and $N=\sum_i N_i$;
    \item \label{i:no-intro} if one of the factors $\G_i$ does not satisfy the $\MCP(K,N)$ for some $K\in \R$ and $N\in [1,\infty)$, then $\G$ does not satisfy the $\MCP(K,N)$ (assuming that the minimizing Sard property holds for all Carnot groups of step smaller or equal than the step of $\G$).
  \end{enumerate}
\end{corollary}

\Cref{thm:decomposable} can be used to produce new examples of Carnot groups, either satisfying or failing the $\MCP$, by taking products of structures listed in \cref{tab:cornucopia}.

\subsection{Genericity results in high dimension and step}

It is well-known that, generically, sub-Riemannian structures with constant rank $k\geq 3$ do not have non-trivial abnormal geodesics. We refer to \cite{CJT-generic} for precise definitions.

\begin{theorem}[Chitour-Jean-Trélat \cite{CJT-generic}]\label{gen:genideal} Let $M$ be a smooth manifold. Let $k\geq 3$ be a positive integer, and let $\mathcal{G}_k$ be the set of pairs $(\distr,g)$ on $M$, where $\distr$ is a rank $k$ distribution on $M$, and $g$ is a metric tensor on $\distr$, endowed with the Whitney $C^\infty$ topology. There exists an open dense subset $W_k$ of $\mathcal{G}_k$ such that, for every element of $W_k$, the corresponding sub-Riemannian structure does not admit non-trivial abnormal geodesics.
\end{theorem}

It is perhaps surprising that, generically, even with no non-trivial abnormal geodesics, the $\MCP$ fails if the dimension is large enough. The next result corresponds to \cref{thm:generic-no-MCP}.

\begin{theorem}[Generic failure of the $\MCP$ for rank $3$ and high dimension]\label{thm:generic-no-MCP-intro}
  Assume that the minimizing Sard property holds for  Carnot groups. In the same setting of \cref{{gen:genideal}}, if
  \begin{equation}\label{eq:lowerbound-intro}
    \dim M \geq (k-1)\left(\frac{k^2}{3}+\frac{5k}{6}+1\right),
  \end{equation}
  then there exists an open dense subset $W_k'\subset W_k$ of $\mathcal{G}_k$ such that for every element of $W_k'$, the corresponding sub-Riemannian structure does not admit non-trivial abnormal geodesics and does not satisfy the $\MCP(K,N)$ for any $K\in \R$, $N\in [1,\infty)$ and any choice of smooth measure.
\end{theorem}
\begin{remark}[Relation with generically nonsubanalytic dimensions]
  The lower bound \eqref{eq:lowerbound-intro} is the same one that appeared in \cite{AG-subanal} in relation with the loss of subanaliticity of the squared distance. In fact, the underlying reason is the same: generically if \eqref{eq:lowerbound-intro} is true, the tangent cone at any point must admit a quotient to the Martinet structure. We also note that the lower bound \eqref{eq:lowerbound-intro} is not sharp and can be improved, but this is out of the scope of the current paper.
\end{remark}

We also record the following result, corresponding to \cref{thm:no-MCP-rank2}. (Note that for distributions of rank $k=2$, one does not have an analogue to \cref{gen:genideal}; this is related with the fact that in this case the Goh condition is automatically verified by any abnormal geodesic).

\begin{theorem}[Failure of the $\MCP$ for rank $2$]\label{thm:no-MCP-rank2-intro}
  Let $(M,\sfd,\mm)$ be a sub-Riemannian metric measure space with constant rank $k=2$. Assume that there exists $x\in M$ such that
  \begin{equation}\label{eq:step3}
    \dim \distr^3_x - \dim \distr^2_x \geq 1.
  \end{equation}
  Assume that the minimizing Sard property holds for all Carnot groups of step $\leq s(x)$. Then $(M,\sfd,\mm)$ does not satisfy the $\MCP(K,N)$ for any $K\in \R$ and $N\in [1,\infty)$.
\end{theorem}

Assumption \eqref{eq:step3} means that the iterated distribution at step $3$ at $x$ contains at least one new direction, see \cref{a:SR} for definition.

\subsection{Closing thoughts}
If a sub-Riemannian space does not have Goh-Legendre geodesics (including the trivial one), then, by the results of Agrachev and Lee in \cite{AAPL-OT}, the squared distance function is locally Lipschitz in charts. Consequently, thanks to the work of Badreddine and Rifford in \cite{BR-MCP}, and assuming that the structure is real-analytic, it must, locally, satisfy the $\MCP$. If  Goh-Legendre geodesics are present, the $\MCP$ (and local Lipschitz regularity) may fail, and indeed they do fail for many examples in \cref{tab:cornucopia}.

We suspect that whenever such a geodesic is present, both local Lipschitz regularity and the $\MCP$ are lost, at least if the Legendre condition is satisfied in a strong sense (see \cref{def:strong-Goh-Legendre}).

\begin{conjecture}[Failure of the $\MCP$ with strong Goh-Legendre geodesics]\label{conj}
  Let $(M,\sfd,\mm)$ be a sub-Riemannian metric measure space. If there is a strong-Goh-Legendre geodesic $\gamma$, then $\sfd^2 :M\times M \to \R_+$ fails to be Lipschitz in charts in a neighbourhood of $\gamma$, and the $\MCP(K,N)$ does not hold for all $K\in \R$ and $N\in [1,\infty)$.
\end{conjecture}

Using \cref{thm:decomposable-intro}, we can produce Carnot groups that have non-trivial Goh-Legendre geodesics and yet satisfy the $\MCP$, by taking products of factors with no non-trivial Goh-Legendre geodesics. They are not \emph{strong} Goh-Legendre ones though, supporting the conjecture.

The case in which only non-strong Goh–Legendre geodesics are present is borderline, and a higher-order analysis may be necessary. The Carnot group $N_{6,2,6}$ in \cref{tab:cornucopia} is an example.

The above conjecture is also supported by \cref{tab:cornucopia}: the groups marked with a dagger  \emph{do} have non-trivial Goh geodesics but \emph{do not} have non-trivial Goh-Legendre ones (see \cref{rmk:Legendre}). This emphasizes the importance of the generalized Legendre condition, which turns out to be as significant as the (definitely more well-known and used in the literature) Goh condition.

We remark that a \emph{trivial} geodesic is Goh-Legendre if and only if the step of the structure at that point is at least $3$ -- of course such geodesic cannot be a strong-Goh-Legendre one. The $\MCP$ is not always lost in presence of such geodesics, but local Lipschitzianity of $\sfd^2$ is lost in a neighbourhood of that point \cite{AAPL-OT}.

The absence of (non-trivial) Goh-Legendre geodesics also implies, for real-analytic structures, that the corresponding distance function is subanalytic, as proven by Agrachev and Gauthier in \cite{AG-subanal}. On the contrary, akin to what happens for the $\MCP$ and, of course, local Lipschitzianity, subanalyticity is lost when a suitable Martinet quotient do exist. All these properties -- local Lipschitzianity, $\MCP$ and subanalyticity -- (or their failure) seem to be a manifestation of the compactness of the set of normal geodesics of a fixed length (or lack thereof). However, a direct connection between them is currently missing.

The results of the present work can be adapted to the sub-Finsler setting. Whilst we have not pursued this direction here, combining our findings with those of \cite{MCPlpsubFH, MCPCDsubFH, CEsubFH, CDsubF} could lead to new developments in sub-Finsler geometry.

\subsection{Structure of the paper}

\Cref{sec:prel} contains an overview of optimal transport on metric measure spaces and their quotients by isometric group actions. In \cref{sec:deltaess}, we introduce the weaker essentially non-branching condition that we require, and discuss its relation with the $\MCP$. We then address the stability of the $\MCP$ under discrete and compact group actions in \cref{sec:discreteactions,sec:compactactions}, respectively. In \cref{sec:Carnothomogeneousspaces}, we present Carnot homogeneous spaces and we prove the factorization argument. We put everything together in \cref{sec:MCPforQuotientsofCarnotHomo}, proving our main theorem concerning the $\MCP$ of quotients of Carnot homogeneous spaces and its consequences. In \cref{sec:genericity}, we prove the generic failure of $\MCP$ in high dimension and rank. In \cref{sec:EngMart}, we prove the failure of the $\MCP$ for the Martinet structure. We end with \cref{sec:variants}, where we comment on the failure of variants of the $\MCP$, such as its entropic version and Milman's  $\mathrm{QCD}$.

\subsection{Acknowledgments}

This project has received funding from the European Research Council (ERC) under the European Union's Horizon 2020 research and innovation programme (grant agreement GEOSUB, No. 945655). This research was funded in part by the Austrian Science Fund (FWF) [10.55776/EFP6], and the authors also acknowledge the INdAM support. We wish to thank Fabio Cavalletti, Andrea Mondino, and Andrei Agrachev for several helpful discussions.

%% file: table.tex
\begin{table}[!hb]
  \centering
  \resizebox{0.90\textwidth}{!}{%
    \begin{tabular}{|c|c|c|c|c|c|c|}
      \hline
      dim. & step $1$    & step $2$                          & step $3$             & step $4$         & step $5$ & step $6$ \\ \hline
      1    & $\green \R$ &                                   &                      &                  &          &          \\ \hline
      3    &             & {\green $N_{3,2}=\mathbb{Heis}$}  &                      &                  &          &          \\ \hline
      4    &             &                                   & \red $N_{4,2} =\Eng$ &                  &          &          \\ \hline
      5    &             & {\green $N_{5,3,1}$, $N_{5,3,2}$} & \red $N_{5,2,3}$     & \red $N_{5,2,1}$ &          &          \\ \hline
      6    &             &
      \begin{tabular}{@{}c@{}}
        \green $N_{6,3,5}$ \\ \green $N_{6,3,6}$ \\ \green $N_{6,4,4a}$
      \end{tabular}
           &
      \begin{tabular}{@{}c@{}}
        $N_{6,2,6}$, {\green $N_{6,3,1}^{\dagger}$}, $N_{6,3,1a}$ \\
        {\red $N_{6,3,3}$}, {\red $N_{6,3,4}$}
      \end{tabular}
           &
      \begin{tabular}{@{}c@{}}
        \red $N_{6,2,5}$ \\ \red  $N_{6,2,5a}$ \\ \red $N_{6,2,7}$
      \end{tabular}
           &
      \begin{tabular}{@{}c@{}}
        \red $N_{6,2,1}$ \\ \red $N_{6,2,2}^*$
      \end{tabular}
           &                                                                                                                 \\ \hline
      7    &             &
      \begin{tabular}{@{}c@{}} 
        {\green $37A$}, {\green $37B$} \\ {\green $37B_1$},  {\green $37C$} \\ {\green $37D$}, {\green $37D_1$} \\ {\green $27A$}, {\green $27B$} \\ {\green $17$}
      \end{tabular}
           &
      \begin{tabular}{@{}c@{}} 
        {\red $357A$}, {\red $357B$}, {\red $257B$}               \\
        {\red $247A$}, {\red $247B$}, {\red $247C$}               \\
        {\red $247D$}, {\red $247E$}, {\red $247E_1$}             \\
        {\red $247F$}, {\green $247F_1^{\dagger}$}, {\red $247G$} \\
        $247H$, $247H_1$, {\red $247I$}                           \\
        {\red $247J$}, $247K$, $247N$                             \\
        {\green $247P^{\dagger}$}, $247P_1$, {\green $147D$}      \\
        {\green $147E$}, {\green $147E_{1}$}, $137A$              \\
        {\green $137A_1^{\dagger}$}, $137C$
      \end{tabular}
           &
      \begin{tabular}{@{}c@{}} 
        \red $2457A^*$ \\
        \red $2457B^*$ \\
        \red $2457L$   \\
        \red $2457L_1$ \\
        \red $2457M$
      \end{tabular}
           &
      \begin{tabular}{@{}c@{}} 
        \red $23457A^*$ \\
        \red $23457B^*$ \\
        \red $23457C^*$ \\
        \red $12457H^*$ \\
        \red $12457L^*$ \\
        \red $12457L_1^*$
      \end{tabular}
           &                                                                                                                 
      \red $123457A$                                                                                                         \\ \hline
    \end{tabular}
  }
  \vspace{1em}
  \caption{
    Each entry is an indecomposable stratified group of dimension $\leq 7$, see \cite{LDT-cornucopia,Gong}. Red: admits a quotient to the Martinet structure.  Green: does not have non-trivial Goh-Legendre curve. Black: does not admit a quotient to Martinet and admits non-trivial Goh-Legendre curves (see \cref{rmk:black}).\\
    $^{*}$: conditional on the validity of the minimizing Sard property (see \cref{rmk:nalon}). $^{\dagger}$: admits non-trivial Goh curves but no Goh-Legendre ones (see \cref{rmk:Legendre}).} \label{tab:cornucopia}
\end{table}

%% file: preliminaries.tex
\section{Preliminaries}\label{sec:prel}


We start this preliminary section with a brief overview of optimal transport on metric measure spaces. We refer the reader to \cite{Vil} and references therein for a complete account.

A metric measure space $(X, \sfd, \mm)$ is a complete, separable, proper, and geodesic metric space $(X,\sfd)$ equipped with a $\sigma$-finite Borel measure $\mm$. We assume for simplicity that $\supp \mm = X$. We recall that a metric space is geodesic if for all $x, y \in X$, there exists $\gamma \in \Geo(X)$ such that $\gamma(0) = x$ and $\gamma(1) = y$, where we denoted the set of geodesics of $(X,\sfd)$ by
\[
  \Geo(X) := \left\{ \gamma \in C([0, 1], X) \mid \sfd(\gamma_s, \gamma_t) = |s - t| \sfd(\gamma_0, \gamma_1) \right\}.
\]

If $\mu$ is a Borel measure on $X$ and $T : X \to X$ is a Borel map, then $T_\sharp \mu$ stands for the pushforward measure, i.e. the measure defined through $T_\sharp \mu(A) = \mu(T^{-1}(A))$ for all Borel set $A \subset X$. The set $\P(X)$ will denote the space of probability measures on $X$. The narrow topology on $\P(X)$ is the one induced by the narrow convergence of measures: a sequence $(\mu_n)_{n \in \N} \subset \P(X)$ converges narrowly to $\mu \in \P(X)$, denoted as $\mu_n \rightharpoonup \mu$, if $(\int_X f \mu_n)_{n \in \N}$ converges to $\int_X f \mu$ for all bounded continuous functions $f \in C_b(X)$. This is referred to as ``weak convergence'' in \cite{Vil}.

We will use the notation $\P_c(X)$ for the set of compactly supported probability measures. For $p\in [1,\infty)$, we denote by $\P_p(X)$ the space of probability measures with finite $p$-momentum:
\begin{equation}
  \P_p(X) := \left\{ \mu \in \P(X) \ \Big| \ \int_X \sfd(x_0, x)^p \mu(\di x) < + \infty \text{ for some (and hence any) } x_0 \in X \right\}.
\end{equation}
The $p$-Wasserstein distance is the metric on $\P_p(X)$ introduced via the Monge-Kantorovich minimization problem from optimal transport theory. Given $\mu_0, \mu_1 \in \P(X)$ and setting $\Cpl(\mu_0, \mu_1)$ as the set of couplings $\pi \in \P(X \times X)$ having $\mu_0$ and $\mu_1$ as first and second marginals, we define
\begin{equation}
  \label{eq:pwasserstein}
  W_p(\mu_0, \mu_1)^p := \inf_{\pi \in \Cpl(\mu_0, \mu_1)} \int_{X \times X} \sfd(x, y)^p \pi(\di x \di y).
\end{equation}

It is well-known that convergence with respect to $W_p$ is equivalent to the narrow convergence of measures plus convergence of $p$-momenta. If $(X, \sfd)$ is a geodesic space (resp.\ complete, separable, proper), then so is $(\P_p(X), W_p)$.

When the minimum in \eqref{eq:pwasserstein} is attained at $\pi \in \P(X \times X)$, we will say that $\pi$ is an optimal transport plan between $\mu_0$ and $\mu_1$, and that $\pi \in \Opt_p(\mu_0,\mu_1)$. A Borel map $T : X \to X$ for which $\pi = (\mathrm{id} \times T)_\sharp \mu_0$ is an optimal coupling is called an optimal transport map.

For $\mu_0, \mu_1 \in \P_p(X)$, we denote by $\OptGeo_p(\mu_0, \mu_1)$ the set of all $\nu \in \P(\Geo(X))$ such that $(e_0, e_1)_{\sharp} \nu \in \Opt_p(X)$. It is a standard fact that if $\nu \in \P(\Geo(X))$, then $\mu_t := (e_t)\sharp \nu$ is a $W_p$-geodesic; and vice versa, any $W_p$-geodesic $(\mu_t)_{t \in [0, 1]} \in \Geo(\P_p(X))$ can be lifted to an optimal dynamical plan $\nu \in \OptGeo_p(\mu_0, \mu_1)$.

In the following, we will mostly be concerned with the case $p = 2$, and will refrain from specifying the value of $p$ whenever it is clear from the context.

\subsection{Measure contraction property}

We introduce the distortion coefficients from the Lott--Sturm--Villani theory of synthetic curvature-dimension bounds \cite{LV-Ricci, S-ActaI, S-ActaII}. Given $K \in \R$ and $N \in [1, \infty)$, we set $\tau_{K, N}^{t}(\theta)$ for $(t, \theta) \in [0, 1] \times [0, \infty]$ by
\begin{equation}
  \tau_{K, N}^{t}(\theta) :=
  \begin{cases}
    \infty                                                                                                                                  & K\theta^2 \geq (N-1)\pi^2 \text{ and } K > 0,   \\
    t^{\frac 1 N} \left(\frac{\sin\left(t\theta \sqrt{K/(N-1)}\right)}{\sin\left(\theta \sqrt{K/(N-1)}\right)}\right)^{1 - \frac 1 N}       & 0<K\theta^2<(N-1)\pi^2,                         \\
    t                                                                                                                                       & K\theta^2 =0, \text{ or } K<0 \text{ and } N=1, \\
    t^{\frac 1 N} \left(\frac{\sinh\left(t\theta \sqrt{|K|/(N-1)}\right)}{\sinh\left(\theta \sqrt{|K|/(N-1)}\right)}\right)^{1 - \frac 1 N} & K\theta^2<0.
  \end{cases}
\end{equation}

We will use the following version of the measure contraction property, introduced in \cite{O-MCP}. The notation $\mm|_{A}$ stands for the restriction of a measure $\mm$ to a Borel set $A$.

\begin{definition}[Measure contraction property]\label{def:MCP}
  Let $K\in \R$ and $N\in [1,\infty)$. A metric measure space $(X,\sfd,\mm)$ satisfies the $\MCP(K,N)$ if for any $o\in X$ and any $\mu_0\in \P_c(X)$ of the form $\mu_0 =\tfrac{1}{\mm(A)}\mm|_A$ for some Borel set $A\subset X$ with $0<\mm(A)<\infty$, there exists $\nu \in \OptGeo(\mu_0,\delta_o)$ such that
  \begin{equation}\label{eq:O-MCP}
    \frac{1}{\mm(A)}\mm \geq (e_t)_{\sharp}\left(\tau_{K,N}^{(1-t)}(\sfd(\gamma_0,\gamma_1))^N\nu\right),\qquad \forall\, t\in [0,1].
  \end{equation}
  Similarly, a metric measure space $(X,\sfd,\mm)$ satisfies the $\MCP_{\loc}(K,N)$ if $X$ can be covered by open neighbourhoods $\mathcal{O}\subset X$ such that for any $o\in \mathcal{O}$ and any $\mu_0\in \P_c(X)$ of the form $\mu_0 =\tfrac{1}{\mm(A)}\mm|_A$ for some Borel set $A\subset \mathcal{O}$ with $0<\mm(A)<\infty$, there exists $\nu \in \OptGeo(\mu_0,\delta_o)$ such that \eqref{eq:O-MCP} holds ($\mu_t$ is not required to be supported in $\mathcal{O}$).
\end{definition}
\begin{remark}[Equivalence with Ohta's $\MCP$]
  In \cite{O-MCP}, condition \eqref{eq:O-MCP} is required to hold for $A \subset B(o, \pi / \sqrt{(N - 1) / K})$ when $K > 0$. Such a restriction is not necessary; see \cite[Rmk.\ 6.10]{CM-globalization}. Furthermore, in \cite{O-MCP}, $\mu_0$ is not required to have compact support, but only finite $2$-momentum. Since $\mm$ is $\sigma$-finite and $(X, \sfd)$ is proper, one can show, via standard stability arguments, that \cref{def:MCP} implies the validity of \eqref{eq:O-MCP} also when $\mu_0 \in \P_2(X)$.
\end{remark}

The $\MCP$ implies a corresponding Bishop-Gromov volume growth inequality for $(X,\sfd,\mm)$, and all its consequences. In particular, we record that a metric measure space satisfying the $\MCP(K, N)$ with $K > 0$ (and $N>1$) must be compact, and that setting
\begin{equation}
Z_{t}(x,A) := \left\{ \gamma_t \mid \gamma \in \Geo(X), \gamma_0 = x, \gamma_1 \in A \right\}
\end{equation}
the set of $t$-intermediate points between a Borel set $A\subset X$ and $x\in X$, the $\MCP(0, N)$ implies
\begin{equation}
  \label{eq:MCPtovolumeineq}
  \mm(Z_t(x,A)) \geq t^N \mm(A).
\end{equation}
Finally, we note that if $\MCP(K,N)$ holds, then $N$ must be greater or equal than the Hausdorff dimension of $(X,\sfd)$, see e.g.\ \cite[Cor.\ 2.7]{O-MCP}.

%

\subsection{Quotients by isometric group actions}\label{sec:quotients}

Let $(X,\sfd)$ be a complete, separable, proper and geodesic metric space. Let $G$ be a topological group and $\alpha :G\times X\to X$ be a (continuous) action. We use the shorthand $gx = \alpha(g,x)$ for all $g\in X$ and $x\in X$. We say that $G$ acts by isometries if for all $g\in G$ and $x,y\in X$ it holds
\begin{equation}
  \sfd(g x, g y) = \sfd(x,y).
\end{equation}
We denote with $X^*:=X/G$ the space of orbits of the action of $G$. Its elements are denoted with $x^*$ for $x\in X$. We define a (pseudo-)distance on $X^*$ by
\begin{equation}\label{eq:pseudod}
  \sfd^*(x^*,y^*) := \inf_{g,g'\in G} \sfd(gx,g'y) = \inf_{g\in G} \sfd(gx,y).
\end{equation}
The action of $G$ on $X$ is \emph{proper} if the map $(g,x)\mapsto (gx,x)$ is proper. When the action of $G$ on $X$ is by isometries and proper, its orbits  are closed, \eqref{eq:pseudod} is a distance, and $(X^*,\sfd^*)$ is a complete, separable, geodesic and proper metric space as well.

Let $\mm$ be a $\sigma$-finite Borel measure on $(X,\sfd)$ so that $(X,\sfd,\mm)$ is a metric measure space. We say that $\mm$ is \emph{$G$-invariant} if for all $g\in G$ and Borel set $E$ it holds
\begin{equation}
  \mm(g E) =\mm(E).
\end{equation}
We say that $G$ acts on $(X,\sfd,\mm)$ by \emph{metric measure isometries} if the action is by isometries, proper, and $\mm$ is $G$-invariant. In this case, we induce on $(X^*,\sfd^*)$ a measure in a natural way, depending on whether $G$ is compact or discrete.

\begin{enumerate}[(1)]
  \item\label{enum:compact}  If $G$ is compact, then we set $\mm^*:= \p_\sharp \mm$, where $\p:X\to X^*$ is the projection map (that is proper). It follows that $\mm^*$ is a $\sigma$-finite Borel measure.
        Note that it is not necessary to assume that $\mm$ is $G$-invariant for $\mm^*$ to be well-defined in this case.

  \item\label{enum:discrete} If $G$ is non-compact but discrete, we must assume that the action of $G$ on $X$ is not only proper, but also \emph{free}, that is, for any $x \in X$, it holds that $gx = x$ if and only if $g = e$. Thus, the quotient map $\q:(X,\sfd) \to (X^*, \sfd^*)$ is a covering map and a local isometry \cite[Prop.\ 8.5(3)]{nonpositivebook}. In this case, we define $\mm^*$ as the unique Borel measure on $X^*$ that turns $\q$ into a local metric measure isometry. Specifically, for any $U \subset X$ and $U^* = \q(U) \subset X^*$ such that the restriction $\q: U \to U^*$ is an isometry, we have
        \begin{equation}
          \mm^*|_{U^*} = \q_\sharp\left( \mm|_{U} \right).
        \end{equation}
\end{enumerate}
In both cases, we call $(X^*, \sfd^*, \mm^*)$ the \emph{metric measure quotient} of $(X, \sfd, \mm)$.

We state here a technical lemma that will be used in \cref{sec:tower} for discrete group actions.

\begin{lemma}\label{lem:technical_invariance}
  Let $(X,\sfd,\mm)$ be a metric measure space, and let $G$ be a topological group with a continuous action $\alpha:G\times X\to X$ by metric measure isometries. Assume that $G$ is compact, or discrete, in which case assume that the action $\alpha$ is proper and free. Let $K$ be a topological group with a continuous action $\beta : K \times X \to X$ (not necessarily by isometries) such that:
  \begin{enumerate}[(i)]
    \item\label{i:technical_invariance1} the action $\beta$ of $K$ commutes with the action $\alpha$ of $G$, namely
          \begin{equation}
            \beta(g,\alpha(k,x)) = \alpha(k,\beta(g,x)),\qquad \forall\, g\in G,\,k\in K\, x\in X;
          \end{equation}
    \item\label{i:technical_invariance2} the measure $\mm$ is invariant by the action $\beta$, namely
          \begin{equation}
            \beta(k,\cdot)_\sharp \mm = \mm,\qquad \forall\, k\in K.
          \end{equation}
  \end{enumerate}
  Let $(X^*,\sfd^*,\mm^*)$ be the metric measure quotient by the action $\alpha$ of $G$. Let $\beta^*: K\times X^*\to X^*$ be the induced continuous action of $K$ on $X^*$ given by
  \begin{equation}\label{eq:equivariant}
    \beta^*(k,x^*) = \beta(k,x)^*,\qquad \forall\, k\in K,\, x\in X,
  \end{equation}
  the star denoting the equivalence class in $X^*$, which is well-defined by assumption \ref{i:technical_invariance1}. Then the measure $\mm^*$ is invariant by the action $\beta^*$.
\end{lemma}
\begin{proof}
  If $G$ is compact (case \labelcref{enum:compact} in \cref{sec:quotients}) and then $\mm^*=\p_\sharp \mm$, where $\p:X\to X^*$ is the projection map, the proof is straightforward. Thus, we assume that $G$ is discrete, non-compact, and its action on $X$ is proper and free (case \labelcref{enum:discrete} in \cref{sec:quotients}). We use the notation $\Gamma \equiv G$, and we abbreviate the actions of $K$ on $X$ and $X^*$ by  $xk=\beta(k,x)$ and $x^*k = \beta^*(k,x^*)$ (this is just a convenient notation, we are not assuming that $\beta$ and $\beta^*$ are right actions).

  It is enough to prove that, for fixed $k\in K$, there is an open cover $\{W_c^*\}_{c \in C}$ of $X^*$ such that
  \begin{equation}
    \mm^*|_{W_c^*} = \left(\beta^*(k,\cdot)_\sharp \mm^*\right)|_{W_c^*}, \qquad \forall\, c\in C.
  \end{equation}
  Denote by $\q:X\to X^*$ the projection, and recall that it is a covering map and a local isometry. Let $\{V_a^*\}_{a \in A}$ be an open cover of $X^*$ by isometrically evenly covered neighbourhoods. More precisely, for all $a\in A$, $\q^{-1}(V_a^*)$ is the disjoint union of open sets $\{V_{a,j}\}_{j\in \Gamma}$, such that $\q: V_{a,j} \to V^*_a$ is an isometry for all $j\in \Gamma$. Since the action of $\Gamma$ is free, the degree of the covering is equal to the cardinality of $\Gamma$ so we use the latter to label the sheets of the covering.

  Consider the open cover $\{V_a^* k\}_{a\in A}$, obtained by acting with a fixed $k\in K$. Since the actions of $k\in K$ on $X$ and $X^*$ are not isometries in general, this open cover is made of evenly covered neighbourhoods, but not \emph{isometrically} so. Since $\q:X\to X^*$ is a covering map and a local isometry, we can write each $V_a^*k$ as the union of isometrically evenly covered neighbourhoods: for all $a\in A$ we find an index set $B_a$ and open sets $U_{a,b}^*\subset X^*$ for $b\in B_a$ such that
  \begin{equation}
    V_a^*k = \bigcup_{b \in B_a} U_{a,b}^*, \qquad \forall \,a\in A.
  \end{equation}
  Thus, up to suitable relabelling, we have found a countable open cover $\{W_c^*\}_{c\in C}$ of $X^*$ by isometrically evenly covered neighbourhoods with the property that $\{W_c^*k^{-1}\}_{c\in C}$ is also an open cover of $X^*$ made by isometrically evenly covered neighbourhoods. More precisely if $\q^{-1}(W_c^*) = \cup_{j\in \Gamma} W_{c, j}$ with $\q: W_{c, j} \to W_{c}^*$ being an isometry for all $c\in C$, $j\in \Gamma$, then $\q^{-1}(W_c^*k^{-1}) = \cup_{j\in \Gamma} W_{c,j} k^{-1}$ with $\q:W_{c,j}k^{-1}\to W_c^*k^{-1}$ being an isometry for all $c\in C$, $j\in \Gamma$.

  In particular, by definition of the quotient measure in case \ref{enum:discrete}, we have for the fixed $k\in K$:
  \begin{equation}\label{eq:definitioninproof}
    \qquad \mm^*|_{W_c^*} = \q_\sharp\left(\mm|_{W_{c,j}}\right)\quad \text{and} \quad \mm^*|_{W_c^*k^{-1}} = \q_\sharp\left(\mm|_{W_{c,j}k^{-1}}\right),\qquad\forall\,c\in C,\, j\in \Gamma.
  \end{equation}
  Then, for $c\in C$ and for any Borel set $E\subset X^*$, we have
  \begin{align}
    \left(\beta^*(k,\cdot)_\sharp \mm^*\right)|_{W_c^*}(E)
     & =  \mm^*\left( (E\cap W_{c}^*)k^{-1}\right)                                                                  \\
     & = \mm\left(\q^{-1}(E k^{-1}) \cap W_{c ,j}k^{-1}\right) & \text{by \eqref{eq:definitioninproof}}             \\
     & = \mm\left((\q^{-1}(E) \cap W_{c, j})k^{-1}\right)                                                           \\
     & = \mm\left(\q^{-1}(E) \cap W_{c, j}\right)              & \text{by assumption \ref{i:technical_invariance2}} \\
     & = \mm^*|_{W_{c}^*}(E) ,                                                                                      
  \end{align}
  concluding the proof.
\end{proof}

%% file: deltaenb.tex
\section{The \texorpdfstring{$\delta$}{delta}-essentially non-branching condition}\label{sec:deltaess}

We recall the essentially non-branching condition, introduced by Rajala and Sturm in \cite{RS-nonbranchingstrongCD}.

\begin{definition}[Essentially non-branching]\label{def:enb}
    A set $\Gamma \subset \Geo(X)$ is \emph{non-branching} if for any $\gamma^1, \gamma^2 \in \Gamma$, we have that $\gamma^1 = \gamma^2$ whenever there is $t \in (0, 1)$ such that $\gamma^1_s = \gamma^2_s$ for all $s \in [0, t]$. A metric measure space $(X, \sfd, \mm)$ is \emph{essentially non-branching} if for every $\mu_0, \mu_1 \in \P_2(X)$ with $\mu_0,\mu_1\ll\mm$, any $\nu \in \OptGeo(\mu_0, \mu_1)$ is concentrated on a set of non-branching geodesics, namely there exists a set of non-branching geodesics $\Gamma$ such that $\nu(\Gamma) = 1$.
\end{definition}

We introduce a discrete variant of the essentially non-branching condition. The main relevance of this new condition is that, in contrast with the standard one, it is implied by the minimizing sub-Riemannian Sard property (see \cref{def:Sard}).

\begin{definition}[$\delta$-essentially non-branching]\label{def:delta-enb}
    A metric measure space $(X,\sfd,\mm)$ is \emph{$\delta$-essentially non-branching} if for any $\mu_0\in \P_c(X)$ with $\mu_0\ll \mm$ and any finite sum of Dirac masses $\mu_1\in \P(X)$ (namely, $\mu_1  = \sum_{j=1}^L\lambda_j \delta_{x_j}$ for some $L\in \N$ and distinct points $x_1,\dots,x_L$), any $\nu \in \mathrm{OptGeo}(\mu_0,\mu_1)$ is concentrated on a set of non-branching geodesics.
\end{definition}

Assuming the measure contraction property, essentially non-branching implies $\delta$-essentially non-branching. In fact in this case, by \cite{CM-OptMaps}, for $\mu_0,\mu_1$ as in \cref{def:delta-enb} there is a unique $\nu \in \mathrm{OptGeo}(\mu_0,\mu_1)$, and it is induced by a map $S:X\to\mathrm{Geo}(X)$, and hence $\nu$ is concentrated on the set $S(\supp(\mu_0))$ which is made of non-branching geodesics. In general, the relationship between the two conditions is not clear.

\subsection{Sard properties and \texorpdfstring{$\delta$}{delta}-essentially non-branching}

We assume some familiarity with sub-Riemannian geometry. A summary tailored to our purposes is provided in \cref{a:SR}.

\begin{definition}[Sub-Riemannian m.m.s.] A \emph{sub-Riemannian metric measure space} $(M,\sfd,\mm)$ is a smooth manifold $M$, equipped with a complete Carnot-Carathéodory distance $\sfd$, and a smooth measure $\mm$, i.e.\ with smooth positive density in local charts.
\end{definition}

\begin{definition}[Sard properties]\label{def:Sard} A sub-Riemannian metric measure space $(M,\sfd,\mm)$ satisfies
    \begin{itemize}
        \item the \emph{minimizing Sard property} if for any $x\in M$ the set of final points of abnormal geodesics from $x$ has zero measure in $M$;
        \item the \emph{$*$-minimizing Sard property} if for any $x\in M$ the set of final points of geodesics from $x$ containing non-trivial abnormal segments has zero measure in $M$.
    \end{itemize}
\end{definition}
The $*$-minimizing Sard property was introduced in \cite[Def.\ 7.13]{BMR-Unification}, and it is a reinforcement of the minimizing Sard property. The ``star'' comes from the fact that, if the $*$-minimizing Sard property holds true, then for any $x\in M$ one can find a geodesically star-shaped and full-measure subset where $\sfd^2(x,\cdot)$ is smooth, see \cite[Lemma 7.14]{BMR-Unification}. If $M$ is a real-analytic manifold and the sub-Riemannian Hamiltonian is real-analytic, it is well-known that geodesics are either abnormal, or do not contain non-trivial abnormal segments; it follows that, in the real-analytic case, the $*$-minimizing and minimizing Sard properties are equivalent.

\begin{theorem}\label{thm:Sardimpliesdeltaessnb}
    Let $(M,\sfd,\mm)$ be a sub-Rieman\-nian metric measure space satisfying the $*$-minimizing Sard property. Then $(M,\sfd,\mm)$ is $\delta$-essentially non-branching.
\end{theorem}
\begin{proof}
    For any $y\in M$ let $B_y\subset M$ be the set of those initial points $\gamma_0$ of geodesics $\gamma \in \mathrm{Geo}(M)$ such that $\gamma_1 = y$ and $\gamma$ is branching, that is there exists $\eta \in \mathrm{Geo}(M)$, with $\eta \neq \gamma$, and $t'\in (0,1)$ such that $\gamma|_{[0,t']} = \eta|_{[0,t']}$. The segment $\gamma|_{[0,t']}$ must be abnormal, see \cite[Cor.\ 6]{MR-Branching}. In particular, $B_y$ is contained in the set $A_y$ of final points of geodesics starting from $y$ that contain a non-trivial abnormal segment. In turn, if the $*$-minimizing Sard property holds, $A_y$ is contained in a Borel set $C_y$ with $\mm(C_y)=0$, see \cite[Lemma 7.14, setting $C_y := M \setminus \mathcal{U}_y$]{BMR-Unification}. 
   
    We can now prove that $(M,\sfd,\mm)$ is $\delta$-essentially non-branching. Let  $\mu_0\in \P_c(X)$ with $\mu_0\ll\mm$, and $\mu_1 = \sum_{j=1}^L\lambda_j \delta_{x_j} \in \P(X)$. Let $\nu \in \mathrm{OptGeo}(\mu_0,\mu_1)$ and $C:=\cup_{j=1}^L C_{x_j}$. By construction, $C$ is Borel, $\mm(C)=0$, and thus $\mu_0(C)=0$. The set of geodesics in $\supp( \nu)$ that can branch is contained in $e_0^{-1}(C)$, and we have $\nu(e_0^{-1}(C)) = (e_0)_{\sharp} \nu(C) = \mu_0(C) = 0$. It follows that $\nu$ is concentrated on the Borel set $\supp( \nu) \setminus e_0^{-1}(C)$, which is non-branching.
\end{proof}

\subsection{Measure contraction properties under \texorpdfstring{$\delta$}{delta}-essentially non-branching}

The next result extends the $\MCP$ to the case where the second marginal $\mu_1$ is a general probability measure that can be approximated by Dirac masses in a suitable quantitative sense. This will be a key tool in showing that the \emph{local} $\MCP$ descends to quotients by compact group actions.

\begin{theorem}[Improved measure contraction for more general second marginals]\label{thm:MCPineqdeltaess}
    Let $(X,\sfd,\mm)$ be a $\delta$-essentially non-branching metric measure space satisfying the $\MCP_{\loc}(K,N)$ for some $K\in \R$ and $N \in [1,\infty)$. Then, for any bounded set $U \subset X$, there exists $\varepsilon > 0$ such that the following property holds. For all $\mu_0, \mu_1 \in \P_c(X)$ satisfying:
    \begin{itemize}
        \item $\mu_0 =\tfrac{1}{\mm(A)}\mm|_A$ for some Borel set $A\subset U$ with $0<\mm(A)<\infty$;
        \item $\supp\mu_0\cup\supp\mu_1 \subset U$;
        \item there exists a sequence of finite sums of Dirac masses $\mu_1^{m} \in \P_c(X)$ and optimal plans $\pi^{m} \in \Opt(\mu_0,\mu_1^{m})$, for $m\in \N$, such that $\mu_1^{m}\rightharpoonup \mu_1$, and
        \begin{equation}\label{eq:essup}
            \pi^{m}-\essup\sfd < \varepsilon, \qquad \forall\, m \in \N;
        \end{equation}
    \end{itemize}
    there exists an optimal dynamical plan $\nu\in\OptGeo(\mu_0,\mu_1)$ such that it holds
    \begin{equation}\label{eq:MCPineqdeltaess}
        \frac{1}{\mm(A)}\mm \geq (e_t)_{\sharp}\left(\tau_{K,N}^{(1-t)}(\sfd(\gamma_0,\gamma_1))^N\nu \right),\qquad \forall\, t\in [0,1].
    \end{equation}
\end{theorem}

\begin{remark}\label{rmk:MCPglob}
    If the diameter of $\supp \mu_0 \cup \supp \mu_1$ is less than $\varepsilon$, then \eqref{eq:essup} holds. However, we will need to apply this result to more general situations, where $U$ consists of orbits of the action of a compact group. Furthermore, if $\MCP_{\loc}$ is replaced by $\MCP$, then condition \eqref{eq:essup} and the restriction of the supports to $U$ can be omitted.
\end{remark}
\begin{remark}\label{rmk:differencesfromCM}
    The proof of \cref{thm:MCPineqdeltaess} follows the same strategy of  \cite[Props.\ 4.2, 4.3]{CM-OptMaps}, but with some technical differences due to our weaker hypotheses ($\delta$-essentially non-branching and $\MCP_{\loc}$). Note also that we employ Ohta's version of $\MCP$ instead of the variant in \cite{CM-OptMaps}. We avoid using the \emph{good geodesics} of \cite[Thm.\ 3.1]{CM-OptMaps} introduced in \cite{Rajala1,Rajala2,RS-nonbranchingstrongCD}. The overall argument is thus simpler, and we provide a self-contained proof here to highlight the differences. Note, though, that we are unable to prove that the optimal dynamical plan $\nu$ and the corresponding optimal plan $(e_0, e_1)_\sharp \nu$ are unique and induced by a map, as was shown in \cite{CM-OptMaps} using the classical essentially non-branching condition.
\end{remark}
\begin{proof}
    \textbf{Step 0.} By the Lebesgue number lemma and the $\MCP_{\loc}(K, N)$, we can choose $\varepsilon > 0$ such that for any neighborhood $\mathcal{O} \subseteq U$ with $\diam(\mathcal{O}) < \varepsilon$, and for any $o \in \mathcal{O}$, as well as for any $\mu_0 \in \P_c(X)$ of the form $\mu_0 = \frac{1}{\mm(A)} \mm|_A$ for some Borel set $A \subset \mathcal{O}$ with $0 < \mm(A) < \infty$, there exists $\nu \in \OptGeo(\mu_0, \delta_o)$ such that \eqref{eq:O-MCP} holds.

    \textbf{Step 1.} Let $m\in\N$, and consider $\mu_1^{m}$ and $\pi^{m}\in \Opt(\mu_0,\mu_1^{m})$ from the statement. Note that $\supp(\pi^{m})$ is a $\sfd^2$-cyclically monotone set\footnote{In \cite[Prop.\ 4.2]{CM-OptMaps}, the authors consider instead a $\sfd^2$-cyclically monotone set containing the support of \emph{any} optimal plan with marginals $\mu_0$ and $\mu_1^{m}$. We cannot work with this set, as it may not satisfy \eqref{eq:essup2}. As a consequence, we can only prove that the specific $\pi^m$ is induced by a map, but not uniqueness of the plan.} \cite[Thm.\ 1.13]{AG-usersguide} and, by \eqref{eq:essup}, it holds
    \begin{equation}\label{eq:essup2}
        \sfd(x,y) < \varepsilon,\qquad \forall\, (x,y)\in \supp(\pi^{m}).
    \end{equation}

    Let $P_j:X\times X\to X$ be the projection on the $j$-th factor, $j=1,2$. Let $S\subset X$ be the set of $x\in X$ such that
    \begin{equation}
        P_2\Big(\supp(\pi^{m}) \cap (\{x\}\times X)\Big)
    \end{equation}
    is not a singleton (and note that this set is analytic). We prove that $\mu_0(S)=0$, from which it will follow that $\pi^{m}$ is induced by a map.

    Suppose by contradiction that $\mu_0(S)>0$. Since $\mu_1^{m}$ is a finite sum of Dirac masses, up to restricting $S$ and relabeling the points in $\supp(\mu_1^{m})$, we can assume that there exist Borel maps
    \begin{equation}
        T_1,T_2:S\to X,\qquad \mathrm{graph}(T_1),\mathrm{graph}(T_2)\subset \supp(\pi^{m}),
    \end{equation}
with $T_1(x)=x_1$, $T_2(x)=x_2$ for all $x\in S$, with $x_1\neq x_2$, $S$ is compact and, furthermore, as a consequence of \eqref{eq:essup2}, we have
    \begin{equation}\label{eq:essup3}
        \sfd(x,x_i)<\varepsilon,\qquad \forall x\in S,\quad i=1,2.
    \end{equation}
    We will now consider the optimal transport problem between the following two measures
    \begin{equation}\label{eq:tildemeasures}
        \bar{\mu}_0:=\frac{1}{\mm(S)}\mm|_{S},\qquad \text{and} \qquad \bar{\mu}_1:=\frac{1}{2}\left(\delta_{x_1}+\delta_{x_2}\right).
    \end{equation}
    We remark that $\tfrac{1}{2}((Id,T_1)_\sharp \bar{\mu}_0 + (Id,T_2)_{\sharp}\bar{\mu}_0)\in\Cpl(\bar{\mu}_0,\bar{\mu}_1)$, and its support is $\sfd^2$-cyclically monotone (since it is a subset of $\supp(\pi^{m})$), therefore it is optimal \cite[Thm.\ 1.13]{AG-usersguide}.


    \textbf{Step 2.} Thanks to the above construction, and the $\MCP_{\loc}$, there exists $\bar{\nu}^i\in \OptGeo(\bar{\mu}_0,\delta_{x_i})$ for $i=1,2$ such that it holds
    \begin{equation}
        \frac{1}{\mm(S)}\mm \geq (e_t)_{\sharp}\left(\tau_{K,N}^{(1-t)}(\sfd(\gamma_0,\gamma_1))^N \bar{\nu}^i\right),\qquad \forall\, t\in[0,1].
    \end{equation}
    Recall that $W_2$-geodesics of optimal dynamical satisfying the Ohta's $\MCP$ inequality are absolutely continuous w.r.t.\ $\mm$, so we let $\bar{\mu}_t^i= \bar{\rho}_t^i\mm$, $i=1,2$. Taking into account the inequality
    \begin{equation}
        \tau_{K,N}^{(1-t)}(\theta)\geq (1-t)e^{-\tfrac{\theta t\sqrt{(N-1)K^{-}}}{N}}, \qquad \forall\, \theta \geq 0,
    \end{equation}
    proved e.g.\ in \cite[Rmk.\ 2.3]{CM-OptMaps} (here $K^{-} = \max\{-K,0\}$), and \eqref{eq:essup3}, we obtain for $i=1,2$ that
    \begin{equation}
        \mm\left(\{\bar{\rho}^i_t>0\}\right) \geq (1-t)^N e^{-\varepsilon t\sqrt{(N-1)K^{-}}}\mm(S),
    \end{equation}
    yielding
    \begin{equation}\label{eq:liminf1}
        \liminf_{t\to 0} \mm\left(\{\bar{\rho}_t^i>0\}\right) \geq \mm(S) = \mm\left(\{\bar{\rho}_0^i>0\}\right).
    \end{equation}

    %
    %
    Denote by $S^{r}$ the closed $r$-tubular neighborhood of $S$. By the dominated convergence theorem, and since $S$ is closed, $\lim_{r\to 0} \mm(S^{r}) =  \mm(S)$. In particular, there exists $r_0>0$ such that
    \begin{equation}\label{eq:liminf2}
        \mm(S^{r_0}) \leq \frac{3}{2} \mm(S)
    \end{equation}
    By construction, for $\bar{\mu}_t^i$-a.e.\ $x\in X$ there is $\gamma \in \Geo(X)$ such that $\gamma_0\in S$, $\gamma_1 = x_i$, and $\gamma_t = x$. For $t\in [0,r_0]$, the measures $\bar{\mu}_t^i$ are thus concentrated on $S^{r_0}$ and it follows that
    \begin{align}
        \mm(S^{r_0})
        & \geq \mm\left(\{\bar{\rho}_t^1>0\}\cup \{\bar{\rho}_t^2>0\}\right) \\
        & \geq \mm\left(\{\bar{\rho}_t^1>0\}\right) + \mm\left(\{\bar{\rho}_t^2>0\}\right) -\mm\left(\{\bar{\rho}_t^1>0\}\cap \{\bar{\rho}_t^2>0\}\right).
    \end{align}
    Using \eqref{eq:liminf1} and \eqref{eq:liminf2} we obtain
    \begin{equation}\label{eq:tausmall}
        \limsup_{t\to 0} \mm\left(\{\bar{\rho}_t^1>0\}\cap \{\bar{\rho}_t^2>0\}\right)\geq \frac{1}{3}\mm(S) >0.
    \end{equation}

    \textbf{Step 3.} Since $(Id,T_i)_\sharp \bar{\mu}_0$ is the unique optimal plan between $\bar{\mu}_0$ and $\delta_{x_i}$ we have $(e_0,e_1)_\sharp\bar{\nu}_i = (Id,T_i)_\sharp\bar{\mu}_0$, for $i=1,2$. Using that $\tfrac{1}{2}((Id,T_1)_\sharp \bar{\mu}_0 + (Id,T_2)_{\sharp}\bar{\mu}_0)\in\Opt(\bar{\mu}_0,\bar{\mu}_1)$, it follows that
    \begin{equation}\label{eq:barnu}
        \bar{\nu}:=\frac{1}{2}(\bar{\nu}_1+\bar{\nu}_2)\in \OptGeo(\bar{\mu}_0,\bar{\mu}_1).
    \end{equation}

As in \cite{CM-OptMaps}, we apply the mixing procedure of \cite{RS-nonbranchingstrongCD}. By \eqref{eq:tausmall}, there is $\tau\in (0,1)$ such that
    \begin{equation}\label{eq:mm_intersection}
        \mm\left(\{\bar{\rho}_\tau^1>0\}\cap \{\bar{\rho}_\tau^2>0\}\right)>0.
    \end{equation}
    Define the following optimal dynamical plans
    \begin{equation}
        \nu^{\mathrm{left}}:=\frac{1}{2}(\mathrm{restr}_0^\tau)_\sharp\left(\bar{\nu}_1 + \bar{\nu}_2\right),\qquad \nu^{\mathrm{left}}:=\frac{1}{2}(\mathrm{restr}_\tau^1)_\sharp\left(\bar{\nu}_1 + \bar{\nu}_2\right),
    \end{equation}
    where $\mathrm{restr}_{t_1}^{t_2}: C([0,1],X)\to C([0,1],X)$ is the restriction-and-reparametrization map (of course it takes geodesics to geodesics). Note that
    \begin{equation}
        \alpha:=(e_1)_\sharp \nu^{\mathrm{left}} = (e_0)_\sharp \nu^{\mathrm{right}} = \frac{1}{2}\left(\bar{\mu}_\tau^1+\bar{\mu}_\tau^2\right).
    \end{equation}
    Note that \eqref{eq:mm_intersection} holds also for the measure $\alpha$, namely
    \begin{equation}\label{eq:alpha_intersection}
        \alpha\left(\{\bar{\rho}_\tau^1>0\}\cap \{\bar{\rho}_\tau^2>0\}\right)  = \frac{1}{2}\int_{\{\bar{\rho}_\tau^1>0\}\cap \{\bar{\rho}_\tau^2>0\}} \Big(\bar{\rho}_{\tau}^1(x)+\bar{\rho}_{\tau}^2(x)\Big)\mm(\di x) >0.
    \end{equation}
    Consider the disintegration of $\nu^{\mathrm{left}}$ (resp.\ $\nu^{\mathrm{right}}$) with respect to $e_1$ (resp.\ $e_0$):
    \begin{equation}
        \nu^{\mathrm{left}} = \int_X \nu_x^{\mathrm{left}} \alpha(\di x),\qquad \nu^{\mathrm{right}} = \int_X \nu_x^{\mathrm{right}} \alpha(\di x),
    \end{equation}
    for uniquely defined measures $\nu_x^{\mathrm{left}}\in\P(\Geo(X)\cap e_1^{-1}(x))$ and $\nu_x^{\mathrm{right}}\in\P(\Geo(X)\cap e_0^{-1}(x))$, for $\alpha$-a.e.\ $x\in X$, respectively. We now glue $\nu^{\mathrm{left}}$  to $\nu^{\mathrm{right}}$, and we show that the glued measure is an optimal dynamical plan between $\bar{\mu}_0$ and $\bar{\mu}_1=\tfrac{1}{2}(\delta_{x_1}+\delta_{x_2})$, and that it contradicts the $\delta$-essentially non-branching assumption. In order to do so, consider the (Lipschitz) map
    \begin{equation}
        \mathrm{Gl}:\left\{(\gamma^1,\gamma^2)\in C([0,1],X)\times C([0,1],X) \mid \gamma^1_1 = \gamma_0^2\right\} \to C([0,1],X),
    \end{equation}
    defined by $\mathrm{Gl}(\gamma^1,\gamma^2) := \gamma^1_{2s}$ if $0\leq s\leq 1/2$ and $\mathrm{Gl}(\gamma^1,\gamma^2):=\gamma^2_{2s-1}$ when $1/2\leq s\leq 1$. We set
    \begin{equation}
        \nu^{\mathrm{mix}}:=\int_X \nu_x\alpha(\di x),\qquad \nu_x:=\mathrm{Gl}_{\sharp}\left(\nu_x^{\mathrm{left}}\times \nu_x^{\mathrm{right}}\right),
    \end{equation}
    where $\nu_x^{\mathrm{left}}\times \nu_x^{\mathrm{right}}$ is the product probability measure on $C([0,1],X)\times C([0,1],X)$. Remember that $\bar{\nu}$ defined in \eqref{eq:barnu} is, by construction, concentrated on the set of geodesics
    \begin{equation}
        \bar{\Gamma}:=(e_0,e_1)^{-1}\Big(S\times \{x_1,x_2\}\Big) \subset \Geo(X)
    \end{equation}
    It follows that for $\alpha$-a.e.\ $x\in X$, the measure $\nu_x$ is concentrated on the set of curves
    \begin{multline}
        \left\{\gamma \in C([0,1],X)\mid \exists \gamma^1,\gamma^2\in\bar{\Gamma} \text{ s.t. } \gamma^1_\tau=\gamma^2_\tau = x, \right. \\
        \left.\mathrm{restr}_0^\tau(\gamma)=\mathrm{restr}_0^\tau(\gamma^1), \quad \mathrm{restr}_\tau^1(\gamma)=\mathrm{restr}_\tau^1(\gamma^2)\right\}.
    \end{multline}
    Furthermore, the set $(e_0,e_1)(\bar{\Gamma})= S\times \{x_1,x_2\}$ and is, by construction, $\sfd^2$-cyclically monotone. Therefore, for $\gamma^1,\gamma^2\in\bar\Gamma$ with $\gamma^1_\tau = \gamma^2_\tau = x$, we obtain
    \begin{align}
        \sfd^2(\gamma_0^1,\gamma_1^1) + \sfd^2(\gamma_0^2,\gamma_1^2)
        & \leq \sfd^2(\gamma_0^1,\gamma_1^2) + \sfd^2(\gamma_0^2,\gamma_1^1)                                                          \\
        & \leq \left( \tau \ell(\gamma^1)+(1-\tau)\ell(\gamma^2)\right)^2+ \left( \tau \ell(\gamma^2)+(1-\tau)\ell(\gamma^1)\right)^2 \\
        & \leq \ell(\gamma^1)^2+\ell(\gamma^2)^2 - 2\tau(1-\tau)\left(\ell(\gamma^1)-\ell(\gamma^2)\right)^2                          \\
        & \leq \ell(\gamma^1)^2+\ell(\gamma^2)^2 = \sfd^2(\gamma_0^1,\gamma_1^1)+ \sfd^2(\gamma_0^2,\gamma_1^2).
    \end{align}
    It follows that all these inequalities are actually equalities. The lengths $\ell(\gamma^1)=\ell(\gamma^2)$ are thus equal, $\gamma=\mathrm{Gl}(\gamma^1,\gamma^2)$ must be a geodesic, and its length is equal to $l_x:=\tfrac{\sfd(x,x_1)}{(1-\tau)}=\tfrac{\sfd(x,x_2)}{(1-\tau)}$. Disintegrating $\nu^{\mathrm{mix}}$ along $e_\tau$ and since $(e_\tau)_\sharp\nu^{\mathrm{mix}}=\alpha = \frac{1}{2}(e_\tau)_{\sharp}(\bar{\nu}^1+\bar{\nu}^2)$, it holds
  \begin{align}
        \int_{\mathrm{Geo}(X)} \sfd^2(\gamma_0,\gamma_1)\nu^{\mathrm{mix}}(\di\gamma) & =\int_X l_x^2(e_\tau)_\sharp (\nu^{\mathrm{mix}})(\di x)                                           \\
        & = \frac{1}{2}\int_X l_x^2 (e_\tau)_\sharp(\bar{\nu}^1+\bar{\nu}^2)(\di x)                           = \frac{1}{2}\int_{\mathrm{Geo}(X)}\sfd^2(\gamma_0,\gamma_1)(\bar{\nu}^1+\bar{\nu}^2)(\di \gamma).
    \end{align}
    Since $\bar{\nu}=\tfrac{1}{2}(\bar{\nu}^1+\bar{\nu}^2)\in \OptGeo(\bar{\mu}_0,\bar{\mu}_1)$, we also have $\nu^{\mathrm{mix}}\in\OptGeo(\bar{\mu}_0,\bar{\mu}_1)$.

    Assume by contradiction that $\nu^{\mathrm{mix}}$ is concentrated on a set of non-branching geodesic $A\subset\Geo(X)$. Then $\nu_x$ is also concentrated on the same set for $\alpha$-a.e.\ $x$. Note that since $A$ is a set of non-branching geodesics, the set $\mathrm{Gl}^{-1}(A)\subset\Geo(X)\times \Geo(X)$ is a graph (more precisely, for all $\gamma\in\Geo(X)$ the set of all $\eta\in\Geo(X)$ such that $(\gamma,\eta)\in \mathrm{Gl}^{-1}(A)$ is either empty or a singleton). It follows by the definition of $\nu_x = \mathrm{Gl}_{\sharp}\left(\nu_x^{\mathrm{left}}\times \nu_x^{\mathrm{right}}\right)$ that $\nu_x^{\mathrm{left}}\times \nu_x^{\mathrm{right}}$ is concentrated on the graph $\mathrm{Gl}^{-1}(A)$. This implies that $\nu_x^{\mathrm{right}}$ must be a Dirac mass.
    
     We get to a contradiction since for all $x\in \{\rho_\tau^1>0\}\cap \{\rho_\tau^2>0\}$, which has positive $\alpha$ measure by \eqref{eq:alpha_intersection}, the measure $\nu_x^{\mathrm{right}}$ has at least two points in its support. In fact, for all such $x$ there is at least one geodesic $\gamma^1 \in \supp(\nu^1)$ and one geodesic $\gamma^2\in\supp(\nu^2)$ passing through $x$. Then $\mathrm{restr}_{\tau}^1(\gamma^i)$ for $i=1,2$ are both in $\supp(\nu^{\mathrm{right}}_x)$.
   
    \textbf{Step 4.} Thanks to the previous steps, in the notation of the statement, there exists a Borel map $T^{m}:X\to X$ such that $\pi^{m}=(Id,T^{m})_\sharp \mu_0$, where we recall that $\mu_1^{m} = \sum_{j=1}^{L(m)}\lambda_j\delta_{x_j}$ with $x_i\neq x_j$ for all $i\neq j$,  and
    $\mu_0 = \tfrac{1}{\mm(A)}\mm|_A$. We use $T^{m}$ to split $A$ into subsets for which we can glue the corresponding $\MCP$ inequalities. For $j=1,\dots,L(m)$, we define
    \begin{equation}
        A_j^m:= (T^m)^{-1}(x_j),\qquad        \mu_{0,j}^m:=\frac{1}{\mm(A_j^m)}\mm|_{A_j^m}\in \P_c(X).
    \end{equation}
    We remark that $A_j^m\subseteq A$ with $\mm(A_j^m) =\lambda_j \mm(A)$, and that
    the diameter of $\supp(\mu_{0,j}^m)\cup \{x_j\}$ is less than $\varepsilon$. By definition of $\varepsilon$ (see \textbf{Step 0}), we find $\nu_j^m\in \OptGeo(\mu_{0,j}^m,\delta_{x_j})$ such that
    \begin{equation}\label{eq:MCPineqalityj}
        \frac{1}{\mm(A_j^m)}\mm \geq (e_t)_\sharp\left(\tau_{K,N}^{(1-t)}(\sfd(\gamma_0,\gamma_1))^N\nu^j \right),\qquad \forall\, t\in[0,1].
    \end{equation}
For all $j=1,\dots,L(m)$, let $\mu_{t,j}^m:=(e_t)_\sharp\nu_j^m = \rho_{t,i}^m\mm$. An elementary arguments shows that the restriction of the right hand side of \eqref{eq:MCPineqalityj} to the set $\{\rho_{t,i}^m=0\}$ is zero, so that we can refine it to
    \begin{equation}\label{eq:MCPineqalityj2}
        \frac{1}{\mm(A)}\mm|_{\{\rho_{t,j}^m > 0\}} \geq (e_t)_\sharp\left(\tau_{K,N}^{(1-t)}(\sfd(\gamma_0,\gamma_1))^N\lambda_j\nu^j \right),\qquad \forall\, t\in[0,1],
    \end{equation}
    where we also used the fact that $\mm(A_j^m)=\lambda_j \mm(A)$.

    Note that for all $i\neq j$, it must hold
    \begin{equation}\label{eq:measurepos}
        \mm\left(\{\rho_{t,i}^m>0\}\cap \{\rho_{t,j}^m>0\}\right)=0,\qquad \forall\, t\in (0,1).
    \end{equation}
    In fact, if for some $i \neq j$ and $t \in (0,1)$ the measure in \eqref{eq:measurepos} were strictly positive, one could apply the mixing argument as in \textbf{Step 3} to construct a dynamical optimal plan between the (normalized) measures $\lambda_i \mu_{0,i}^m + \lambda_j \mu_{0,j}^m$ and $\lambda_i \delta_{x_i} + \lambda_j \delta_{x_j}$, thereby violating the $\delta$-essentially non-branching assumption.\footnote{In \cite{CM-OptMaps}, a contradiction is derived from the fact that, under the essentially non-branching assumption, any optimal plan between absolutely continuous measures and finite sub of Dirac deltas is induced by a map. Here, we argue directly using the $\delta$-essentially non-branching property, since we do not have access to that result.}

    Let then $\nu^{m}:=\sum_{j=1}^{L(m)} \lambda_j \nu_j^m$. By definition $(e_0)_\sharp\nu^{m} = \mu_0$ and $(e_1)_{\sharp}\nu^{m} = \mu_1^{m}$. Furthermore, we have $(e_0,e_1)_\sharp\nu^{m} = \pi^{m}$ by construction, and so $\nu^{m}\in\OptGeo(\mu_0,\mu_1^{m})$. Taking the sum of \eqref{eq:MCPineqalityj2} over $j$, and using \eqref{eq:measurepos}, we obtain that, for all $m\in\N$, it holds
    \begin{equation}\label{eq:endstep4-Ohta}
        \frac{1}{\mm(A)}\mm\geq (e_t)_\sharp\left(\tau_{K,N}^{(1-t)}(\sfd(\gamma_0,\gamma_1))^N\nu^m\right),\qquad \forall\, t\in[0,1].
    \end{equation}
    \textbf{Step 5.} Recall that, by construction, $\mu_1^m\rightharpoonup\mu_1$. In order to take limit for $m\to \infty$ in \eqref{eq:endstep4-Ohta}, we use the stability of optimal transport for sequences of optimal dynamical plans. Since $(X,\sfd)$ is separable, complete and proper, we can apply \cite[Lemma 6.1]{CM-globalization}, for example. We obtain that there exists $\nu\in\OptGeo(\mu_0,\mu_1)$ such that, up to extraction, $\nu^m \rightharpoonup \nu$. Using this fact, we proceed taking the limit for $m\to\infty$ in \eqref{eq:endstep4-Ohta}.

    To do so, let first $B\subset X$ be an open set. From \eqref{eq:endstep4-Ohta}, for all $m\in\N$, it holds
    \begin{equation}
        \frac{1}{\mm(A)}\mm(B) \geq \int_{\Geo(X)}\mathbbm{1}_{e_t^{-1}(B)}(\gamma) g_t(\gamma)\nu^m(\di \gamma), \qquad \forall\,t\in[0,1],
    \end{equation}
    where we used the shorthand $g_t(\gamma)\equiv \tau_{K,N}^{(1-t)}(\sfd(\gamma_0,\gamma_1))^N$. We remark that for any fixed $t\in [0,1]$, the function $g_t:\Geo(X)\to [0,+\infty]$ is non-negative and continuous. Thus, for any $M>0$, the function $g_t\wedge M$ is lower semicontinuous and bounded. Similarly, since $B$ is open, the function $\mathbbm{1}_{e_t^{-1}(B)}$ is also lower semicontinuous and bounded. We then obtain
    \begin{equation}
        \frac{1}{\mm(A)}\mm(B)\geq \int_{\Geo(X)}\mathbbm{1}_{e_t^{-1}(B)}(\gamma) (g_t\wedge M)(\gamma)\nu^m(\di \gamma),
    \end{equation}
    where the integrand in the right hand side is bounded and lower semicontinuous. By \cite[Cor.\ 2.2.6]{Bogachev-Weak}, we get that for all $M>0$, all $t\in[0,1]$ and any open set $B\subset X$,
    \begin{align}
        \frac{1}{\mm(A)}\mm(B) & \geq  \liminf_{m\to\infty} \int_{\Geo(X)}\mathbbm{1}_{e_t^{-1}(B)}(\gamma) (g_t\wedge M)(\gamma)\nu^m(\di \gamma) \\
        & \geq  \int_{\Geo(X)}\mathbbm{1}_{e_t^{-1}(B)}(\gamma) (g_t\wedge M)(\gamma)\nu(\di \gamma)  \\
        & = (e_t)_{\sharp}\left(\tau_{K,N}^{(1-t)}(\sfd(\gamma_0,\gamma_1)\wedge M)^N \nu\right)(B).\label{eq:bogachev}
    \end{align}
    Since both sides of \eqref{eq:bogachev} are $\sigma$-finite Borel measures, we conclude by measure regularity that \eqref{eq:bogachev} holds also when $B$ is replaced by an arbitrary Borel set.
    
    By monotone convergence, we can take $M\to\infty$ and obtain the claimed $\MCP$ inequality.
  \end{proof}

%% file: discreteactions.tex
\section{Discrete group actions}\label{sec:discreteactions}

In this section, we study the $\MCP$ for quotients by isometric group action, in the discrete case, see \cref{sec:quotients}.

\begin{theorem}\label{thm:quotient1}
    Let $(X,\sfd,\mm)$ be a metric measure space satisfying the $\MCP_{\loc}(K,N)$ for some $K\in \R$ and $N \in [1,\infty)$. Let $G$ be a discrete group of metric measure isometries, acting properly and freely. Then, the metric measure quotient $(X^*,\sfd^*,\mm^*)$ satisfies the $\MCP_{\loc}(K,N)$. If $(X,\sfd,\mm)$ is $\delta$-essentially non-branching, then $(X^*,\sfd^*,\mm^*)$ is $\delta$-essentially non-branching.
\end{theorem}
\begin{proof}
    In this case, the projection $\q : X\to X^*$ is a local metric measure isometry between $(X,\sfd,\mm)$ and $(X^*,\sfd^*,\mm^*)$. So it is clear that $(X^*,\sfd^*,\mm^*)$ satisfies the $\MCP_{\loc}(K,N)$.

    To prove the statement about $\delta$-essentially non-branching, let $\mu_0^*\in \P_c(X^*)$ with $\mu_0^*\ll\mm^*$, $\mu_1^*=\sum_{i=1}^{L^*} \lambda_j \delta_{x^*_j} \in \P(X^*)$ and $\nu^*\in\OptGeo(\mu_0^*,\mu_1^*)$. Without loss of generality, we can assume that there are open sets $U \subset X$ and $U^* \subset X^*$ such that $\q : U \to U^*$ is a metric measure isometry, with inverse $\sigma: U^*\to U$, and that $\supp (\mu_0^*) \subset U^*$. Let $\mu_0 := \sigma_\sharp \mu_0^* \in \P_c(X)$. By construction, $\mm|_U = \sigma_\sharp (\mm^*|_{U^*})$, and hence $\mu_0 \ll \mm$.

Since $\q$ is a local isometry, we can define a geodesic lifting map $\theta: \Geo(X^*) \cap e_0^{-1}(U^*) \to \Geo(X) \cap e_0^{-1}(U)$, that associates to $\gamma^* \in \Geo(X^*)$ with $e_0(\gamma^*) \in U^*$  the unique $\gamma \in \Geo(X)$ with $\gamma_0 = \sigma(\gamma^*_0)$ and such that $\q (\gamma) = \gamma^*$. See e.g.\ \cite[Lemma 3.4.17]{BBI}. Note that $\theta$ preserves the length of geodesics, and thus the distance between their endpoints.

    Let $D := \mathrm{diam}(\supp (\mu_0^*), \supp (\mu_1^*)) < \infty$. Consider the set of geodesics $\supp (\nu^*) \subset \Geo(X^*)$. The set of lifts $\theta(\supp (\nu^*)) \subset \Geo(X)$ is contained in a compact ball of radius $D$ within the compact set $\supp (\mu_0)$. Furthermore, the set of final points $e_1(\theta(\supp (\nu^*)))$ is contained in the discrete set $\cup_{j=1}^{L^*}  \q^{-1}(x_j^*)$, so that $e_1(\theta(\supp (\nu^*)))$ is finite.
    
Therefore define:
    \begin{equation}
        \nu := \theta_\sharp \nu^* \in \P(\Geo(X)).
    \end{equation}
    By construction, $(e_0)_\sharp \nu = \mu_0$, while $\mu_1 := (e_1)_\sharp \nu$ is a finite sum of Dirac deltas. Furthermore, $\q_\sharp \mu_i = \mu_i^*$ for both $i=0,1$.

    We claim that $\nu \in \OptGeo(\mu_0, \mu_1)$. By contradiction, let $\pi := (e_0, e_1)_\sharp \nu \in \Cpl(\mu_0, \mu_1)$, and assume that there exists $\tilde{\pi} \in \Cpl(\mu_0, \mu_1)$ such that
    \begin{equation}
        \int_{X \times X} \sfd(x, y)^2 \, \tilde{\pi}(\di x \di y) < \int_{X \times X} \sfd(x, y)^2 \, \pi(\di x \di y).
    \end{equation}
    By construction, $\pi$ is concentrated on a set of $(x, y) \in X \times X$ with $\sfd(x, y) = \sfd^*(x^*, y^*)$, where $x^*=\q (x)$, $y^*=\q (y)$, so that
    \begin{equation}
        \int_{X \times X} \sfd(x, y)^2 \, \pi(\di x \di y) = \int_{X^* \times X^*} \sfd^*(x^*, y^*)^2 \, \pi^*(\di x^* \di y^*),
    \end{equation}
    where $\pi^*=(e_0,e_1)_{\sharp}\nu^*\in\Opt(\mu_0^*,\mu_1^*)$. By definition of the quotient distance, we have $\sfd(x, y) \geq \sfd^*(x^*, y^*)$ for any $x, y \in X$, so that
    \begin{equation}
        \int_{X \times X} \sfd(x, y)^2 \, \tilde{\pi}(\di x \di y) \geq \int_{X^* \times X^*} \sfd^*(x^*, y^*)^2 \, (\q_\sharp \tilde{\pi})(\di x^* \di y^*).
    \end{equation}
    Thus, the cost of $\q_\sharp \tilde{\pi} \in \Cpl(\mu_0^*, \mu_1^*)$ is strictly less than the one of $\pi^*$, giving a contradiction.

    By the $\delta$-essentially non-branching assumption, $\nu$ is concentrated on a set $\Gamma$ of non-branching geodesics, and we can assume that $\Gamma \subseteq \supp( \nu)$. The support of $\nu$ is made of geodesics that are lifts of geodesics in $\supp (\nu^*)$, and two geodesics in $\supp (\nu)$ are branching on $X$ if and only if the corresponding projections are branching on $X^*$. It follows that $\q (\Gamma) \subset \Geo(X^*)$ is a set of non-branching geodesics, and
    \begin{equation}
        \nu^*(\q (\Gamma)) = \nu^*(\theta^{-1}(\Gamma)) = (\theta_\sharp \nu^*)(\Gamma) = \nu(\Gamma) = 1.
    \end{equation}
Since $\nu^*$ was arbitrary, $(X^*, \sfd^*, \mm^*)$ is $\delta$-essentially non-branching.
\end{proof}

%% file: compactactions.tex
\section{Compact group actions}\label{sec:compactactions}

Next, we study the $\MCP$ for quotients by isometric group action, in the compact case, see \cref{sec:quotients}. We first need to recall some notation and results from \cite{GKMS-quotients}.

\subsection{Equivariant optimal transport}

We work in the notation of \cref{sec:quotients}. A measure $\mu\in\P(X)$ is \emph{$G$-invariant} if for any Borel set $E\subset X$ it holds
\begin{equation}
  \mu(gE) = \mu(E),\qquad \forall\,g\in G.
\end{equation}
We denote by $\P^G(X)\subset\P(X)$ the subset of $G$-invariant measures. Similarly, we  use the notations
\begin{equation}\label{eq:Ginv}
  \Cpl(\mu_0,\mu_1)^G,\qquad \Opt(\mu_0,\mu_1)^G,\qquad \OptGeo(\mu_0,\mu_1)^G, \qquad \P(\Geo(X))^G,
\end{equation}
for $\mu_0,\mu_1\in\P(X)^G$, where $G$-invariance is understood for the appropriate action of $G$.


Let $\p :X\to X^*$ be the projection to the metric quotient. Any measure on $X^*$ can be lifted to a $G$-invariant one via the map
\begin{equation}\label{eq:liftofmeasures}
  \Lambda:\P(X^*) \to \P^G(X), \qquad \Lambda(\mu^*)(E):= \int_X \mm_G\Big(\{g\in G\mid g \bar{x}(x^*)\cap E \neq \emptyset\}\Big)\mu^*(\di x^*),
\end{equation}
for any  Borel set $E\subset X$, where $\mm_G$ is the (bi-invariant) Haar probability measure on $G$, and where $\bar{x}:X^*\to X$ is a Borel inverse to the projection $\p :X\to X^*$ (which exists by standard measurable selection results, see e.g. \cite[Sec.\ 6.9]{Bogachev-measure}). The choice of $\bar{x}$ does not affect the definition of $\Lambda$, since $\mm_G$ is $G$-invariant.

The following result describes the relation between optimal transport on $(X,\sfd)$ and its quotient $(X^*,\sfd^*)$. It is an extension of \cite[Thm.\ 3.2, Cor.\ 3.4]{GKMS-quotients} that follows from their construction, with more details concerning (dynamical) optimal plans that we need, in particular \cref{i:equivOT3,i:equivOT4,i:equivOT5,i:equivOT6}. We only need the case $p=2$, but we provide a general statement.

\begin{lemma}[Equivariant optimal transport]\label{lem:equivOT}
  Let $\Lambda: \P(X^*) \to \P^G(X)$ be the lift of measures defined in \eqref{eq:liftofmeasures}, let $p \in [1,\infty)$, and let $\p : X \to X^*$ be the quotient map.
  Then:
  \begin{enumerate}[(i)]
    \item\label{i:equivOT1} $\Lambda: \P(X^*)\to\P^G(X)$ is an isomorphism, with inverse $\p_\sharp$, preserving the subsets of absolute continuous measures, the ones with compact supports, and those with finite $p$-momentum;
    \item\label{i:equivOT2} $\Lambda : \P_p(X^*)\to \P_p(X)\cap \P^G(X)$ is an isometric embedding with respect to the $W_p$-distance.
  \end{enumerate}
  Define the sets:
  \begin{align}
    \mathcal{OD}        & :=\{(x,y)\in X\times X \mid \sfd(x,y) = \sfd^*(\p (x),\p (y))\} \subset X\times X, \\
    \Geo_{\mathcal{OD}} & :=\{\gamma \in \Geo(X) \mid \p (\gamma)\in \Geo(X^*)\}\subset\Geo(X).
  \end{align}
  Then, for any $\mu_0,\mu_1\in\P_p(X)\cap \P^G(X)$, it holds that
  \begin{enumerate}[(i),resume]
    \item\label{i:equivOT3} any $\pi\in\Opt_p(\mu_0,\mu_1)^G$ is concentrated on $\mathcal{OD}$;
    \item\label{i:equivOT4} $\p_\sharp :\Opt_p(\mu_0,\mu_1)^G \to \Opt_p(\p_\sharp\mu_0,\p_\sharp\mu_1)$ and is surjective;
    \item\label{i:equivOT5} any $\nu\in\OptGeo_p(\mu_0,\mu_1)^G$ is concentrated on $\Geo_{\mathcal{OD}}$;
    \item\label{i:equivOT6} $\p_\sharp :\OptGeo_p(\mu_0,\mu_1)^G\to \OptGeo_p(\p_\sharp\mu_0,\p_\sharp\mu_1)$ and is surjective.
  \end{enumerate}
\end{lemma}
\begin{proof}
  \cref{i:equivOT1,,i:equivOT2} correspond to \cite[Thm.\ 3.2]{GKMS-quotients}, and we do not repeat the proof here. We only recall that key to the proof is a (non-canonical) lifting of optimal plans, namely, for any $\mu_0^*,\mu_1^*\in\P_p(X^*)$ and optimal plan $\pi^*\in\Opt_p(\mu_0^*,\mu_1^*)$, one can build a $G$-invariant plan $\pi\in\Opt_p(\mu_0,\mu_1)^G$ between the lifts $\mu_i=\Lambda(\mu_i^*)\in \P_p(X)\cap\P^G(X)$, concentrated on $\mathcal{OD}$, such that $\p_\sharp \pi = \pi^*$, and thus it holds
  \begin{equation}
    W_p(\mu_0,\mu_1)^p = \int_{X\times X}\sfd(x,y)^p\pi(\di x \di y) = \int_{X^*\times X^*} \sfd^*(x^*,y^*)\pi^*(\di x^*\di y^*) = W_p(\mu_0^*,\mu_1^*)^p.
  \end{equation}

  Starting from these facts, we prove the additional items. Let $\mu_0,\mu_1\in\P_p(X)\cap\P^G(X)$, and let $\pi\in\Opt_p(\mu_0,\mu_1)^G$ (not necessarily the one given by the above procedure). Note that $\p_\sharp \pi \in \Cpl(\p_\sharp \mu_0,\p_\sharp\mu_1)$. Then it holds
  \begin{align}
    W_p(\mu_0,\mu_1)^p
    & = \int_{X\times X} \sfd(x,y)^p\pi(\di x\di y) \\
    & = \int_{\mathcal{OD}} \sfd(x,y)^p\pi(\di x\di y) + \int_{X\times X\setminus \mathcal{OD}} \sfd(x,y)^p\pi(\di x\di y)                    \\
    & \geq \int_{\mathcal{OD}} \sfd^*(\p (x),\p (y))^p\pi(\di x\di y) + \int_{X\times X\setminus \mathcal{OD}} \sfd^*(\p (x),\p (y))^p\pi(\di x\di y) \\
    & =\int_{X^*\times X^*}\sfd^*(x^*,y^*)^p (\p_\sharp \pi)(\di x^*\di y^*)                                                                    \geq W_p(\p_\sharp\mu_0,\p_\sharp \mu_1)^p = W_p(\mu_0,\mu_1)^p,
  \end{align}
  where the first inequality is strict unless $\pi$ is concentrated on $\mathcal{OD}$, and where the second inequality is strict unless $\p_\sharp \pi\in \Opt_p(\p_\sharp\mu_0,\p_\sharp \mu_1)$. The last equality is a consequence of \cref{i:equivOT2} and the fact that, since $\mu_i$ are $G$-invariant, it holds $\mu_i=\Lambda(\p_\sharp \mu_i)$. Therefore, all inequalities must be equalities, thus proving \cref{i:equivOT3,i:equivOT4}. Note that the surjectivity part of \cref{i:equivOT4} follows from the existence of at least one $G$-invariant lift of a given optimal plan, as described earlier.

  We prove \cref{i:equivOT5,i:equivOT6}. Let $\nu\in\OptGeo_p(\mu_0,\mu_1)^G$. Then, letting $\mu_t=(e_t)_\sharp\nu$, we have that $\mu_t \in\P(X)^G$. Furthermore, letting $\mathrm{restr}_t^s:\Geo(X)\to\Geo(X)$ be the restriction map for $0\leq t<s\leq 1$, it holds
  \begin{equation}
    \nu_t^s:=(\mathrm{restr}_t^s)_{\sharp}\nu\in\OptGeo_p(\mu_t,\mu_s)^G.
  \end{equation}
  The fact that the restriction of optimal dynamical plans is still an optimal dynamical plan simply follows from the triangle inequality of the Wasserstein distances. Since $(e_0,e_1)\circ\mathrm{restr}_t^s = (e_t,e_s)$, we have $(e_t,e_s)_\sharp\nu\in \Opt_p(\mu_t,\mu_s)^G$. By \cref{i:equivOT3}, $(e_t,e_s)_\sharp\nu$ is concentrated on $\mathcal{OD}$, and thus $\nu$ is concentrated on the Borel set $(e_t,e_s)^{-1}(\mathcal{OD})$.
  It follows that $\nu$ is concentrated on the set of geodesics $\gamma\in\Geo(X)$ such that
  \begin{equation}
    \sfd(\gamma_t,\gamma_s)=\sfd^*(\p (\gamma)_t,\p (\gamma)_s),\qquad \forall \, t,s\in [0,1]\cap \mathbb{Q},
  \end{equation}
  but by continuity this is equivalent to $\p (\gamma)\in\Geo(X^*)$, thus proving \cref{i:equivOT5}.

  Finally, note that $(e_0,e_1)_\sharp \circ \p_\sharp \nu = \p_\sharp \circ (e_0,e_1)_\sharp\nu$ (with some abuse of notation, the first $\p$ is the one acting on $\Geo(X)$, while the second $\p$ is the one acting on $X\times X$). Thus, by \cref{i:equivOT4}, $\p_\sharp$ maps $\OptGeo_p(\mu_0,\mu_1)^G$ to $\OptGeo_p(\p_\sharp\mu_0,\p_\sharp\mu_1)$, which proves the first part of \cref{i:equivOT6}.

  To prove the surjectivity part of \cref{i:equivOT6}, fix a Borel geodesic lifting map $\theta:\Geo(X^*)\to\Geo(X)$ such that $\p\circ \theta = Id$ (this can be done via a standard measurable selection argument, e.g.\ \cite[Thm.\ 6.9.6]{Bogachev-measure}).
  With this choice, for any $\nu^*\in\OptGeo_p(\p_\sharp\mu_0,\p_\sharp\mu_1)$, define
  \begin{equation}
    \nu\in\P(\Geo(X))^G,\qquad  \nu(E):= \int_G (\theta_\sharp\nu^*)(g E)\,\mm_{G}(\di g),
  \end{equation}
  for any Borel set $E\subset\Geo(X)$, and where $\mm_{G}$ is the Haar probability measure of $G$. By construction, $\nu$ is $G$-invariant and $\p_\sharp\nu = \nu^*$ (since $\p\circ \theta = Id$). Furthermore, $\nu$ is concentrated on $\Geo_{\mathcal{OD}}$, and thus $\sfd(\gamma_0,\gamma_1) = \sfd^*(\p (\gamma)_0, \p (\gamma)_1)$ for $\nu$-a.e.\ $\gamma$. Thus, letting $\pi = (e_0,e_1)_\sharp\nu$, we have that $(P_1)_{\sharp}\pi = \mu_0$ and $(P_2)_{\sharp}\pi = \mu_1$, that $\p_\sharp\pi = (e_0,e_1)_\sharp \nu^* \in \Opt_p(\p_\sharp\mu_0, \p_\sharp\mu_1)$, and that $\pi$ is concentrated on $\mathcal{OD}$.
  Thus, we have that
  \begin{align}
    \int_{X\times X} \sfd(x,y)^p \pi(\di x\di y)
    & = \int_{X\times X} \sfd^*(\p (x),\p (y))^p \pi(\di x\di y)                  \\
    & = \int_{X^*\times X^*} \sfd^*(x^*,y^*)^p \p_\sharp\pi(\di x^*\di y^*) \\
    & = W_p(\pi_\sharp\mu_0, \pi_\sharp\mu_1)^{p} = W_p(\mu_0, \mu_1)^{p},
  \end{align}
  where in the last equality we used \cref{i:equivOT2}. It follows that $\pi$ is optimal, and thus we have found $\nu\in\OptGeo_p(\mu_0,\mu_1)^G$ such that $\p_\sharp\nu = \nu^*$.
\end{proof}

\subsection{Uniform approximations of invariant measures}

We intend to apply \cref{thm:MCPineqdeltaess} to $G$-invariant lifts $\mu_0,\mu_1 \in \P_c(X)$ of measures $\mu_0^*,\mu_1^* \in \P_c(X^*)$ for the projection $\p :X\to X^*$. Even if $\mu_0^*,\mu_1^*$ are supported within an arbitrarily small ball in $X^*$, approximations of $\mu_1$ with finite sums of Dirac masses (which can always be found) will not generally satisfy assumption \eqref{eq:essup} of \cref{thm:MCPineqdeltaess}. The following lemma shows that, under suitable assumptions, we can select special approximations that satisfy this property. We do not know whether the proof of \cref{lem:goodapprox} can be adapted in the case of more general compact groups and choices of $\mu_1$, but there are obstructions to our construction when $G$ in non-abelian, see the footnote in the proof.

\begin{lemma}[Uniform approximations of $G$-invariant measures by Dirac masses]\label{lem:goodapprox}
  Let $(X,\sfd)$ be a metric space, and let $G$ be a connected, compact, abelian Lie group acting freely on $X$ by isometries.  Let $\p :(X,\sfd)\to (X^*,\sfd^*)$ be the projection to the metric quotient. Let $\mu_0^*\in \P_c(X^*)$ and $\mu_1^*=\delta_{o^*}\in\P_c(X^*)$, for some $o^*\in X^*$. Consider the unique $\mu_0,\mu_1\in \P_c^G(X)$ such that
  \[
    \p_\sharp \mu_i = \mu_i^*,\qquad i=0,1.
  \]
  Then, there exists a sequence of finite sums of Dirac masses $\mu_1^{m}\in\P_c(X)$, and optimal plans $\pi^{m}\in \Opt(\mu_0,\mu_1^{m})$, for $m\in \N$, such that
  \begin{enumerate}[(a)]
    \item \label{i:a} $\mu_1^{m} \rightharpoonup \mu_1$;
    \item \label{i:b} $\supp(\mu_1^{m})\subseteq \supp(\mu_1) = \p^{-1}(o^*)$, and each $\mu_1^{m}$ is $G_m$-invariant;
    \item \label{i:c} $
      \displaystyle \limsup_{m\to\infty}\left(\pi^{m}-\essup\sfd\right)\leq   \diam\left(\supp(\mu_0^*) \cup \supp(\mu_1^*)\right).
      $
  \end{enumerate}
\end{lemma}
\begin{proof}
  Since the action of $G$ is free and transitive on the fibers, we can identify $\p^{-1}(o^*)=\supp(\mu_1)$ with $G$ via the homeomorphism $g \mapsto go$ for a fixed $o \in \supp(\mu_1)$. In this way, $\sfd|_{\supp(\mu_1)}$ is identified with a bi-invariant metric $\sfd_G$ on $G$, and $\mu_1$ is identified with the Haar probability measure $\mm_G$ on $G$, so that $(\supp(\mu_1), \sfd|_{\supp(\mu_1)}, \mu_1)$ and $(G, \sfd_G, \mm_G)$ are identified as metric measure spaces. We remark that under our assumptions, $G$ is the standard torus $\S^1 \times \dots \times \S^1$. Thus, with an elementary construction\footnote{For any $\varepsilon>0$ we can find a finite subgroup $G_m$ of the torus $G$ that is an $\varepsilon$-net for $(G,\sfd_G)$, and a fundamental domain $D_m\subset G$ for the action of $G_m$ on $G$ such that $G = \cup_{g\in G_m} g D_m$, and $\mm_G(\partial D_m)=0$. We note that such an approximation cannot exist, even in a weaker sense, if $G$ is not abelian, see \cite{Turing-approximations}.}, we can find for all $m \in \N$ finite subgroups $G_m < G$ with ${\# G_m}  \nearrow \infty$, open sets $A_i^m \subseteq \supp(\mu_1)$, and points $y_i^m \in A_i^m$ for $i = 1, \dots, {\# G_m}$ such that
  \begin{enumerate}[(1)]
    \item\label{i:A1} $A_i^m\cap A_j^m = \emptyset$ if $i\neq j$;
    \item\label{i:A2} $\mu_1\left(\cup_{i=1}^{{\# G_m}}A_i^m\right) = 1$;
    \item\label{i:A3} $\diam(A_i^m)\leq 1/m$ for all $i=1,\dots,m$;
    \item\label{i:A4} $G_m$ acts transitively and freely on $\{y_1^m,\dots,y_{{\# G_m}}^m\}$ and on $\{A_1^m,\dots,A_{{\# G_m}}^m\}$, so that  in particular $\mu_1(A_i^m) = 1/{\# G_m}$ for all $i=1,\dots,{\# G_m}$.
  \end{enumerate}
  For all $m\in \N$, define $\mu_1^m\in\P_c(X)$ by
  \begin{equation}
    \mu_1^m = \sum_{i=1}^{{\# G_m}} \mu_1(A_i^m) \delta_{y_i^m} = \frac{1}{{\# G_m}} \sum_{i=1}^{{\# G_m}} \delta_{y_i^m},
  \end{equation}
  By \labelcref{i:A1,i:A2,i:A3} we can use \cref{lem:approxwithdelta}, yielding that $\mu_1^m \rightharpoonup \mu_1$, which is \cref{i:a}. By construction, $\supp(\mu_1^m) \subseteq \supp(\mu_1)$, and by \labelcref{i:A4}, the measure $\mu_1^m$ is $G_m$-invariant, i.e.\ \cref{i:b}.

  To prove \cref{i:c}, note that both $\mu_0, \mu_1^m$ are $G_m$-invariant. It is easy to prove that the averaging of a coupling with respect to any subgroup of $G$ (in this case, w.r.t.\ $G_m$) does not change its $W_2$-cost. It follows that there exists a $G_m$-invariant optimal plan $\pi^m \in \Opt(\mu_0, \mu_1^{m})$. Let $(x,y) \in \supp( \pi^m) \subseteq \supp (\mu_0) \times \supp (\mu_1^m)$. There is $\tilde{y}\in\supp(\mu_1) = \p^{-1}(o^*)$ such that
  \begin{equation}
    \sfd(x,\tilde{y}) = \sfd^*(\p (x),o^*) \leq  \diam(\supp(\mu_0^*)\cup\supp(\mu_1^*))=:D.
  \end{equation}
  By \labelcref{i:A2}, any point in $\supp( \mu_1)$ must be in the closure of some $A_i^m$, but by \labelcref{i:A3} these sets have diameter $\leq 1/m$, so that $\{y_1^m,\dots,y_{{\# G_m}}^m\}$ is a $2/m$-net in $\supp(\mu_1)$. It follows that there is $g \in G_m$ (indeed $g y \in \supp( \mu_1^m)$) such that
  \begin{equation}
    \sfd(x,g y)\leq \sfd(x,\tilde{y}) + \sfd(\tilde{y},gy) \leq D+\frac{ 2}{m}.
  \end{equation}
  Let $k$ be the order of $g \in G_m$ (the least integer such that $g^k=e$). Since $\pi^m$ is $G_m$-invariant
  \begin{equation}
    \left(g^{j} x, g^{j} y\right) \in \supp(\pi^m),\qquad\forall\, j =1,\dots,k.
  \end{equation}
  We test $\sfd^2$-cyclical monotonicity of $\supp(\pi^m)$ for the set $\{(x,y),(gx,gy),\dots,(g^{k-1}x,g^{k-1}y)\}$ and the shift $\{(x,gy),\dots,(g^{k-1}x,y)\}$. It holds
  \begin{equation}
    \sfd(x,y)^2 =\frac{1}{k} \sum_{j=1}^{k} \sfd\left(g^{j-1} x,g^{j-1}y\right)^2 \leq \frac{1}{k}\sum_{j=1}^k \sfd\left(g^{j-1} x,g^j y\right)^2 =\sfd(x,gy)^2 \leq \left(D +\frac{2}{m}\right)^2.
  \end{equation}
  This yields, since $(x,y)\in\supp(\pi^m)$ is arbitrary, that
  \begin{equation}
    \pi^m-\essup \sfd = \sup\Big\{\sfd(x,y) \,\Big|\, (x,y)\in\supp(\pi^m)\Big\} \leq D + \frac{2}{m},
  \end{equation}
  concluding the proof of \cref{i:c}.
\end{proof}

\subsection{Local measure contraction properties on the quotient}

We are now ready to prove the main result of this section, showing that, under the $\delta$-essentially non-branching assumption, the \emph{local} $\MCP$ is preserved under quotients by compact groups of isometries.

\begin{theorem}\label{thm:quotient2}
  Let $(X,\sfd,\mm)$ be a $\delta$-essentially non-branching metric measure space satisfying the $\MCP_{\loc}(K,N)$ for some $K\in \R$ and $N \in [1,\infty)$. Let $G$ be a connected, compact, abelian Lie group of metric measure isometries, acting freely. Then, the metric measure quotient $(X^*,\sfd^*,\mm^*)$ satisfies the $\MCP_{\loc}(K,N)$.
\end{theorem}
\begin{remark}[The case of global $\MCP$]\label{rmk:globalMCP}
  The additional assumptions on $G$ are necessary to verify \eqref{eq:essup} of \cref{thm:MCPineqdeltaess} by means of \cref{lem:goodapprox}. If one assumes $\MCP(K,N)$ instead of $\MCP_{\loc}(K,N)$, then \cref{thm:MCPineqdeltaess} holds in greater generality (see \cref{rmk:MCPglob}), and we do not need \cref{lem:goodapprox}. In this case, \cref{thm:quotient2} is valid for a general compact topological group $G$ acting (possibly not freely) on $X$ by metric measure isometries.
\end{remark}
\begin{remark}[Non-branching properties of the quotient]\label{rmk:nonbranchingquot}
  In \cite[Cor.\ 3.5]{GKMS-quotients}, it is proved that the essentially non-branching property descends to quotients by compact group actions. We do not know whether this holds for the $\delta$-essentially non-branching property (this is true in the discrete case, see \cref{thm:quotient1}).
\end{remark}
\begin{proof}
  Let $U^*\subseteq X^*$ be an open, bounded set, and let $U:=\p^{-1}(U^*)\subseteq X$. Let $\varepsilon > 0$ be the constant from \cref{thm:MCPineqdeltaess}. Let $\mathcal{O}^*$ be any open set contained in $\p (U)$ such that $\diam(\mathcal{O}^*) < \varepsilon$.
  Pick $\mu_0^*\in \P_c(X^*)$ of the form $\mu_0^*=\tfrac{1}{\mm^*(A^*)}\mm^*|_{A^*}$ for some Borel set $A^*\subset \mathcal{O}^*$ with $0<\mm(A^*)<\infty$. Let also $o^*\in \mathcal{O}^*$ and $\mu_1^*:=\delta_{o^*}$. In particular, it holds that
  \begin{equation}\label{eq:boundondiameter}
    \diam\left(\supp(\mu_0^*)\cup \supp(\mu_1^*)\right) < \varepsilon.
  \end{equation}
  Let $\mu_0, \mu_1$ be the unique $G$-invariant lifts of $\mu_0^*$ and $\mu_1^*$; see \cref{lem:equivOT}. Letting $A := \p^{-1}(A^*)$, we have $A \subset U$, $\mm(A) = \mm^*(A^*)$, and thus $\mu_0 = \tfrac{1}{\mm(A)}\mm|_{A}$.
  %
  %
  Furthermore, by \cref{lem:goodapprox}, we find a sequence of finite sums of Dirac masses $\mu_1^m \in \P_c(X)$ such that $\mu_1^m \rightharpoonup \mu_1$, and optimal plans $\pi^m \in \Opt(\mu_0, \mu_1)$ satisfying \eqref{eq:essup} (thanks to \eqref{eq:boundondiameter} paired with \cref{i:c} of \cref{lem:goodapprox}). Thus, we can apply \cref{thm:MCPineqdeltaess}, and we obtain $\nu \in \OptGeo(\mu_0, \mu_1)$ such that
  \begin{equation}\label{eq:MCPineqdeltaess-application-Ohta}
    \frac{1}{\mm(A)}\mm \geq (e_t)_{\sharp}\left(\tau_{K,N}^{(1-t)}(\sfd(\gamma_0,\gamma_1))^N\nu\right),\qquad \forall\, t\in [0,1].
  \end{equation}

  Since averaging preserves the transport cost of plans with $G$-invariant marginals, letting $\mm_G$ be the Haar probability measure on $G$, we have that $\overline{\nu} \in \OptGeo(\mu_0, \mu_1)^G$, where
  \begin{equation}
    \overline{\nu} := \int_G \left(g_\sharp \nu\right) \mm_{G}(\di g).
  \end{equation}
  We can then take the average of \eqref{eq:MCPineqdeltaess-application-Ohta}, and using the fact that both $\mm$ and $\sfd$ are $G$-invariant, we obtain the validity of \eqref{eq:MCPineqdeltaess-application-Ohta} for $\overline{\nu}$ in place of $\nu$, namely
  \begin{equation}\label{eq:MCPineqdeltaess-application2-Ohta}
    \frac{1}{\mm(A)}\mm \geq (e_t)_{\sharp}\left(\tau_{K,N}^{(1-t)}(\sfd(\gamma_0,\gamma_1))^N\overline{\nu}\right),\qquad \forall\, t\in [0,1].
  \end{equation}

  To conclude, observe that by \cref{i:equivOT5,,i:equivOT6} of \cref{lem:equivOT}, $\overline{\nu}$ is concentrated on the set $\Geo_{\mathcal{OD}}$ of geodesics that project to geodesics on the quotient, and the projection $\nu^* := \p_\sharp \overline{\nu}$ is an optimal dynamical plan between its marginals, namely $\nu^* \in \OptGeo(\mu_0^*, \delta_{o^*})$. In other words, for $\overline{\nu}$-a.e.\ $\gamma \in \Geo(X)$, we have that $\gamma^* := \p (\gamma) \in \Geo(X^*)$ and $\sfd(\gamma_0, \gamma_1) = \sfd^*(\gamma_0^*, \gamma_1^*)$. Moreover, $\p \circ e_t = e_t \circ \p$. Thus, applying $\p_{\sharp}$ to both sides of \eqref{eq:MCPineqdeltaess-application2-Ohta}, we obtain
  \begin{equation}\label{eq:MCPineqdeltaess-application3-Ohta}
    \frac{1}{\mm^*(A^*)}\mm^* \geq (e_t)_{\sharp}\left(\tau_{K,N}^{(1-t)}(\sfd^*(\gamma^{*}_0,\gamma^{*}_1))^N\nu^*\right),\qquad \forall\, t\in [0,1].
  \end{equation}
  In summary, we can cover $X^*$ with neighborhoods $\mathcal{O}^*$ such that, for any $\mu_0^* \in \P_c(X^*)$ of the form $\mu_0^* = \tfrac{1}{\mm^*(A^*)} \mm^*|_{A^*}$, where $A^* \subset \mathcal{O}^*$ is a Borel set with $0 < \mm^*(A^*) < \infty$, and for any $o^* \in \mathcal{O}^*$, there exists $\nu^* \in \OptGeo(\mu_0^*, \delta_{o^*})$ for which \eqref{eq:MCPineqdeltaess-application3-Ohta} holds.
\end{proof}

%% file: carnot.tex
\section{Carnot homogeneous spaces}\label{sec:Carnothomogeneousspaces}

The following discussion presents the construction of Carnot homogeneous spaces, which appear as tangent cones of general sub-Riemannian manifolds, as explained in \cite{Bellaiche}. At a first read, it could be helpful to consider the case of Carnot groups, which corresponds to choosing the trivial subgroup $\H={e}$ in the notation used below.

\begin{definition}[Stratified groups]
  A Lie algebra $\g$ is \emph{stratified} if it admits a decomposition
  \begin{equation}
    \g = \g_1\oplus \dots\oplus \g_s,
  \end{equation}
  where, setting $\g_i =0$ for all $i\geq s+1$, it holds
  \begin{equation}
    [\g_1,\g_i] = \g_{i+1}, \qquad \forall\, i=1,\dots,s.
  \end{equation}
  The subspaces $\g_i$ are called \emph{strata}, and $s$ is the \emph{step} of the stratification. A \emph{stratified group} is a simply connected Lie group $\G$ whose Lie algebra is stratified. A \emph{Carnot group} is a stratified Lie group $\G$ together with a choice of scalar product on $\g_1$. We identify the Lie algebra $\g$ with the subspace of left-invariant vector fields, so that we usually call the scalar product of a Carnot group a \emph{left-invariant} scalar product on $\g_1$.
\end{definition}

Any stratified group/algebra is nilpotent (of step $s$), and as a consequence of this and simple connectedness, the group exponential map $\exp_{\G} : \g \to \G$ is be a smooth diffeomorphism.

\begin{definition}[Dilations]\label{def:dilations}
  Let $\G$ be a stratified group with stratified Lie algebra $\g = \g_1\oplus \dots\oplus \g_s$. The \emph{dilation} of factor $\lambda \in \R$ is the Lie group endomorphism $\delta_\lambda : \G\to \G$  such that
  \begin{equation}\label{eq:dilation}
    (\delta_\lambda)_* V = \lambda^i V,\qquad \forall\, V \in \g_i,\quad \forall\, i=1,\dots,s,
  \end{equation}
  where the star denotes the push-forward. Note that if $\lambda \neq 0$ then $\delta_\lambda$ is invertible with inverse $\delta_{1/\lambda}$. With a slight abuse of notation, we will also use the term ``dilation'' to refer to the Lie algebra endomorphism $(\delta_\lambda)_* : \g \to \g$.
\end{definition}

%

Let $\G$ be a stratified group, and fix a closed, dilation-invariant subgroup $\H<\G$. This is equivalent to the Lie subalgebra $\h$ of $\H$ being dilation-invariant, or equivalently, to $\h$ having a stratification compatible with that of $\g$, namely
\begin{equation}\label{eq:dil-inv-sub}
  \h = \h_1\oplus \dots\oplus \h_s, \qquad \text{with} \qquad \h_i\subseteq \g_i, \quad \forall\,i=1,\dots,s.
\end{equation}
Note that $\h$ need not be generated by $\h_1$, meaning that $\H$ itself need not be a stratified group.

Let $\M:= \G \slash \H $ be the quotient of $\G$ by the action of $\H$ on it by left multiplication (i.e.\ the set of right cosets). The smooth manifold $\M$ has dimension $\dim\G-\dim\H$ and inherits several natural structures, as outlined below in \cref{sec:homogstructure,sec:dilstructure,sec:distinguishedLie,sec:metstructure,sec:invmeas}.

\subsection{Homogeneous structure}\label{sec:homogstructure}

The left multiplication by elements of $\G$ does not, in general, induce a left action of $\G$ on $\M$ (unless $\H$ is normal). However, the \emph{right} multiplication of $\G$ always induces a \emph{right} action $\beta:\G \times \M \to \M$ defined by
\begin{equation}
  \beta(g',\H g) := \H(gg'), \qquad \forall\, g,g'\in \G.
\end{equation}
Since it is a right action, we use the appropriate shorthand $\H   gg'=\beta(g',\H g)$.  Of course, if $\H=\{e\}$ so that $\M=\G$, then $\beta$ coincides with the right multiplication on $\G$.

\subsection{Dilations}\label{sec:dilstructure}
Since $\H$ is dilation-invariant, the dilations of $\G$ induce dilations on $\M$ by setting
\begin{equation}\label{eq:dilationsonM}
  \delta_\lambda(\H g) := \H \delta_\lambda(g),\qquad \forall\, g\in \G,\,\lambda\in \R.
\end{equation}
The fixed point of dilations of $\M$ is the coset of the identity, namely $\H e$, that we call the \emph{origin}. Of course, if $\M=\G$, these dilations coincide with the ones given by \cref{def:dilations}.

\subsection{Distinguished algebra of vector fields}\label{sec:distinguishedLie}
Let $\pi :\G\to \M$ be the quotient map. Its push-forward induces a Lie algebra homomorphism $\pi_* : \g \to \mathrm{Vec}(\M)$. Thus, although $\M$ is not generally a group, it still carries a distinguished stratified Lie algebra of vector fields given by $\pi_*\g\subset \mathrm{Vec}(\M)$. For $\pi_*$ to be well-defined, it is crucial that $\M$ is defined as the set of right cosets, and not the left ones.

The dilation-invariant subalgebra $\pi_*\h \subset \pi_*\g$  consists of the vector fields in $\pi_*\g$ that vanish at the origin $\H e$. The distinguished subalgebras of vector fields $\pi_*\g$ and $\pi_*\h$ are both dilation-invariant, and $\pi_*\g$ is stratified with strata $\pi_*\g_i$, where $(\delta_\lambda)_*$ acts by multiplication by $\lambda^i$.

Of course, if $\M=\G$, the distinguished algebra of vector fields is just the Lie algebra of left-invariant vector fields on $\G$.

\subsection{Sub-Riemannian metric structure}\label{sec:metstructure}

Assume that $\G$ is a Carnot group. Any orthonormal basis $\{X_1,\dots,X_k\}\subset \g_1\subset\mathrm{Vec}(\G)$ induces smooth vector fields $\{\pi_*X_1,\dots,\pi_*X_k\} \subset \pi_*\g_1\subset\mathrm{Vec}(\M)$, which in turn define a sub-Riemannian structure on $\M=\G \slash \H$, in the sense of \cref{a:SR}. We denote by $\sfd_{\M}$ the corresponding sub-Riemannian metric. One can easily verify that $\sfd_{\M}$ does not depend on the choice of the orthonormal basis. Furthermore, the dilations introduced in \cref{sec:dilstructure} satisfy
\begin{equation}
  \sfd_{\M}\left(\delta_{\lambda}(\H g),\delta_{\lambda}(\H g')\right) =\lambda \sfd_{\M}\left(\H g,\H g'\right), \qquad \forall \, g,g'\in \G,\,\lambda \in \R.
\end{equation}
\begin{remark}[Carnot-Carathéodory metrics]
  When $\M=\G$, the above construction yields the standard Carnot-Carathéodory metric on the Carnot group. It is left-invariant, in the sense that
  \begin{equation}
    \sfd_{\G}(g g',g g'') =\sfd_{\G}(g',g''),\qquad \forall\, g,g',g''\in\G.
  \end{equation}
  We emphasize that $\sfd_{\G}$ is also not right-invariant unless $\G$ is abelian.
\end{remark}
\begin{remark}[Equivalence with the quotient metric]\label{rmk:equiv_with_quotient_metric}
  If $\G$ is a Carnot group with Carnot–Ca\-ra\-théo\-do\-ry metric $\sfd_{\G}$, and $\H<\G$ is a closed, dilation-invariant subgroup, then setting $\M = \G \slash \H$ and defining the metric $\sfd_{\M}$ as described in \cref{sec:metstructure}, we have
  \begin{equation}\label{eq:dmetric}
    \sfd_{\M}(\H g,\H g') = \inf_{h,h'\in \H} \sfd_{\G}(hg,h'g') = \inf_{h\in \H} \sfd_{\G}(g,hg'),
  \end{equation}
  where the second equality follows from the left-invariance of $\sfd_{\G}$. In other words, the sub-Riemannian metric   coincides with the quotient metric described in \cref{sec:quotients} (note that the action of $\H$ on $\G$ is proper).  We may freely use either formulation depending on context.
\end{remark}

\subsection{Invariant measure}\label{sec:invmeas}

For a reference to what follows, see \cite[Sec.\ 2.7]{Federer} or \cite[Sec.\ 15]{HR-harmonic}. Since both $\H$ and $\G$ are unimodular, $\M$ has a unique (up to rescaling) right-invariant Haar measure, namely a Borel measure invariant under the right action of $\G$ on $\M$ (see \cref{sec:homogstructure}). We denote a choice of such a measure by $\mm_{\M}$. This measure is then characterized by the property
\begin{equation}
  \mm_{\M}(E g) = \mm_{\M}(E),
\end{equation}
for all Borel set $E\subset \M$, where $E g= \beta(g,E)$. We call $\mm_{\M}$ an \emph{invariant measure} on $\M$, or also a \emph{right-invariant measure} if we want to emphasize the right action.

When $\M=\G$, the measure $\mm_{\G}$ is (a choice of) the classical right-invariant Haar measure. Furthermore, since $\G$ is nilpotent and therefore unimodular, $\mm_{\G}$ is also left-invariant, and thus bi-invariant. In other words, for any Borel set $E\subset \G$ and all $g\in\G$, it holds
\begin{equation}
  \mm_{\G}(g E) = \mm_{\G}(E g) = \mm_{\G}(E).
\end{equation}

\subsection{Carnot homogeneous spaces} We are now ready to give a general definition putting together all the structures defined in \cref{sec:homogstructure,sec:dilstructure,sec:distinguishedLie,sec:metstructure,sec:invmeas}.

\begin{definition}[Carnot homogeneous space]\label{def:Carnothomo}
  A \emph{Carnot homogeneous space} is the quotient $\M=\G \slash \H$ of a Carnot group $\G$ by a dilation-invariant closed subgroup $\H<\G$. Let $\sfd_{\M}$ be the corresponding quotient metric, and $\mm_{\M}$ a choice of invariant measure. We denote by $(\M,\sfd_{\M})$ and $(\M,\sfd_{\M},\mm_{\M})$ the resulting metric space and metric measure space, respectively.
\end{definition}

The terminology ``Carnot homogeneous space'' is used in \cite[Sec.\ 5.6]{Bellaiche}, while some authors prefer the terminology ``self-similar space''.

As a metric space, any Carnot homogeneous space $(\M,\sfd_{\M})$ is complete, separable, geodesic, and proper. We stress that $\sfd_{\M}$ is not invariant under the right action of $\G$ on $\M$ described in \cref{sec:homogstructure}, unless $\H$ is normal and $\M = \G \slash \H$ is abelian. This is a consequence of the fact that the Carnot-Carathéodory metric on $\G$ is left-invariant, but not right-invariant in general. Thus, Carnot homogeneous spaces are not homogeneous in the sense of metric geometry.

Any Carnot group equipped with its left-invariant Carnot-Carathéodory metric and a bi-invariant Haar measure is a Carnot homogeneous space, corresponding to the choice $\H = \{e\}$ in \cref{def:Carnothomo}.  More generally, we have the following.

\begin{remark}[Consistency with Carnot groups]\label{rmk:consistency}
  If $\H\triangleleft \G$ is normal, then $\M = \G\slash\H$ is a stratified group, with group multiplication given by $(\H g)(\H g'):=\H gg'$, and stratified Lie algebra isomorphic to $\g_1\slash\h_1\oplus\dots\oplus \g_s\slash\h_s$. By construction, the dilations $\delta_{\lambda}:\M\to\M$ of \cref{sec:dilstructure} coincide with the ones of \cref{def:dilations}, seeing $\M$ as a stratified group. If $\G$ is a Carnot group, then $\M$ is also a Carnot group (with scalar product on $\g_1\slash\h_1$ induced by the one on $\g_1$ via the surjective map $\pi_*:\g_1\to\g_1\slash\h_1$), and $\sfd_{\M}$ is the corresponding Carnot-Carathéodory metric,  left-invariant by the left multiplication of $\M$. Finally, by construction, the measure $\mm_{\M}$ is invariant under the right action $\beta:\G\times \M\to\M$ defined in \cref{sec:homogstructure}. In this case, we note that this right action corresponds to the right multiplication in $\M$, in the sense that
  \begin{equation}
    (\H g)(\H g') = \beta(g',\H g), \qquad \forall g,g'\in\G.
  \end{equation}
  It follows that $\mm_{\M}$ is a classical right-invariant Haar measure on the Lie group $\M$, and since the latter is unimodular, $\mm_{\M}$ is also invariant under the left multiplication of $\M$. In other words, the measure $\mm_{\M}$ is in this case a Haar bi-invariant measure on the group $\M$. Therefore, if $\H$ is a normal subgroup, all the structures defined in the previous sections for the Carnot homogeneous space $\M=\G \slash \H$ coincide with the ones obtained by regarding $\M$ directly as a Carnot group.
\end{remark}

\begin{remark}[Minimal group and irreducible presentation]\label{rmk:minimalquotient}
  Different choices of $\G$ and $\H$ may yield the same Carnot homogeneous space $\M$ (i.e.\ the same smooth manifold, the same stratified Lie algebra of vector fields $\pi_*\g$, and the same subalgebra $\pi_*\h$ of vector fields vanishing at the origin, up to isomorphisms). However, there is a ``smallest'' possible choice. Letting $\k:=\ker\pi_*|_{\g}$, it is easy to see that $\k \subset\g$ is a dilation-invariant ideal, and it is clear that $\k\subseteq\h$. In particular, $\K=\exp_{\G}(\k)$ is a closed, stratified normal subgroup of both $\H$ and $\G$. Therefore, it holds:
  \begin{equation}
    \M = \G \slash \H =  (\G\slash\K) \slash (\H\slash\K).
  \end{equation}
  A given Carnot homogeneous space $\M= \G \slash \H$ can thus be realized equivalently as the quotient $\M=\tilde{\G} \slash \tilde{\H}$, where $\tilde{\G}:=\G\slash\K$ and $\tilde{\H}:=\H\slash\K$ (which are stratified groups with Lie algebras denoted by $\tilde{\g}$ and $\tilde{\h}$, respectively). In this way, the corresponding Lie algebra homomorphism $\tilde{\pi}_*:\tilde{\g}\to\mathrm{Vec}(\M)$ is injective, and the original stratified Lie algebras $\pi_*\h\subset\pi_*\g\subset\mathrm{Vec}(\M)$ are \emph{isomorphic} (as stratified Lie algebras) to the Lie algebras of $\tilde{\H}$ and $\tilde{\G}$. Furthermore, if $\G$ is a Carnot group, then $\tilde{\G}$ is also a Carnot group and its sub-Riemannian distance coincides with the quotient distance. Then, for all $g,g'\in \G$, it holds:
  \begin{align}
    \sfd_{\tilde{\G} \slash \tilde{\H}}\big(\tilde{\H} \K g,\tilde{\H} \K g'\big) & = \inf_{\K h,\K h' \in \tilde{\H}} \sfd_{\tilde{\G}}\big((\K h)(\K g),(\K h')(\K g')\big) \\
    & = \inf_{\K h,\K h' \in \tilde{\H}}\sfd_{\G \slash \K}(\K hg,\K h'g') \\
    & = \inf_{h,h'\in \H} \inf_{k,k'\in \K} \sfd_\G(khg,k'h'g') \\
    & = \inf_{h,h'\in \H} \sfd_{\G}(hg,h'g') = \sfd_{\G \slash \H}(g,g'),
  \end{align}
  where we used that $\K$ is a normal subgroup of $\H$ and $\G$. Therefore, the metric structure induced on $\M$ by realizing $\M$ as $\tilde{\G} \slash \tilde{\H}$ is the same as one obtained by realizing $\M$ as $\G \slash \H$.

  We call $\tilde{\G}$ the \emph{minimal group} of the Carnot homogeneous space, and $\tilde{\H}$ the \emph{minimal subgroup}. The minimal group and subgroup of a Carnot homogeneous space are unique up to isomorphisms of stratified groups, since they are the ones associated with the Lie algebras $\pi_*\h\subset\pi_*\g\subset\mathrm{Vec}(\M)$. We say that the presentation of the Carnot homogeneous space $\M = \G \slash \H$ is \emph{irreducible}.
\end{remark}

\subsection{Isometries}\label{sec:isometries}

If $(X,\sfd,\mm)$ and $(\tilde{X},\tilde{\sfd},\tilde{\mm})$ are metric measure spaces, a \emph{metric isometry} is a distance-preserving invertible map from $X$ to $\tilde{X}$, and a \emph{metric measure isometry} a metric isometry that also pushes forward $\mm$ to $\tilde{\mm}$. It is known that any metric isometry between Carnot groups is smooth, see \cite{Ham-isometries,LDO-Isometries}. Furthermore, up to a constant scaling of the measure, any metric isometry is also a metric measure isometry, see \cite{nostropopp}. We do not know whether these results hold for Carnot homogeneous spaces. For this reason, we directly define isometries of Carnot homogeneous spaces as smooth maps that lift to isometries of Carnot groups.

\begin{definition}[Smooth isometries of Carnot homogeneous spaces]\label{def:iso}
  Let $\M$ and $\tilde{\M}$ be Carnot homogeneous spaces with irreducible presentations $\M = \G \slash \H$ and $\tilde{\M} = \tilde{\G} \slash \tilde{\H}$. A \emph{smooth isometry of Carnot homogeneous spaces} is a diffeomorphism $\phi:\M\to \tilde{\M}$ for which there exists a metric isometry $\Phi:\G\to\tilde{\G}$ such that the following diagram commute:
  \begin{equation}\label{eq:cdisometry}
    \begin{tikzcd}
      \G \arrow[r,"\Phi"] \arrow[d,"\pi"] & \tilde{\G} \arrow[d,"\tilde{\pi}"] \\
      \M= \G \slash \H \arrow[r,"\phi"] & \tilde{\G} \slash \tilde{\H}=\tilde{\M}
    \end{tikzcd}
  \end{equation}
  We will say in this case that $\M$ and $\tilde{\M}$ are smoothly isometric.
\end{definition}

\begin{proposition}[Smooth isometries are metric measure isometries]\label{prop:iso} Let $(\M,\sfd_{\M},\mm_{\M})$ and $(\tilde{\M},\sfd_{\tilde{\M}},\mm_{\tilde{\M}})$ be Carnot homogeneous spaces. Then there exists a constant $c>0$ (depending only on the choices of invariant measures) such that any smooth isometry $\phi:\M\to\tilde{\M}$ is a metric measure isometry between $(\M,\sfd_{\M},\mm_{\M})$ and $(\tilde{\M},\sfd_{\tilde{\M}}, c \mm_{\tilde{\M}})$.
\end{proposition}
\begin{proof}
  Observe that from \eqref{eq:cdisometry}, for any $g\in\G$, it must hold
  \begin{equation}\label{eq:intertwining}
    \Phi(\H g) = \tilde{\H}\Phi(g).
  \end{equation}
  It follows that for all $g,g'\in \G$ we have
  \begin{align}
    \sfd_{\tilde{\M}}(\phi(\H g),\phi(\H g')) & = \sfd_{\tilde{\M}}(\tilde{\H}\Phi(g),\tilde{\H}\Phi(g')) & \text{by \eqref{eq:cdisometry}}\\
    & = \inf_{\tilde{h}\in\tilde{\H}}\sfd_{\tilde{\G}}(\Phi(g),\tilde{h}\Phi(g')) \\
    & = \inf_{h\in\H}\sfd_{\tilde{\G}}(\Phi(g),\Phi(h g')) & \text{by \eqref{eq:intertwining}}\\
    & = \inf_{h\in\H}\sfd_{\G}(g,h g') & \text{$\Phi$ is a metric isometry}\\
    & = \sfd_{\M}(\H g,\H g'),
  \end{align}
  proving that $\phi$ is a metric isometry.

  The proof of the fact that $\tilde{\phi}$ is a metric \emph{measure} isometry is ultimately a consequence of the fact that $\G$ is unimodular. A short proof goes as follows.
  One special choice of invariant metric on a Carnot group is given by the Popp measure, see \cite{montgomerybook}. As proven in \cite{nostropopp}, it is preserved by any metric isometry.\footnote{This is proved in \cite{nostropopp} for the case of smooth \emph{self}-isometries of equiregular sub-Riemannian manifolds. Isometries between (a priori different) Carnot groups  are diffeomorphisms by \cite{Ham-isometries}, and therefore must send a global horizontal frame to a corresponding one. Thus, one may apply the same argument as in \cite[Prop.\ 5]{nostropopp}.} In other words, letting $\mm_{\G}$ and $\mm_{\tilde{\G}}$ be the Popp measure on $\G$ and $\tilde{\G}$, respectively, it holds
  \begin{equation}\label{eq:preservemG}
    \Phi_\sharp \mm_{\G} = \mm_{\tilde{\G}}.
  \end{equation}

  Define the Haar measures $\mm_{\H} :=\mm_{\G}|_{\H}$ and $\mm_{\tilde{\H}}=\mm_{\tilde{\G}}|_{\tilde{\H}}$ on $\H$ and $\tilde{\H}$, respectively. As a consequence of \eqref{eq:intertwining}, \eqref{eq:preservemG} and the right-invariance of $\mm_{\G}$, $\mm_{\tilde{\G}}$, we obtain
  \begin{equation}\label{eq:preservemH}
    (\Phi \circ R_{g})_{\sharp} \mm_{\H} = (R_{\Phi(g)})_{\sharp} \mm_{\tilde{\H}},\qquad \forall\,g\in \G.
  \end{equation}
  where $R$ denote the right multiplication.

  By \cite[Thm.\ 1.2.13]{CG-NilpotentGroupsBook}, a right-invariant measure on $\M = \G\slash\H$, which we may assume is precisely $\mm_{\M}$, can be computed in terms of $\mm_{\G}$ and $\mm_{\H}$ as follows. Let $\beta \in C(\G)$ be such that (i) $\beta\geq 0$; (ii) for any compact $K \subset \M$, $\supp(\beta)\cap\pi^{-1}(K)$ is compact, and (iii) $\int_{\H} \beta(hg)\mm_{\H}(\di h)=1$ for all $g\in \G$. Then, we have
  \begin{equation}
    \int_{\M} u(m)\mm_{\M}(\di m) = \int_{\G}\beta(g)( u\circ\pi)(g)\mm_{\G}(\di g),\qquad \forall\, u \in C_c(\M).
  \end{equation}
  The function $\tilde{\beta}_{\Phi} := \beta\circ\Phi^{-1}\in C(\tilde{\G})$ satisfies conditions (i), (ii), and, due to \eqref{eq:preservemH}, also (iii). Then, a change of variable argument shows that
  \begin{equation}\label{eq:moltetilde}
    \int_{\tilde{\M}} \tilde{u}(\tilde{m})(\phi_{\sharp}\mm_{\M})(\di \tilde{m}) = \int_{\tilde{\G}}\tilde{\beta}_{\Phi}(\tilde{g})( \tilde{u}\circ\tilde{\pi})(\tilde{g})\mm_{\tilde{\G}}(\di \tilde{g}),\qquad \forall\, \tilde{u} \in C_c(\tilde{\M}).
  \end{equation}
  It follows that $\phi_{\sharp}\mm_{\M}$ is an invariant measure on $\tilde{\M}$. A priori, it depends on $\phi$ through $\tilde{\beta}_{\Phi}$.

  It remains to show that different choices of smooth isometries yield the same invariant measure on $\tilde{\M}$. Let $\phi_1,\phi_2:\M\to\tilde{\M}$ be two smooth isometries, corresponding to metric isometries $\Phi_1,\Phi_2:\G\to\tilde{\G}$ satisfying \eqref{eq:cdisometry}. In the notation of \eqref{eq:moltetilde}, $\tilde{\beta}_{\Phi_2} = \tilde{\beta}_{\Phi_1} \circ \bar{\Phi}$, where $\bar{\Phi} = \Phi_1\circ\Phi_2^{-1}:\tilde{\G}\to \tilde{\G}$ is a metric self-isometry of $\tilde{\G}$. We again use the fact that metric isometries of Carnot groups preserve the Popp measure, and since $\bar{\Phi}$ is a \emph{self}-isometry, it must hold that $\bar{\Phi}_{\sharp} \mm_{\tilde{\G}} = \mm_{\tilde{\G}}$. It follows by a change-of-variable argument in \eqref{eq:moltetilde} that $\phi_{1\sharp}\mm_{\M} = \phi_{2\sharp}\mm_{\M} = c \mm_{\M'}$.
\end{proof}

\begin{definition}[Engel and Martinet]\label{def:EngMar}
  The \emph{Engel group} $\Eng$ is the stratified group with stratified algebra $\e = \e_1 \oplus \e_2 \oplus \e_3$, for which there exists a basis $X_1,X_2,Y,Z$ satisfying
  \begin{equation}
    \e_1 = \mathrm{span}\{X_1,X_2\},\qquad \e_2 = \mathrm{span}\{Y\}, \qquad \e_3 = \mathrm{span}\{Z\},
  \end{equation}
  and with the following non-zero bracket relations:
  \begin{equation}
    [X_1,X_2] = Y, \qquad [X_1,Y] = Z.
  \end{equation}
  It is a Carnot group when equipped with the scalar product making $X_1,X_2$ orthonormal.

  The \emph{Martinet structure} $\Mar$ is the Carnot homogeneous space realized as the quotient $\Mar=\Eng \slash \H$ of the Engel group by the closed dilation-invariant subgroup $\H$ whose Lie algebra is $\h = \e_2$. The corresponding sub-Riemannian structure, which has constant rank, is generated by the family of vector fields $\pi_*X_1$, $\pi_*X_2$.
\end{definition}

\begin{lemma}[Uniqueness of $\Eng$ and $\Mar$]\label{lem:uniquenessEngMar}
  Let $\M$ be a Carnot homogeneous space whose minimal group is isomorphic (as a stratified group) to the Engel group, and whose sub-Riemannian distribution has constant rank. Then precisely one of the following two cases holds:
  \begin{itemize}
    \item either $\M$ is a Carnot group, in which case $\M$ is isometric to $\Eng$;
    \item or $\M$ is  not a Carnot group, in which case $\M$ is isometric to $\Mar$.
  \end{itemize}
\end{lemma}
\begin{proof}
  By \cref{rmk:minimalquotient}, we can assume that $\M=\G \slash \H$ where $\G$ is isomorphic as a stratified Lie group to the Engel group, and $\H$ is a closed, dilation-invariant subgroup. The subgroup $\H$ is determined by a stratified subalgebra $\h\subset \g$, say
  \begin{equation}
    \h = \h_1\oplus\h_2\oplus\h_3,\qquad \text{with}\qquad \h_i\subseteq \g_i,\quad i=1,2,3.
  \end{equation}
  If $\pi_*:\g \to \mathrm{Vec}(\M)$ has non-trivial kernel, the distinguished Lie algebra $\pi_*\g$ is not isomorphic to the Engel algebra. Thus, we must have $\ker\pi_*=\{0\}$.

  It is well-known that, letting $\exp_{\G}:\g\to\G$ be the group exponential map, it holds
  \begin{equation}\label{eq:adexp}
    \Ad_{\exp_{\G}(X)}(W) = W + \sum_{i=1}^{s-1} \frac{1}{i!} \ad_X^i(W),\qquad \forall \, X,W\in \g,
  \end{equation}
  see e.g.\ \cite[p.\ 12]{CG-NilpotentGroupsBook}, where $\Ad: \G\to \mathrm{GL}(\g)$ and $\ad:\g \to \mathfrak{gl}(\g)$ are defined respectively by
  \begin{equation}
    \Ad_g(W) :=\left.\frac{d}{dt}\right|_{t=0}  g \exp_{\G}(t W)g^{-1}, \qquad \ad_X(Y):=[X,Y],\qquad \forall\,g\in \G,\, X,Y,W\in \g.
  \end{equation}
  Using \eqref{eq:adexp} it follows that, for $g \in \G$ and $V\in \g$, it holds
  \begin{equation}\label{eq:characterizationofzeroes}
    (\pi_* V)|_{\H g} =0,\qquad \Longleftrightarrow \qquad \Ad_g(V) \in \h.
  \end{equation}
  Taking $V\in\h_3$ in \eqref{eq:characterizationofzeroes} we see that $\pi_*V$ is the zero vector field, but this contradicts the fact that $\ker\pi_*=\{0\}$ so that $\h_3 =\{0\}$. Similarly, taking $V\in \h_1$, in \eqref{eq:characterizationofzeroes} we see that $\pi_*V|_{\H e}=0$, but this contradicts the constant rank assumption unless $\h_1=\{0\}$.

  The only possibility left is that $\h \subseteq \g_2$. There are two cases: $\h = \{0\}$ or $\h=\g_2$. The corresponding structures correspond to the Engel group and the Martinet structure, respectively.

  We prove now that, in both cases, the corresponding structures are isometric (as Carnot homogeneous spaces) to $\Eng$ or $\Mar$, respectively. Consider the following fact. Any Carnot group $\G$ with a stratified Lie algebra $\g = \g_1 \oplus \g_2 \oplus \g_3$, where $\dim \g_1 = 2$ and $\dim \g_2 = \dim \g_3 = 1$, is isometric to $\Eng$. Indeed, for any $B\in \g_2$, $B\neq 0$, we can write
  \begin{equation}
    \g_1 = \left(\ker\ad_{B}\right)^\perp \oplus \ker \ad_{B}.
  \end{equation}
  Choose an orthonormal basis $\{A_1,A_2\}$ of $\g_1$ adapted to the above splitting, with $A_2 \in \ker\ad_B$, and we can assume that $B = [A_1,A_2] \in \g_2$. By construction, $C := [A_1,B] \in \g_3$ is non-zero. By construction, there is a unique stratified Lie algebra isomorphism $\Phi_*:\g \to \e$ such that
  \begin{equation}
    \Phi_*: A_1\mapsto X_1,\quad A_2\mapsto X_2,\quad B\mapsto Y,\quad C\mapsto Z,
  \end{equation}
  which lifts to a stratified Lie group isomorphism $\Phi : \G\to \Eng$. By construction, $\Phi_*$ preserves the scalar product between the first strata, and $\Phi$ is an isometry between Carnot groups.

  Thus, for the choice $\h = \{0\}$, there exists an isometry $\Phi: \G \to \Eng$ (the one just constructed). For the choice $\h=\g_2$, letting $\Phi$ be the aforementioned isometry, note that $\Phi_*(\g_2)=\e_2$, so that $\Phi(\exp(\g_2)) = \exp(\e_2)$. Therefore, the map $\Phi$ descends to a well-defined diffeomorphism on the quotients, $\phi: \G \slash \exp(\g_2) \to \Eng\slash \exp(\e_2) = \Mar$, which, by construction, is an isometry of Carnot homogeneous spaces (see \cref{prop:iso}).
\end{proof}
\begin{remark}[Carnot homogeneous spaces from Engel] Besides $\Eng$ and $\Mar$, other Carnot homogeneous spaces have Engel as their minimal group. They all take the form $\M=\Eng\slash\H$, with $\H<\Eng$ being a closed, dilation-invariant subgroup with Lie algebra $\h\subset\e$. Classifying them is equivalent to finding all dilation-invariant subalgebras $\h\subset \e$ such that $\pi_*:\e \to\mathrm{Vec}(\M)$ has trivial kernel. As discussed in the proof above, we must have $\h_3=\{0\}$. A case-by-case analysis reveals that we may choose:
  \begin{enumerate}[label=(\roman*)]
    \item \label{case-Engel} $\h=\{0\}$, in which case $\M=\Eng$;
    \item \label{case-Martinet} $\h = \e_2$, in which case $\M=\Mar$;
    \item \label{case-2Grushin} $\h = \mathrm{span} \{X_2, Y\}$, in which case we obtain a rank-varying sub-Riemannian structure of dimension $2$ with an Engel-type growth vector, known as the $2$-Grushin plane \cite{B-alphagrushin};
    \item \label{case-Grushin-Engel} $\h =\mathrm{span}\{a X_1 + bX_2\}$ with $(a,b)\neq (0,0)$, in which case we obtain a rank-varying sub-Riemannian structure of dimension $3$ with an Engel-type growth vector.
  \end{enumerate}
  In this work, we consider only $\Eng$ and $\Mar$ in the list above, and we will need to exploit the explicit knowledge of the geodesic structures of $\Mar$, see \cref{sec:EngMart}.
\end{remark}

\subsection{Quotients to Engel or Martinet}

We study Carnot groups admitting the Martinet structure as a quotient. In doing so, we prove a fact of independent interest but key to our work: admitting a quotient to $\Eng$ is equivalent to admitting one to $\Mar$.

\begin{theorem}[Stratified groups admitting Martinet or Engel quotients]\label{thm:quotienttoMartinet}
  Let $\G$ be a Carnot group with Lie algebra $\g$ and step $s \geq 3$. The following conditions are equivalent:
  \begin{enumerate}[label=(\roman*)]
    \item \label{item:quot-martinet} there exists a closed, dilation-invariant subgroup $\K<\G$ such that $ \G \slash \K$ is a Carnot homogeneous space smoothly isometric to $\Mar$;
    \item \label{item:quot-engel} there exists a closed, and dilation-invariant normal subgroup $\H \triangleleft \G$ such that $\G \slash \H$ is a Carnot group smoothly isometric to $\Eng$;
    \item \label{item:quot-algebraicondition} there exists $\h_3\subset \g_3$ with codimension $1$ such that the subspace
      \begin{equation}\label{eq:subspace}
        \h_2:=\{Y \in \g_2 \mid [\g_1,Y] \in \h_3\}\subset \g_2
      \end{equation}
      has codimension $1$.
  \end{enumerate}
\end{theorem}
\begin{remark}[A sufficient condition]\label{rmk:almostfree}
   Letting $k:=\dim\g_1$, the equivalent characterizations of \cref{item:quot-martinet-intro,,item:quot-engel-intro,,item:quot-algebraicondition-intro} are verified provided that $k=2$, or $k>2$ and
  \begin{equation}\label{eq:almostfree}
    \dim\G \geq (k-1)\left(\frac{k^2}{3}+\frac{5k}{6}+1\right).
  \end{equation}
  This condition appeared in \cite[Lemma 4]{AG-subanal}. Note that the right hand side of \eqref{eq:almostfree} is one less than the dimension of the free stratified group with $k$ generator and step $3$.
\end{remark}

\begin{proof}
  $\labelcref{item:quot-engel} \Rightarrow \labelcref{item:quot-algebraicondition}$. To avoid confusion with the statement, we momentarily denote by $\bar{\h}$ the Lie algebra of $\H$. Observe that $\H \triangleleft \G$ if and only if its Lie algebra $\bar{\h} \subseteq\g$ is an ideal. Since $\H$ is dilation-invariant, $\bar{\h}$ admits a compatible stratification as in \eqref{eq:dil-inv-sub}. Thus, the quotient group $\G \slash \H$ has its stratified Lie algebra isomorphic to $\g \slash \bar{\h}$, with stratification
  \begin{equation}
    \g\slash\bar{\h} = \g_1\slash\bar{\h}_1 \oplus \g_2\slash\bar{\h}_2 \oplus \dots \oplus \g_s\slash \bar{\h}_s.
  \end{equation}
  By assumption, $\g/\bar{\h}$ is isomorphic to the Engel algebra. In particular, each stratum $\g_j/\bar{\h}_j$ is isomorphic to the corresponding stratum of the Engel algebra. Therefore, we must have
  \begin{equation}
    \bar{\h}_i = \g_i,\qquad \text{for } i=4,\dots,s,
  \end{equation}
  and, moreover,
  \begin{align}
    \dim \g_3 & = \dim \bar{\h}_3 +1, \\
    \dim \g_2 & = \dim \bar{\h}_2 +1, \\
    \dim \g_1 & = \dim \bar{\h}_1 +2.
  \end{align}
  Set then $\h_3 := \bar{\h}_3$, and recall the definition of $\h_2$ from \eqref{eq:subspace}. Since $\bar{\h}$ is an ideal, we have $\bar{\h}_2 \subseteq \h_2$, so the codimension of $\h_2$ is at most $1$. It cannot be zero; otherwise, we would have $\h_2 = \g_2$, and thus, by construction, $[\g_2,\g_1] = [\h_2,\g_1] \subseteq \h_3 \subsetneq \g_3$, contradicting the stratification assumption. Hence, we must have $\bar{\h}_2 = \h_2$.

  $\labelcref{item:quot-algebraicondition} \Rightarrow \labelcref{item:quot-engel}$. Let $\h_3 \subseteq \g_3$ and $\h_2\subseteq \g_2$ be as in the assumption.
  Since $\h_3$ has has codimension one, we may choose $Z \in \g_3$ such that
  \begin{equation}
    \g_3 = \R Z \oplus \h_3.
  \end{equation}
  There also exist $Y \in \g_2$ and $X_1 \in \g_1$ such that $[X_1,Y]=Z$. In particular, we have that
  \begin{equation}
    \g_2 = \R Y \oplus \h_2.
  \end{equation}
  Consider now the map
  \begin{equation}
    P:=\mathrm{proj}_{Z}\circ \ad_{Y}:\g_1 \to \g_3.
  \end{equation}
  It is has rank $1$ since $P(X_1) = -Z$, so that
  \begin{equation}
    \g_1 = \R X_1 \oplus \ker P.
  \end{equation}
  We claim that
  \begin{equation}\label{eq:claimkerP}
    [\ker P,\ker P] \subseteq \h_2.
  \end{equation}
  Indeed, taking $X,X'\in \ker P$, we have that $[X,X']=\alpha Y +\h_2$ for some $\alpha \in \R$. By the Jacobi equation, we obtain
  \begin{align}
    0 & = [X_1,[X,X']]+[X',[X_1,X]]+ [X,[X',X_1]] \\
    & = [X_1,\alpha Y + \h_2] + [X',\beta Y + \h_2] + [X, \gamma Y+\h_2]\\
    & = \alpha Z +\h_3,
  \end{align}
  for some $\beta,\gamma \in \R$, where in the last equality we used the fact that $X,X'\in \ker P$ and $[\g_1,\h_2]\subseteq \h_3$. It follows that $\alpha =0$, proving the claim.

  As a consequence of \eqref{eq:claimkerP} and the stratification assumption, there is $X_2 \in \ker P$ such that $[X_1,X_2]\notin \h_2$. In particular, we can choose $X_2 \in \ker P$ such that
  \begin{equation}
    Y':=[X_1,X_2] = Y + \h_2.
  \end{equation}
  Using the definition of $\h_2$ and the fact that $[X_1,Y]=Z$, it holds
  \begin{equation}
    Z':=[X_1,Y'] = [X_1,Y+\h_2] = Z + \h_3.
  \end{equation}
  Consider now the linear map
  \begin{equation}
    Q:=\mathrm{proj}_{Y}\circ \ad_{X_1} : \ker P \to \R Y,
  \end{equation}
  where the projection is taken according to the splitting $\g_2 = \R Y \oplus \h_2$. The map $Q$ is surjective because $Q(X_2) = Y$. Thus, denoting by $\h_1 := \ker Q$, we have
  \begin{equation}
    \ker P = \R X_2 \oplus \h_1.
  \end{equation}
  In summary, the splittings we have obtained so far are
  \begin{align}
    \g_1 & = \R X_1 \oplus \R X_2 \oplus \h_1, \\
    \g_2 & = \R Y' \oplus \h_2, \\
    \g_3 & = \R Z' \oplus \h_3.
  \end{align}
  Finally, set $\h_j := \g_j$ for all $j=4,\dots,s$.

  We claim that $\h:=\h_1 \oplus \dots \oplus \h_s$ is an ideal. In fact, $\h_1 \subseteq \ker P$ and using \eqref{eq:claimkerP}, we have
  \begin{equation}
    [\h_1,\h_1] \subseteq [\ker P,\ker P] \subseteq \h_2.
  \end{equation}
  Since also $X_2 \in \ker P$, it follows that
  \begin{equation}
    [X_2,\h_1] \subseteq [\ker P,\ker P]\subseteq \h_2.
  \end{equation}
  Furthermore, since $\h_1 = \ker Q$, and using the definition of $Q$, we have
  \begin{equation}
    [X_1,\h_1] \subseteq \h_2,
  \end{equation}
  proving that $[\g_1,\h_1]\subseteq \h$. We also have
  \begin{equation}
    [\g_2,\h_1] = [\R Y + \h_2,\h_1] \subseteq [Y,\h_1] + [\h_2,\h_1] \subseteq \h_3,
  \end{equation}
  where we have used that $\h_1 \subseteq \ker P$ implies $[Y,\h_1]\subseteq \h_3$, together with the definition of $\h_2$. Finally, by construction of $\h_2$, it holds $[\g_1,\h_2]\subseteq \h_3$.

  We have thus proved that
  \begin{equation}
    [\g_1,\h_1]\subseteq \h, \qquad [\g_2,\h_1]\subseteq \h,\qquad [\g_1,\h_2]\subseteq \h.
  \end{equation}
  Any other $[\g_j,\h_k]$ for $j+k\geq 4$ not listed there belongs to $\g_j = \h_j$ for some $j\geq 4$ by stratification, and this proves that $[\g,\h]\subseteq \h$. Thus, $\H:=\exp_{\G}(\h)$ is a dilation-invariant normal subgroup and $\G \slash \H$ is a Carnot group isomorphic (as a stratified group) to $\Eng$. The Carnot homogeneous space $\G\slash\H$ is isometric to $\Eng$ by \cref{lem:uniquenessEngMar}.

  $\labelcref{item:quot-martinet} \Rightarrow \labelcref{item:quot-algebraicondition}$. By \cref{def:iso} of a smooth isometry between Carnot homogeneous spaces, the projection map $\pi: \G \to \G/\K = \M$ induces a stratified Lie algebra homomorphism $\pi_*:\g \to \mathrm{Vec}(\M)$ whose image is isomorphic to the Engel algebra. Setting $\h := \ker \pi_*$, we deduce that $\g/\h$ is isomorphic to the Engel algebra. We then repeat the steps from $\labelcref{item:quot-engel} \Rightarrow \labelcref{item:quot-algebraicondition}$.

  $\labelcref{item:quot-engel} \Rightarrow \labelcref{item:quot-martinet}$. Remember that isometric Carnot groups must be isomorphic, see \cite{LDO-Isometries}. Thus, by assumption, there exists a stratified Lie group isomorphism $\Psi:\Eng \to \G/\H$. Consider the ideal $\h \subset \g$, and define
  \begin{equation}
    \k:=\mathrm{span}\{w \in \g \mid w + \h \in \Psi_*\e_2\}.
  \end{equation}
  Note that $\k$ is a closed, dilation-invariant subalgebra of $\g$ with $\dim\k = \dim\h + 1$. Furthermore, observe that $\k$ is not an ideal (otherwise, $\e_2$ would be an ideal in the Engel algebra $\e$). Let $\K = \exp_{\G}(\k) < \G$ denote the corresponding closed, dilation-invariant subgroup. By definition, $\M := \G/\K$ is a Carnot homogeneous space. Consider the quotient map $\pi:\G\to \M$ and its push-forward $\pi_*:\g\to\mathrm{Vec}(\M)$. By construction, the minimal group of $\M$ is isomorphic to $\g\slash\ker\pi_*$. Indeed, $\ker\pi_*$ is an ideal, and by construction, $\ker\pi_* \subseteq \k$. Thus, by the previous observation, we must have $\ker\pi_*\subsetneq\k$. Furthermore, it is easy to show that $\h\subseteq\ker\pi_*$ (since $\h$ is an ideal). Therefore, we must have $\ker\pi_*=\h$, so the minimal group of $\M$ is $\G \slash\H$, and the minimal subgroup is $\K\slash \H =\Psi(\exp(\e_2))$, see \cref{rmk:minimalquotient}. It follows that $\M$ is a Carnot homogeneous space (that is not a Carnot group) whose distribution has constant rank, and its minimal group is isomorphic to the Engel group. By \cref{lem:uniquenessEngMar}, $\M$ is isometric to $\Mar$.
\end{proof}

\subsection{The factorization argument}\label{sec:tower}

To motivate the next construction, suppose that $\G$ is a Carnot group admitting a closed, dilation-invariant subgroup $\K< \G$ so that $\M=\G\slash\K$ is a Carnot homogeneous space. Our aim is to deduce validity (resp.\ failure) of the $\MCP$ for $\M$ (resp.\ $\G$) from the validity of the $\MCP$ for $\G$ (resp.\ $\M$). The whole \cref{sec:quotients} was developed  for this purpose. Unfortunately, except in the trivial case,$\K$ is neither compact nor discrete. Hence we cannot apply neither \cref{thm:quotient1} (for discrete group actions) nor \cref{thm:quotient2} (for compact group actions). Suppose, however, that $\K$ possesses a co-compact lattice, namely a discrete, normal subgroup $\Gamma\triangleleft \K$ such that $\K\slash\Gamma$ is a compact Lie group (we will see that it acts on $\G\slash\Gamma$ by metric measure isometries). In this case, we can factor the quotient map $\pi :\G\to \M$ as
\begin{equation}\label{eq:quotientnaive}
  \begin{tikzcd}
    \pi: \G \arrow[r,"\q"] & \G\slash\Gamma \arrow[r,"\p"] & (\G\slash\Gamma)\slash(\K\slash\Gamma) =\M,
  \end{tikzcd}
\end{equation}
where $\q$ is a quotient by a discrete group action and $\p$ is a quotient by a compact group of isometries, to which one could apply \cref{thm:quotient1,thm:quotient2}. However, it is well-known that a simply connected nilpotent Lie group admits a co-compact lattice if and only if it is rational, i.e.\ its Lie algebra admits a basis for which all structure constants are rational numbers \cite[Sec.\ 5]{CG-NilpotentGroupsBook}. This condition does not always hold. Furthermore, even in that case, the quotient $\K\slash\Gamma$ might not be abelian, which is a key assumption required in the proof of \cref{thm:quotient2}. To overcome these obstacles, we replace the naive decomposition \eqref{eq:quotientnaive} with a more elaborate construction, consisting of a ``tower'' of simpler quotients such that, at each step, both \cref{thm:quotient1,,thm:quotient2} can be applied. This construction is the content of \cref{thm:tower} below.

To state it, we need a version of \eqref{eq:quotientnaive} where both sides are Carnot homogeneous spaces. Thus, we first prove the following (mostly tautological) result, which generalizes the fact that, if $\H \triangleleft \G$ is a normal, closed, dilation-invariant subgroup of a Carnot group, then the quotient $\G\slash\H$ is itself a Carnot group (see \cref{rmk:consistency}, corresponding to setting $\H = \{e\}$ in the next statement).

\begin{proposition}[Isometric actions on Carnot homogeneous spaces]\label{prop:quotientofCarnothomogeneous}
  Let $\M=\G\slash\H$ be a Carnot homogeneous space, and let $\Htilde<\G$ be a closed, dilation-invariant subgroup with $\H\triangleleft\Htilde$. Then the left multiplication of $\Htilde$ on $\G$ induces a proper and free action of $R:=\Htilde\slash\H$ on $\M$ by metric measure isometries. The quotient of $\M$ by this action, equipped with the quotient metric, is smoothly isometric to the Carnot homogeneous space $\tilde{\M}=\G\slash\Htilde$, in other words
  \begin{equation}
    \M\slash R = \left(\G\slash\H\right) \slash\left( \Htilde\slash\H\right) = \G\slash\Htilde = \tilde{\M}
  \end{equation}
  as metric spaces. We denote by $\pi:\M\to\tilde{\M}$ the corresponding quotient map.
\end{proposition}
\begin{proof}
  The action of the group $R=\Htilde\slash\H$, denoted by $\alpha:R\times \M\to\M$, is defined by
  \begin{equation}
    \alpha(\H \tilde{h},\H g) :=\H \tilde{h} g.
  \end{equation}
  This action is well-defined because $\H\triangleleft\Htilde$. It is immediate to verify that this action is left, smooth, proper, and free. For $\tilde{h}\in \Htilde$, denote by $\phi = \alpha(\H \tilde{h},\cdot):\M\to\M$. Letting $\Phi = L_{\tilde{h}}:\G\to \G$ be the left multiplication by $\tilde{h}$ (which is an isometry of $\G$), we see that the diagram \eqref{eq:cdisometry} commutes, so that $\phi$ is an isometry by \cref{def:iso}.

  Denote by $(\M^*,\sfd_{\M^*})$ the metric quotient of $(\M,\sfd_{\M})$ by the isometric action of $R$ (in the sense of \cref{sec:quotients}). It is elementary to verify that the map sending the $R-$orbit of $\H g\in \M=\G\slash\H$ to $\Htilde g\in \tilde{\M}= \G\slash\Htilde$ is a (smooth) metric isometry between $(\M^*,\sfd_{\M^*})$ and $(\tilde{\M},\sfd_{\tilde{\M}})$.
\end{proof}

\begin{theorem}[Factorization of quotients between Carnot homogeneous spaces]\label{thm:tower}
  Let $\G$ be a Carnot group of step $s$.  Let $\M=\G\slash \H$ be a Carnot homogeneous space admitting a quotient to $\tilde{\M}=\G\slash \Htilde$ in the sense of \cref{prop:quotientofCarnothomogeneous}. Then, there exist
  \begin{enumerate}[(a)]
    \item\label{i:towera} closed, dilation-invariant subgroups $\H_i<\G$ for $i=1,\dots,s$, with $\H_0=\H$ and $\H_s = \Htilde$, satisfying
      \begin{equation}
        \H_0\triangleleft\H_1\triangleleft\dots\triangleleft\H_{s-1}\triangleleft \H_{s}   < \G;
      \end{equation}
    \item\label{i:towerb} Carnot homogeneous spaces $\M_i = \G\slash \H_i$ for $i=0,\dots,s$, with $\M_0=\M$ and $\M_s=\tilde{\M}$,  each equipped with their (left-invariant) Carnot-Carathéodory metrics and a choice of right-invariant measure;
    \item\label{i:towerc}  for $i=0,\dots,s-1$, each $\M_{i}$ admits a quotient $\pi_{i+1}:\M_{i}\to\M_{i+1}$ in the sense of \cref{prop:quotientofCarnothomogeneous}. More precisely, the quotient group $R_{i}:=\H_{i+1}\slash \H_{i}$ is an abelian group acting freely and properly on $\M_i$ by metric isometries such that
      \begin{equation}
        \M_{i+1} = \M_{i}\slash  R_{i},
      \end{equation}
      as metric spaces;
    \item\label{i:towerd} discrete normal subgroups $\Gamma_i \triangleleft R_i$, for $i=0,\dots,s-1$, such that
      \begin{itemize}
        \item[--] $\Gamma_i$ acts freely and properly on $\M_i$ by metric measure isometries. We denote by $\M_i\slash\Gamma_i$ the corresponding quotient metric measure space and $\q_i:\M_i\to \M_i\slash\Gamma_i$ the projection map (see \cref{sec:quotients});
        \item[--] $R_i\slash\Gamma_i$ is a compact abelian Lie group, acting freely by metric measure isometries on $\M_i\slash\Gamma_i$. We denote by $(\M_i\slash\Gamma_i)\slash(R_i\slash\Gamma_i)$ the corresponding quotient metric measure space and $\p_i : \M_i\slash\Gamma_i \to (\M_i\slash\Gamma_i)\slash(R_i\slash\Gamma_i) $ the projection map  (see \cref{sec:quotients});
      \end{itemize}
      such that, for all $i=0,\dots,s-1$, the quotient $\pi_{i+1}: \M_{i} \to \M_{i+1}$ can be factorized as
      \begin{equation}\label{eq:factorization}
        \begin{tikzcd}
          \pi_{i+1}: \M_i \arrow[r,"\q_{i +1}"] & \M_i\slash\Gamma_i \arrow[r,"\p_{i +1}"] & (\M_i\slash\Gamma_i)\slash(R_i\slash\Gamma_i) =\M_{i+1}.
        \end{tikzcd}
      \end{equation}
  \end{enumerate}
  In particular, the quotient $\pi:\M\to\tilde{\M}$ can be factorized as a finite number of compositions of local metric measure isometries (the $\q_i$'s) and quotients by compact groups of metric measure isometries (the $\p_i$'s), as described by the following diagram:
  \begin{equation}
    \begin{tikzcd}[sep=scriptsize]
      \M =\G\slash\H \arrow[rrrrrr,"\pi"] 
      & & &  & & & \tilde{\M} =\G\slash\tilde{\H} \\
      \M_0   \arrow[equal,u] \arrow[r,"\q_1"] & \M_0\slash\Gamma_0 \arrow[r,"\p_1"] & \M_{1} \arrow[r] & \dots  \arrow[r]  & \M_{s-1} \arrow[r,"\q_{s}"] &  \M_{s-1}\slash\Gamma_{s-1} \arrow[r,"\p_{s}"] & \M_{s} \arrow[equal,u]
    \end{tikzcd}
  \end{equation}
\end{theorem}
\begin{remark}[The case of Carnot groups]\label{rmk:tower_Carnot_group}
  If $\M = \G\slash\H$ and $\tilde{\M} = \G\slash\Htilde$ are Carnot groups, then each subgroup $\H_i$ is normal in $\G$, so that every $\M_i = \G\slash\H_i$ is a Carnot group.
\end{remark}
\begin{remark}[Repetitions of factors]\label{rmk:repetitions}
  Note that the list of factors $\M_i$ in \cref{thm:tower} may include repetitions, resulting in some of the maps $\pi_i$ being trivial. For example, consider the case $\M = \Eng$ and $\tilde{\M} = \Mar$. Using the notation from \cref{def:EngMar}, we have here $\H = \{e\}$ and $\tilde{\H} = \exp(\e_2)$, so that $\H_1 = \{e\}$ and $\H_2 = \H_3 = \exp(\e_2)$. It follows that the corresponding factorization from \cref{thm:tower} is given by
  \begin{equation}
    \begin{tikzcd}[sep=scriptsize]
      \M_0\arrow[equal,d] \arrow[r,"\pi_1"] &
      \M_1\arrow[equal,d] \arrow[r,"\pi_2"] &
      \M_2\arrow[equal,d] \arrow[r,"\pi_3"] &
      \M_3\arrow[equal,d] \\
      \Eng &
      \Eng &
      \Mar &
      \Mar \\
    \end{tikzcd}
  \end{equation}
\end{remark}
\begin{remark}
  The invariant measure on $\M_{i+1}$ is \emph{not} induced by the quotient of \cref{i:towerc}, since in \cref{sec:quotients} we only prescribed how to do so for compact or discrete group actions. Even if, for the specific action of $R_i$, it is is possible to induce directly an invariant measure on the quotient $\M_{i+1}=\M_{i}\slash R_i$ starting from one on $\M_i$, we chose not do so for clarity. Instead, the measure on $\M_{i+1}$ is given by the construction in \cref{sec:quotients} applied to the actions of \cref{i:towerd}.
\end{remark}
\begin{proof}
  \textbf{Step 1.} We first explain the general building block of the construction that will be repeated a finite number of times to obtain the final result. This step, in particular, describes the factorization \eqref{eq:factorization} for each fixed $i$. Let $\M=\G\slash\H$ be a Carnot homogeneous space, and let $\Htilde<\G$ be a closed, dilation-invariant subgroup with $\H\triangleleft\Htilde$. Define
  \begin{equation}
    \tilde{\M}:=\G\slash\Htilde,
  \end{equation}
  equipped with its metric structure $\sfd_{\tilde{\M}}$ (see \cref{sec:metstructure}). We have not yet chosen an invariant measure on $\tilde{\M}$. By \cref{prop:quotientofCarnothomogeneous}, the left multiplication of $\Htilde$ on $\G$ induces a proper and free action of the group $R:=\Htilde\slash\H$ on $\M$, and its quotient is identified as a metric space with $\tilde{\M}$. Denote the corresponding projection by $\pi:\M\to\tilde{\M}$.

  If $R$ is abelian, it admits a co-compact lattice $\Gamma$, i.e\ a discrete subgroup such that $R\slash\Gamma$ is compact. Fix such a lattice, and note that since $R$ is abelian, the subgroup $\Gamma \triangleleft R$ is normal, ensuring that $R\slash\Gamma$ is a connected, compact, abelian Lie group (a standard torus).

  The action of $R$ on $\M$ induces a proper and free action of the discrete subgroup $\Gamma$ on $(\M,\sfd_{\M},\mm_{\M})$ by metric measure isometries. Consider the quotient metric measure space $(\M\slash\Gamma,\sfd_{\M\slash\Gamma},\mm_{\M\slash\Gamma})$ as described in \cref{sec:quotients}, and let $\q:\M \to \M\slash\Gamma$ be the corresponding quotient map. We emphasize that this quotient is not a Carnot homogeneous space.

  Considering the compact abelian Lie group $R\slash\Gamma$, we observe that the action of $R$ on $\M$ by left multiplication induces an action of $R\slash\Gamma$ on $\M\slash\Gamma$, which we denote for clarity as follows:
  \begin{equation}\label{eq:groupactionalpha}
    \rho:R\slash\Gamma\times \M\slash\Gamma \to \M\slash\Gamma,\qquad \rho(\Gamma r, \Gamma m) := \Gamma r m.
  \end{equation}
  Note that \eqref{eq:groupactionalpha} is well-defined since $R$ is abelian. Since the action of $R$ on $\M$ is free, the induced action $\rho$ is also free (and, indeed, proper, as $R\slash\Gamma$ is compact).

  We claim that $\rho$ is an action by metric measure isometries on $(\M\slash\Gamma,\sfd_{\M\slash\Gamma},\mm_{\M\slash\Gamma})$. To verify the metric part, note that for all $r\in R$ and $m,m'\in \M$, we have
  \begin{align}
    \sfd_{\M\slash\Gamma}\Big(\rho(\Gamma r,\Gamma m),\rho(\Gamma r,\Gamma m')\Big) & = \sfd_{\M\slash\Gamma}\left(\Gamma rm,\Gamma rm'\right) & \text{by def.\ of group action \eqref{eq:groupactionalpha}} \\
    & = \inf_{\gamma\in\Gamma}\sfd_{\M}\left(rm,\gamma r m'\right) & \text{by def.\ of $\sfd_{\M\slash\Gamma}$}\\
    & = \inf_{\gamma\in\Gamma}\sfd_{\M}\left(rm,r \gamma m'\right) & \text{$R$ is abelian}\\
    & = \inf_{\gamma\in\Gamma}\sfd_{\M}\left(m, \gamma m'\right) & \text{$R$ acts by isometries on $\M$} \\
    & = \sfd_{\M\slash\Gamma}\left(\Gamma m, \Gamma m'\right).
  \end{align}
  For the measure part, recall from \cref{sec:quotients} that $\mm_{\M\slash\Gamma}$ is defined as the unique Borel measure on $\M\slash\Gamma$ such that, for any open set $U\subset \M$ on which the restriction $\q:U\to \q(U)$ is a metric isometry, the following holds
  \begin{equation}\label{eq:defofmmGgamma}
    \mm_{\M\slash\Gamma}|_{\q(U)} = \q_{\sharp}\left(\mm_{\M}|_U\right).
  \end{equation}
  Let $E \subset \M\slash\Gamma$ be a Borel set, which can always be expressed in the form $E = \q(F)$ for some Borel set $F \subset \M$. Assume, without loss of generality, that $E \subset \q(U)$ and $F \subset U$. From the above computation, if $\q: U \to \q(U)$ is a metric isometry, then for every $r \in R$, $\q: r U \to \q(r U)$ is also a metric isometry. It follows that, for any $r \in R$, we have
  \begin{align}
    \mm_{\M\slash\Gamma}\Big(\rho(\Gamma r,E)\Big) & = \mm_{\M\slash\Gamma}(\q(r F))\\
    & = \mm_{\M}\left(\q^{-1}(\q(rF))\cap rU \right)\\
    & = \mm_{\M}(r F\cap r U)\\
    & = \mm_{\M}(F\cap U)  & \text{the action of $R$ preserves $\mm_{\M}$}\\
    & =\mm_{\M \slash \Gamma}(E).
  \end{align}
  This concludes the proof of the claim: the compact abelian group $R\slash\Gamma$ acts freely and properly by metric measure isometries on the metric measure space $(\M\slash\Gamma, \sfd_{\M\slash\Gamma}, \mm_{\M\slash\Gamma})$. According to the construction in \cref{sec:quotients}, we denote the quotient metric measure space by
  \begin{equation}
    \Big((\M\slash\Gamma)\slash(R\slash\Gamma),\sfd_{(\M\slash\Gamma)\slash(R\slash\Gamma)},\mm_{(\M\slash\Gamma)\slash(R\slash\Gamma)}\Big),
  \end{equation}
  and by $\p:\M\slash\Gamma \to (\M\slash\Gamma)\slash(R\slash\Gamma)$ the corresponding quotient map.

  We need to verify that the result of this construction is (isometric to) the Carnot homogeneous space $\tilde{\M} = \G \slash \Htilde$ equipped with a choice of invariant measure. To do this, we first define a metric isometry $\psi$ such that the following diagram commutes
  \begin{equation}\label{eq:diagram}
    \begin{tikzcd}
      \M =\G\slash\H\arrow[rd, "\q"] \arrow[rr, "\pi"] & & \tilde{\M}=\G \slash \Htilde \arrow[d, dashed, "\psi"] \\
      & \M\slash\Gamma \arrow[r, "\p"] &  (\M\slash\Gamma)\slash(R\slash\Gamma)
    \end{tikzcd}
  \end{equation}
  Recall that $\pi(\H g) = \Htilde g$ (see \cref{prop:quotientofCarnothomogeneous}). The map $\psi$ is defined as follows:
  \begin{equation}
    \psi: \G \slash \Htilde \to (\M\slash\Gamma)\slash(R\slash\Gamma), \qquad \psi(\Htilde g):=  (R\slash\Gamma) (\Gamma (\H g)),\qquad \forall\, g\in \G.
  \end{equation}
  It is straightforward to check that $\psi$ is well-defined, is a smooth diffeomorphism, is surjective, and that the diagram \eqref{eq:diagram} commutes. We now show that it is a metric isometry. Indeed, for all $g, g' \in \G$, letting $m = \H g$ and $m' = \H g' \in \M$, we have
  \begin{align}
    \sfd_{(\M\slash\Gamma)\slash(R\slash\Gamma)}\Big( \psi(m),\psi(m')\Big)
    & =  \sfd_{(\M\slash\Gamma)\slash(R\slash\Gamma)}\Big( (R\slash\Gamma) \Gamma m,(R\slash\Gamma) \Gamma m' \Big) & \text{by def.\ of $\psi$}\\
    & = \inf_{ r \in R} \sfd_{\M\slash\Gamma}\Big(\rho(\Gamma r,\Gamma m),\Gamma m'\Big) & \text{by def.\ of $\sfd_{(\M\slash\Gamma)\slash(R\slash\Gamma)}$}\\
    & = \inf_{ r \in R} \sfd_{\M\slash\Gamma}\left(\Gamma r m,\Gamma m'\right) & \text{by \eqref{eq:groupactionalpha}}\\
    & = \inf_{ r \in R} \inf_{\gamma \in \Gamma} \sfd_{\M}(\gamma r m, m') & \text{by def.\ of $\sfd_{\G\slash\Gamma}$}\\
    & =\inf_{r\in R} \sfd_{\M}(r m,m') & \text{since $\Gamma < R$} \\
    & =\inf_{\tilde{h}\in \Htilde} \sfd_{\M}(\H \tilde{h} g,\H g') & \text{since $R=\Htilde\slash\H$} \\
    & =\inf_{\tilde{h}\in \Htilde} \inf_{h \in\H} \sfd_{\G}(h \tilde{h} g, g') & \text{by def.\ of $\sfd_{\M}$} \\
    & =\inf_{\tilde{h}\in \Htilde} \sfd_{\G}(\tilde{h} g, g') & \text{since $\H <\Htilde$} \\
    & = \sfd_{\M}(\Htilde g, \Htilde g'). & \text{by def.\ of $\sfd_{\M}$}
  \end{align}
  We claim that the Borel measure on $\tilde{\M}=\G \slash \Htilde$ defined by
  \begin{equation}\label{eq:mmGprime}
    \mm_{\tilde{\M}}:=\left(\psi^{-1}\right)_{\sharp} \mm_{(\M\slash\Gamma)\slash(R\slash\Gamma)}
  \end{equation}
  is invariant by the right action of $\G$ on $\tilde{\M}=\G \slash \Htilde$. Indeed, consider a Borel set $E\subset \G \slash \Htilde$ and $g\in\G$. Writing $E=\pi(F)$ for $F:=\pi^{-1}(E)$ and using \eqref{eq:diagram}, we obtain
  \begin{align}
    \mm_{\tilde{\M}}(E g) & = \mm_{(\M\slash\Gamma)\slash(R\slash\Gamma)}(\psi(E g)) \\
    & = \mm_{\M\slash\Gamma}(\p^{-1}(\psi (E g))) \\
    & = \mm_{\M\slash\Gamma}(\q(\pi^{-1}(E g))) \\
    & = \mm_{\M\slash\Gamma}(\q(F g)) = \mm_{\M\slash\Gamma}(\q(F) g),\label{eq:last_measure}
  \end{align}
  where, in the last equality, the multiplication on the right by $g \in \G$ denotes the well-defined right action of $\G$ on $\M \slash \Gamma$, induced by the action on $\M$ (see \cref{sec:homogstructure}), formally given by
  \begin{equation}\label{eq:thetaaction}
    \G\times \M\slash\Gamma \to \M\slash\Gamma,\qquad (g,\Gamma m)\mapsto \Gamma m g.
  \end{equation}
  An application of \cref{lem:technical_invariance} yields that\footnote{We apply \cref{lem:technical_invariance} to the case $(X,\sfd,\mm) = (\M,\sfd_{\M},\mm_{\M})$, with $G=\Gamma$ and its action of $\alpha$ given by the restriction to $\Gamma<R$ of the action \cref{prop:quotientofCarnothomogeneous}, and $K=\G$ acting on $\M$ from the right, with the action $\beta$ being the one from \cref{sec:homogstructure}.} $\mm_{\M\slash\Gamma}$ is invariant by the action \eqref{eq:thetaaction}. It follows from \eqref{eq:last_measure} that $\mm_{\tilde{\M}}(Eg) = \mm_{\tilde{\M}}(E)$, concluding the proof of the claim.

  Thus, $\mm_{\tilde{\M}}$ is an invariant measure on the Carnot homogeneous space $\tilde{\M}=\G \slash \Htilde$ such that
  \begin{equation}
    \psi:\Big((\M\slash\Gamma)\slash(R\slash\Gamma),\sfd_{(\M\slash\Gamma)\slash(R\slash\Gamma)},\mm_{(\M\slash\Gamma)\slash(R\slash\Gamma)}\Big) \to \Big(\tilde{\M},\sfd_{\tilde{\M}},\mm_{\tilde{\M}}\Big)
  \end{equation}
  is a metric measure isometry. With this choice of measure on $\tilde{\M}$, we have shown that the quotient map $\pi : \M \to \tilde{\M}$ can be factorized (up to the metric measure isometry $\psi$) via the composition of $\p$ and $\q$, as shown in the commutative diagram \eqref{eq:diagram}.

  \textbf{Step 2.} With the building block from the previous step at our disposal, we show how to use it to prove the result. In other words, we show how to choose the groups $\H_{i}$, for $i=0,\dots,s$, such that the previous step corresponds to the factorization \eqref{eq:factorization} for each fixed $i$.

  In the notation of the statement of the theorem, let $\h\subseteq\tilde{\h}\subseteq\g$ be the Lie algebras of $\H$ and $\Htilde$, respectively. They are both dilation-invariant so that they are stratified with
  \begin{align}
    \h = \h_{1}\oplus\dots\oplus\h_{s},\qquad \tilde{\h} = \tilde{\h}_1\oplus\dots\oplus\tilde{\h}_s,
  \end{align}
  where $\h_{j}\subseteq\tilde{\h}_j\subseteq \g_j$ for all $j=1,\dots,s$ and where some of the $\h_{j},\tilde{\h}_j$ can be trivial. Define
  \begin{equation}
    \H_i := \exp_{\G}\left(\h_{1}\oplus\dots\oplus \h_{s-i}\oplus \tilde{\h}_{s-i+1} \oplus\dots\oplus \tilde{\h}_s\right),\qquad \forall\, i=0,\dots,s.
  \end{equation}
  In other word we are enlarging, one stratum at the time, the Lie algebra of $\H=\H_0$ to the corresponding one of $\Htilde = \H_s$.

  To prove \cref{i:towera} we claim that the Lie algebra of $\H_i$ is an ideal in the one of $\H_{i+1}$ so that $\H_i\triangleleft \H_{i+1}$ for all $i=0,\dots,s-1$. To prove it, fix $i=0,\dots,s-1$, let $X \in \mathrm{Lie}(\H_{i+1})$ and decompose it as $X=A+B$ where
  \begin{equation}
    A \in \h_1\oplus\dots\oplus \h_{s-i-1},\qquad B \in \tilde{\h}_{s-i}\oplus\dots\oplus \tilde{\h}_s.
  \end{equation}
  Note that $A\in\mathrm{Lie}(\H_i)$, so that $[A,\mathrm{Lie}(\H_{i})]\subset \mathrm{Lie}(\H_i)$. On the other hand, by construction of $B$, the fact that $\mathrm{Lie}(\H_i)\subseteq \tilde{\h}$, and the stratification property it holds
  \begin{equation}
    [B,\mathrm{Lie}(\H_i)]\subseteq \left[\bigoplus_{j=0}^{i} \tilde{\h}_{s-i+j},\tilde{\h}\right] \subseteq \bigoplus_{j=1}^{i} \tilde{\h}_{s-i+j} \subseteq \mathrm{Lie}(\H_{i}).
  \end{equation}
  It follows that $[X,\mathrm{Lie}(\H_i)]  = [A+B,\mathrm{Lie}(\H_i)] \subset \mathrm{Lie}(\H_i)$, concluding the proof of the claim.

  Note that, if $\M,\tilde{\M}$ are Carnot groups, then $\H$ and $\tilde{\H}$ are normal in $\G$. It is easy to see that each $\H_i$ is normal in $\G$, so that each $\M_i$ is a Carnot group, thus explaining \cref{rmk:tower_Carnot_group}.

  There is nothing to prove for \cref{i:towerb}, as it simply provides the definition of the spaces $\M_i=\G\slash\H_i$. The choice of invariant measures is as given in \cref{i:towerd}.

  \Cref{i:towerc} then follows immediately from the definition of the spaces $\M_i$ and \cref{prop:quotientofCarnothomogeneous}.

  \Cref{i:towerd} is the result of the construction explained in \textbf{Step 1} of the proof. In particular, for each $\M_{i+1}$, we choose the invariant measure resulting from the factorization \eqref{eq:factorization}.
\end{proof}

%% file: cornucopia.tex
\section{\texorpdfstring{$\MCP$}{MCP} on Carnot homogeneous spaces}\label{sec:MCPforQuotientsofCarnotHomo}

We now state our main theorem, proving that the $\MCP$ descends to quotients of Carnot homogeneous spaces in the sense of \cref{prop:quotientofCarnothomogeneous}. 

\begin{theorem}[MCP for quotients of Carnot homogeneous spaces]\label{thm:MCPforQuotientsofCarnotHomo}
Suppose that a Carnot homogeneous space $\M=\G\slash\H$ admits a quotient to a Carnot homogeneous space $\tilde{\M}=\G\slash\tilde{\H}$. Assume that the minimizing Sard property holds for all factors $\M,\M_1,\dots,\M_{s-1}$ of the map $\pi:\M\to\tilde{\M}$ in \cref{thm:tower} (this is the case, for example, if the minimizing Sard property holds for all Carnot homogeneous spaces of step $\leq s$, where $s$ is the step of $\M$). Then, if $\M$ satisfies the $\MCP(K,N)$ for some $K \in \R$ and $N \in [1,\infty)$, the space $\tilde{\M}$ also satisfies the $\MCP(K,N)$.
\end{theorem}
\begin{proof}
    By \cref{thm:tower}, such a quotient $\pi:\M\to\tilde{\M}$ can be factorized into a finite number of compositions of local metric measure isometries (the $\q_i$'s) and quotients by compact groups of metric measure isometries (the $\p_i$'s), as illustrated by the diagram:
    \begin{equation}
        \begin{tikzcd}[sep=scriptsize]
            \M =\G\slash\H \arrow[rrrrrr,"\pi"] 
            & & &  & & & \tilde{\M} =\G\slash\tilde{\H} \\
            \M_0   \arrow[equal,u] \arrow[r,"\q_1"] & \M_0\slash\Gamma_0 \arrow[r,"\p_1"] & \M_{1} \arrow[r] & \dots  \arrow[r]  & \M_{s-1} \arrow[r,"\q_{s}"] &  \M_{s-1}\slash\Gamma_{s-1} \arrow[r,"\p_{s}"] & \M_{s} \arrow[equal,u]
        \end{tikzcd}
    \end{equation}
    %
    where each space appearing in the sequence is equipped with a suitable metric and measure making it a metric measure space, as specified in \cref{thm:tower}.

    For each $i = 0,\dots,s-1$, the space $\M_i$ is a Carnot homogeneous space of step $\leq s$ and, by assumption, satisfies the minimizing Sard property. Thus, by \cref{thm:Sardimpliesdeltaessnb}, $\M_i$ is $\delta$-essentially non-branching. Recall from \cref{thm:tower} that each map $\q_{i +1}:\M_i \to \M_i/\Gamma_i$ is the quotient by a discrete group $\Gamma_i$ of metric measure isometries of $\M_i$ acting properly and freely.  It follows from \cref{thm:quotient1} that if $\M_i$ satisfies the $\MCP(K,N)$ for some $K \in \R$ and $N \in [1,\infty)$, then the quotient space $\M_i/\Gamma_i$ satisfies the $\MCP_{\loc}(K,N)$ and is also $\delta$-essentially non-branching. Furthermore, by \cref{thm:tower} again, the map $\p_{i +1}:\M_i/\Gamma_i \to \M_{i+1}$ is the quotient of the metric measure space $\M_i/\Gamma_i$ by a connected, abelian, compact Lie group acting freely by metric measure isometries on $\M_i/\Gamma_i$. It follows from \cref{thm:quotient2} that $\M_{i+1}$ satisfies the $\MCP_{\loc}(K,N)$. By dilations, the $\MCP_{\loc}(K,N)$ for $\M_{i+1}$ is equivalent to $\MCP(K,N)$. Iterating this argument $s$ times, we reach the conclusion.
\end{proof}

\subsection{Failure of the measure contraction property}\label{sec:failure}

We postpone to \cref{sec:EngMart} the proof that the Martinet structure does not satisfy the $\MCP(K,N)$ for any $K\in\R$ and $N \in [1,\infty)$ (see \cref{thm:Martinet-noMCP}). Taking this result for granted, we apply the machinery developed so far to prove the failure of the $\MCP$ for Carnot groups admitting a quotient to the Martinet structure.

\begin{theorem}[No $\MCP$ with Martinet quotients]\label{thm:noMCP}
    Let $\G$ be a Carnot group of step $s\geq 3$ that admits a quotient to the Martinet structure (i.e., such that it satisfies one of the equivalent conditions of \cref{thm:quotienttoMartinet}). Assume that the minimizing Sard property holds for all Carnot groups of step $\leq s$. Then $\G$ does not satisfy the $\MCP(K,N)$ for any $K\in\R$ and $N \in [1,\infty)$.
\end{theorem}
\begin{proof}
By \cref{thm:quotienttoMartinet}, $\G$ admits a quotient to $\Eng$. By \cref{rmk:tower_Carnot_group}, the validity of the minimizing Sard property for Carnot groups of step $\leq s$ allows to apply \cref{thm:MCPforQuotientsofCarnotHomo}. We deduce that $\Eng$ satisfies the $\MCP(K,N)$. Since $\Eng$ admits a quotient onto $\Mar$, another application of \cref{thm:MCPforQuotientsofCarnotHomo} implies that $\Mar$ also satisfies the $\MCP(K,N)$. Note that for the quotient $\Eng\to\Mar$, as explained in \cref{rmk:repetitions}, the list of factors in \cref{thm:tower} is $\M_0=\M_1=\Eng$, $\M_2=\M_3=\Mar$, and all of them satisfy the minimizing Sard property, which again allows to apply \cref{thm:MCPforQuotientsofCarnotHomo}. This contradicts the fact that the Martinet structure does not satisfy any $\MCP(K,N)$, which will be	 proven in \cref{thm:Martinet-noMCP}.
\end{proof}

The minimizing Sard property stipulates that for any initial point $x$ of a sub-Riemannian manifold, the set of final points of abnormal geodesics starting from $x$ has zero measure. It is one of the main open problems in sub-Riemannian geometry, see \cite[Prob.\ 3]{A-openproblems}, \cite[Conj.\ 1]{RT-MorseSard}. However, it holds in several cases, which we outline now. Firstly, it is known to be true for all Carnot groups of step $\leq 3$. This is well-known for step $\leq 2$, while the step $3$ case is shown in \cite{LDMOPV-Sard}. Secondly, recall that \emph{filiform} Carnot groups are those whose Lie algebra satisfies
\begin{equation}
    \dim\g_1=2,\qquad \dim\g_2=\dots=\dim\g_s=1.
\end{equation}
In \cite{BNV-Sardfiliform}, it is proven that filiform Carnot groups satisfy the minimizing Sard property. Finally, in \cite{BV-Dynamical}, the minimizing Sard property is proven for Carnot groups of rank 2 and step 4. It is easy to check that these structures always admit a quotient to the Martinet structure. Thus, we obtain the following unconditional results.

\begin{corollary}[Examples of Carnot groups failing the $\MCP$]\label{cor:noMCPSardok}
    The following Carnot groups do not satisfy the $\MCP(K,N)$ for all $K\in\R$ and $N \in [1,\infty)$:
    \begin{itemize}
        \item Carnot groups of step $3$ that admit a quotient to the Martinet structure (in particular, the Engel group and all the free Carnot groups of step $3$);
        \item Filiform Carnot groups of step $s\geq 3$;
        \item Carnot groups of step $4$ and rank $2$.
    \end{itemize}
\end{corollary}

\subsection{Validity of the measure contraction property}\label{sec:validity}

Yet, many Carnot groups of step $\geq 3$ do satisfy the $\MCP$. In fact, we recall the following result, which is the state of the art regarding the validity of the $\MCP$ on Carnot groups.

\begin{theorem}[Badreddine, Rifford  \cite{BR-MCP}]\label{thm:BadRif}
    Any Carnot group whose distance function $\sfd_{\G}:\G\times \G\to \R$ is locally Lipschitz in charts outside of the diagonal satisfies the $\MCP(K,N)$ for all $K\leq 0$ and some $N\in [1,\infty)$.
\end{theorem}

\emph{Goh-Legendre geodesics} are a class of abnormal geodesics satisfying second-order necessary conditions for minimality, see \cref{a:Goh-Legendre}. If a sub-Riemannian structure admits no non-trivial Goh-Legendre abnormal geodesics, then all abnormal geodesics have infinite index \cite[Thm.\ 12.12]{nostrolibro}, and this implies that the corresponding distance function is locally Lipschitz in charts outside of the diagonal, see \cite[Thm.\ 12.11]{nostrolibro} or \cite[Thm.\ 5.7]{AAPL-OT}. We have the following.

\begin{corollary}\label{cor:goh-ideal}
    Any Carnot group with no non-trivial Goh-Legendre geodesics satisfies the $\MCP(K,N)$ for all $K\leq 0$ and some $N \in [1,\infty)$.
\end{corollary}

\subsection{Carnot groups in low dimensions}\label{sec:cornucopia}

Building on the previous results, we initiate the classification of Carnot groups that satisfy the $\MCP$. A stratified group is \emph{indecomposable} if its Lie algebra is not the direct sum of two non-trivial Lie algebras. A classification of indecomposable  stratified groups of dimension $\leq 7$ is provided in \cite{LDT-cornucopia,Gong}, to which we refer for details.

\begin{theorem}[A cornucopia of Carnot groups failing the $\MCP$]\label{thm:cornucopia}
  Among the indecomposable  stratified groups of dimension $\leq 7$, listed in \cref{tab:cornucopia}, equipped with any choice of Carnot-Carathéodory metric and invariant measure turning it into a Carnot group:
  \begin{enumerate}[(i)]
    \item\label{i:cornucopiared} the ones marked in red do not satisfy the $\MCP(K,N)$ for any $K\in \R$, $N\in [1,\infty)$ (an asterisk indicates this is conditional on the validity of the Sard property for that group);
    \item\label{i:cornucopiagreen} the ones marked in green do satisfy the $\MCP(K,N)$ for all $K\leq 0$ and some $N\in [1,\infty)$.
  \end{enumerate}
\end{theorem}
\begin{proof}
    \textbf{Proof of \cref{i:cornucopiared}.}
The criterion for admitting a Martinet quotient of \cref{thm:quotienttoMartinet} is checked using the Lie algebra relations listed in \cite{LDT-cornucopia} in terms of a specific basis. The groups marked in red in \cref{tab:cornucopia} are precisely all those that satisfy the criterion. A choice of subspaces $\h_3,\h_2$ that verify the criterion is given in \cref{tab:h2h3}, and the computations are omitted.
\begin{table}[ht]
\begin{tabular}[t]{|c|c|c|} \hline
    \multicolumn{3}{|c|}{Step 3}            \\ \hline \hline
    & $\h_2$     & $\h_3$    \\ \hline
    $N_{4,2}$ & $0$        & $0$       \\ \hline
    $N_{5,2,3}$    & $0$        & $X_5$     \\ \hline
    $N_{6,3,3}$    & $X_6$        & $0$     \\ \hline
    $N_{6,3,4}$    & $X_6$      & $0$       \\ \hline
    $357A$         & $X_6, X_7$ & $0$       \\ \hline
    $357B$         & $X_7$      & $X_6$     \\ \hline
    $257B$         & $X_7$      & $0$       \\ \hline
    $247A$         & $X_5$      & $X_7$     \\ \hline
    $247B$         & $X_5$      & $X_7$     \\ \hline
    $247C$         & $X_4$      & $X_6$     \\ \hline
    $247D$         & $X_5$      & $X_7$     \\ \hline
    $247E$         & $X_4-X_5$  & $X_7$     \\ \hline
    $247E_1$       & $X_5$      & $X_7$     \\ \hline
    $247F$         & $X_4-X_5$  & $X_6-X_7$ \\ \hline
    $247G$         & $X_4-X_5$  & $X_6-X_7$ \\ \hline
    $247I$         & $X_4$      & $X_6$     \\ \hline
    $247J$         & $X_4$      & $X_7$     \\ \hline
\end{tabular}
\begin{tabular}[t]{|c|c|c|} \hline
    \multicolumn{3}{|c|}{Step 4}   \\ \hline \hline
    & $\h_2$ & $\h_3$ \\ \hline
    $N_{5,2,1}$  & $0$    & 0      \\ \hline
    $N_{6,2,5}$  & $0$    & $X_5$  \\ \hline
    $N_{6,2,5a}$ & $0$    & $X_5$  \\ \hline
    $N_{6,2,7}$  & $0$    & $X_6$  \\ \hline
    $2457A$      & $X_7$  & $0$    \\ \hline
    $2457B$      & $X_6$  & $0$    \\ \hline
    $2457L$      & $0$    & $X_5$  \\ \hline
    $2457L_1$    & $0$    & $X_5$  \\ \hline
    $2457M$      & $0$    & $X_5$  \\ \hline
\end{tabular}
\begin{tabular}[t]{|c|c|c|} \hline
    \multicolumn{3}{|c|}{Step 5}  \\ \hline \hline
    & $\h_2$ & $\h_3$ \\ \hline
    $N_{6,2,1}$ & 0      & $0$    \\ \hline
    $N_{6,2,2}$ & 0      & $0$    \\ \hline
    $23457A$    & $0$    & $X_7$  \\ \hline
    $23457B$    & $0$    & $X_7$  \\ \hline
    $23457C$    & $0$    & $0$    \\ \hline
    $12457H$    & $0$    & $X_5$  \\ \hline
    $12457L$    & $0$    & $X_5$  \\ \hline
    $12457L_1$  & $0$    & $X_5$  \\ \hline
    \multicolumn{3}{c}{}          \\
    \multicolumn{3}{c}{}          \\
    \multicolumn{3}{c}{}          \\
    \multicolumn{3}{c}{}          \\
    \multicolumn{3}{c}{}          \\
    \multicolumn{3}{c}{}          \\ \hline
    \multicolumn{3}{|c|}{Step 6}  \\ \hline \hline
    & $\h_2$ & $\h_3$ \\ \hline
    $123457A$   & 0      & 0      \\ \hline
\end{tabular}
\vspace{1em}
\caption{Indecomposable  stratified groups of dimension $\leq 7$ admitting a quotient to Martinet, marked in red in \cref{tab:cornucopia}. The generators for a choice of the subspaces $\h_2,\h_3$ verifying \cref{thm:quotienttoMartinet} are given in terms of the bases of \cite{LDT-cornucopia}.}\label{tab:h2h3}
\end{table}

Choose a scalar product on the first layer giving these stratified groups the structure of a Carnot group. By \cref{thm:noMCP}, these Carnot groups -- provided that the minimizing Sard property holds -- do not satisfy any $\MCP(K,N)$. According to the discussion preceding \cref{cor:noMCPSardok}, the Carnot groups for which the minimizing Sard property is known to be true are those of step $3$, all filiform ones (which, in \cref{tab:cornucopia}, for step $\geq 4$, are $N_{5,2,1}$, $N_{6,2,1}$, and $123457A$), and Carnot groups of step $4$ and rank $2$ (which, in \cref{tab:cornucopia}, are $N_{6,2,5}$, $N_{6,2,5a}$, $N_{6,2,7}$, and $2457L$).

\textbf{Proof of \cref{i:cornucopiagreen}.} In Carnot groups of step $1$ and $2$, it is well-known that all geodesics are normal, so the minimizing Sard property holds and there are no non-trivial Goh-Legendre geodesics. We then apply \cref{cor:goh-ideal}. All other groups marked in green in \cref{tab:cornucopia} are of step $3$. One can directly verify that all of them have a medium-fat distribution, and thus they do not have non-trivial Goh(-Legendre) abnormal   horizontal curves (see \cite[Ex.\ 2.1]{riffordbook}). Therefore, we can also apply \cref{cor:goh-ideal} to them. The only exceptions are $N_{6,3,1}$, $247F_{1}$, $247P$ and $137A_1$, marked with a dagger in \cref{tab:cornucopia}: their distribution is not medium-fat, and they have Goh geodesics, see \cref{rmk:Legendre}. Nevertheless, we can verify that they do not have non-trivial Goh–\emph{Legendre} geodesics. For  stratified groups, the Goh–Legendre condition can be written in terms of the adjoint representation, and in step $3$, it admits the following simplified formulation (see \cref{a:Goh-Legendre}): there exists $\lambda_0 \in \g^*\setminus\{0\}$ such that
\begin{align}
\langle \lambda_0, X\rangle  & =0, \qquad \forall\, X\in\g_1, \tag{abnormal condition} \\
\langle\lambda_0,Y\rangle =\langle \lambda_0, [X_{u(t)},Y]\rangle & =0,\qquad \forall\, Y\in\g_2,\, a.e.\ t\in[0,1],  \tag{Goh condition}  \\
\langle \lambda_0, [[X_{u(t)},X],X]\rangle & \geq 0,\qquad \forall\, X\in\g_1,\, a.e.\ t\in [0,1],  \tag{generalized Legendre condition}
\end{align}
where $\langle \cdot,\cdot\rangle$ denotes the duality pairing, $X_{u(t)} = \sum_{i=1}^k u_i(t)X_i$, and $u(t)$ is the control of the geodesic. Using the explicit Lie algebra structure from \cite{LDT-cornucopia} we check that the Goh-Legendre condition is false for an arbitrary non-trivial horizontal curve. We omit the computations.
\end{proof}

\begin{remark}[Goh vs Legendre]\label{rmk:Legendre}
The Goh condition alone is not sufficient in the proof. In fact, all groups marked with a dagger in \cref{tab:cornucopia} do have a non-trivial Goh horizontal curves. Such  curves, however, do not satisfy the generalized Legendre condition. For $N_{6,3,1}$, $247F_{1}$, and $247P$, such curves have control $u = (u_1,0,0)$. For $137A_{1}$, any horizontal curve with control $u = (0,0,u_3,u_4)$ satisfies the Goh condition. Taking these controls to be constant corresponds to segments in the first stratum, that are geodesics for any choice of sub-Riemannian metric.
\end{remark}

We conclude with a straightforward result concerning the decomposable case.

\begin{corollary}[MCP for decomposable Carnot groups]\label{thm:decomposable}
Let $\G$ be a Carnot group that admits a decomposition as a direct product of Carnot groups
\begin{equation}
\G = \G_1\times\dots\times\G_n,
\end{equation}
with product metric $\sfd_{\G}^2 = \sum_{j=1}^n \sfd_{\G_j}^2$, and product measure $\mm_{\G} = \otimes_{i=j}^n \mm_{\G_j}$. Then
\begin{enumerate}[(i)]
\item \label{i:yes} if each factor $\G_i$ satisfies the $\MCP(K_i,N_i)$ for some $K_i\in \R$ and $N_i\in [1,\infty)$, then $\G$ satisfies the $\MCP(K,N)$ with $K=\min_i K_i$ and $N=\sum_i N_i$;
\item \label{i:no} if one of the factors $\G_i$ does not satisfy the $\MCP(K,N)$ for some $K\in \R$ and $N \in [1,\infty)$, then $\G$ does not satisfy the $\MCP(K,N)$ (assuming that the minimizing Sard property holds for all Carnot groups of step smaller or equal than the step of $\G$).
\end{enumerate}
\end{corollary}
\begin{proof}
\cref{i:yes} follows from the well-known tensorization property of the $\MCP$, see \cite{O-products}. To prove \cref{i:no}, note that $\G$ admits a quotient to $\G_i$, and one can use \cref{thm:MCPforQuotientsofCarnotHomo}.
\end{proof}

%% file: generic.tex
\section{Failure of the \texorpdfstring{$\MCP$}{MCP} for generic structures}\label{sec:genericity}

Purpose of this section is to prove \cref{thm:generic-no-MCP-intro}, which we recall here.

\begin{theorem}[Generic failure of the $\MCP$ for rank $3$ and high dimension]\label{thm:generic-no-MCP}
 Assume that the minimizing Sard property holds for  Carnot groups. In the same setting of \cref{{gen:genideal}}, if
  \begin{equation}\label{eq:lowerbound}
    \dim M \geq (k-1)\left(\frac{k^2}{3}+\frac{5k}{6}+1\right),
  \end{equation}
  then there exists an open dense subset $W_k'\subset W_k$ of $\mathcal{G}_k$ such that for every element of $W_k'$, the corresponding sub-Riemannian structure does not admit non-trivial abnormal geodesics and does not satisfy the $\MCP(K,N)$ for any $K\in \R$, $N\in [1,\infty)$ and any choice of smooth measure.
\end{theorem}
\begin{proof}
Remember the set $W_k$ from \cref{gen:genideal}. By \cite[Lemma 4]{AG-subanal}, if \eqref{eq:lowerbound} holds, then the growth vector at $x$, i.e. $(\dim\distr_x^1,\dim\distr_x^2,\dots,\dim\distr_x^{s(x)})$, is generically maximal, and the tangent cone of $(M,\sfd)$ at $x$ admits a quotient to the Martinet structure. This means that, up to taking a smaller open dense set $W_k'\subset W_k$, for any pair $(\distr,g) \in W_k'$ we can apply \cref{thm:noMCP}.
\end{proof}

We also prove \cref{thm:no-MCP-rank2-intro}, whose statement we recall here.

\begin{theorem}[Failure of the $\MCP$ for rank $2$]\label{thm:no-MCP-rank2}
Let $(M,\sfd,\mm)$ be a sub-Riemannian metric measure space with constant rank $k=2$. Assume that there exists $x\in M$ such that 
  \begin{equation}
  \dim \distr^3_x - \dim \distr^2_x \geq 1.
  \end{equation}
  Assume that the minimizing Sard property holds for all Carnot groups of step $\leq s(x)$. Then $(M,\sfd,\mm)$ does not satisfy the $\MCP(K,N)$ for any $K\in \R$ and $N\in [1,\infty)$.
  \end{theorem}
\begin{proof}
Under these assumptions, the growth vector at $x$ is either $(2,2,3,\dots)$, $(2,2,4,\dots)$, $(2,3,4,\dots)$, or $(2,3,5,\dots)$. In each case the tangent cone of $(M,\sfd)$ at $x$ admits a quotient to the Martinet structure (see also the proof of \cite[Prop.\ 7]{AG-subanal}). We conclude using \cref{thm:noMCP}.
\end{proof}

%% file: martinet-engel.tex
\section{Failure of the \texorpdfstring{$\MCP$}{MCP} for the Martinet structure}\label{sec:EngMart}

The purpose of this section is to disprove the $\MCP$ in the Martinet structure. We refer to \cref{a:SR} for basic notions in sub-Riemannian geometry and in particular their geodesics.

The Martinet structure was defined as a Carnot homogeneous space in \cref{def:EngMar} as the quotient $\pi:\Eng\to \Mar = \Eng\slash\H$ of the Engel group by the subgroup $\H =\exp(\e_2)$. Let $X_1,X_2,Y,Z$ be the basis given in \cref{def:EngMar}, and consider the map $\phi:\R^3\to \Mar$
\begin{equation}
  \phi(x_1,x_2,z):=\H \exp(z Z)\exp(x_2 X_2)\exp(x_1 X_1),
\end{equation}
where $\exp:\e\to \Eng$ is the group exponential map. By \cite[Thm.\ 1.2.12]{CG-NilpotentGroupsBook} the map $\phi$ is an analytic diffeomorphism giving global coordinates on $\Mar$ and such that, in these coordinates, the Lebesgue measure is right-invariant by the action of $\Eng$ on $\Mar$. Via routine computations, in these coordinates, the sub-Riemannian structure on $\Mar$ is generated by the vector fields
\begin{equation}
  \pi_* X_1 = \partial_{x_1},\qquad \pi_*X_2 := \partial_{x_2} + \frac{x_1^2}{2} \partial_z,
\end{equation}
In order to stay consistent with \cite{martinetagrachev}, which we will use in the following, we relabel $x_1 = y$ and $x_2 = x$. With this relabeling the Hamiltonian of this sub-Riemannian structure on $\R^3$ is
\begin{equation}
  H(\lambda) = \frac{1}{2} \left[ \left( u + \frac{y^2}{2} w \right)^2 + v^2 \right], \quad \text{where} \quad \lambda = \left(x, y, z, u \di x + v \di y + w \di z\right) \in T^*\R^3.
\end{equation}
Solutions $\lambda : [0,1]\to T^*\R^3$  to the Hamilton's equation are called extremal and (normal) geodesics are projections $\pi:T^*\R^3\to \R^3$ of extremals. The corresponding Hamilton's equations are
\begin{align}
  \dot{x} & = u + \frac{y^2}{2} w,                              & \dot{u} & = 0,                                        \\
  \dot{y} & = v,                                                & \dot{v} & = - \left( u + \frac{y^2}{2} w \right) y w, \\
  \dot{z} & =  \left( u + \frac{y^2}{2} w \right)\frac{y^2}{2}, & \dot{w} & = 0.
\end{align}
We will only need to consider geodesics starting from the origin and thus always set the initial conditions $x_0 = y_0 = z_0 = 0$. As usual, we must have $(u_0, v_0) \neq 0$ in order to have a non-trivial curve and we can introduce cylindrical coordinates on each fibre of the cotangent bundle by setting $r > 0$, and $\theta \in \mathbb{S}^1$ the unique parameters such that
\begin{equation}
  u_0 = r \cos(\theta),\qquad  v_0 = r \sin(\theta).
\end{equation}
We observe that $u(t) = u_0$ and $w(t) = w_0$ are constant functions and that $y(t)$ satisfies the second-order differential equation
\begin{equation}
  \label{duffing}
  \ddot{y} + u_0 w_0 y + \frac{w_0^2}{2} y^3 = 0, \qquad y(0) = 0, \qquad \dot{y}(0) = v_0.
\end{equation}
Equation \eqref{duffing} can be solved by means of Jacobi elliptic functions. We refer to \cite[Ch.\ 16 and 17]{AS-Handbook} for notation and conventions we use throughout. Set $y(t) = A \JacobiCN{\omega t + \phi}{m}$. We explain how the parameters $A, \omega, \phi$ and $m$ are found. Putting the ansatz into \eqref{duffing} gives
\begin{equation}
  (1 - 2m) \omega^2 = u_0 w_0, \qquad A^2 w_0^2 = 4 m \omega^2.
\end{equation}

From the initial conditions $A \JacobiCN(\phi \mid m) = 0$ and $-A \omega \JacobiSN(\phi \mid m) \JacobiDN(\phi \mid m) = v_0$, we deduce that $m \in [0, 1)$, $\phi = \JacobiK(m)$ and $-A \omega \sqrt{1 - m} = v_0$, or $\phi = 3 \JacobiK(m)$ and $A \omega \sqrt{1 - m} = v_0$. With a simple algebraic simplification, we obtain
\begin{equation}\label{eq:correspondence}
  m = \frac{r - \mathrm{sgn}(w_0) u_0}{2 r} = \frac{1 - \mathrm{sgn}(w_0) \cos(\theta)}{2}, \quad \omega^2 = r |w_0|, \quad A^2 \omega^2 = 4 m r^2.
\end{equation}
To make the correspondence between the variables $(r, \theta, w_0) \mapsto (A, \omega, m)$ given by \eqref{eq:correspondence} one-to-one and smooth, we restrict it to the range:
\begin{equation}
  (0, + \infty) \times (\mathbb{S}^1 \setminus \{(\pm 1, 0)\}) \times (\mathbb{R} \setminus \{0\}) \to (\mathbb{R} \setminus \{0\})^2 \times (0, 1)
\end{equation}
and choose $\phi = \JacobiK(m)$, $\mathrm{sgn}(\omega) = \mathrm{sgn}(w_0)$ and $\mathrm{sgn}(A \omega) = - \mathrm{sgn}(v_0)$.

Once we have found an equation for $y(t)$, all the other unknown functions can be derived by integration. Indeed, it is shown in \cite[Section 4.5]{martinetagrachev} that
\begin{align}
  v(t) & = - A \omega \JacobiSN{\omega t + \phi}{m}\JacobiDN{\omega t + \phi}{m},                                                       \\
  x(t) & = \mathrm{sgn}(w_0)\left(- r t + \frac{2 r}{\omega} \Big(\JacobiE(\omega t + \phi \mid m) - \JacobiE(\phi\mid m)\Big) \right), \\
  z(t) & = \frac{1}{3 w_0} \left( r^2 t - u_0 x(t) - y(t) v(t) \right).
\end{align}
The geodesic flow is an analytic function of the initial data, and so it is enough to take the limit $w_0 \to 0$ (resp. $v_0 \to 0$) in the above equations to obtain the cases $w_0 = 0$ (resp. $v_0 = 0$).

\begin{remark}
  How the extremal $\lambda(t)$ behave at the limit $v_0 \to 0$ can be different depending on the sign of $u_0$ and $w_0$. Indeed, if the signs of $u_0$ and $w_0$ are such that $v_0 \to 0$ corresponds to taking $m \to 0$, then the elliptic functions reduce to standard trigonometric functions. If, however, they are such that $v_0 \to 0$ implies $m \to 1$, hyperbolic functions appear instead. This is the origin of the asymmetric cusp in Martinet spheres, see \cite[Fig.\ 13.18, 13.19]{nostrolibro}.
\end{remark}

\begin{definition}
  For $t \in [0,1]$, let $\exp_0^t:=\exp_0(t\cdot):T_0^*\R^3\to\R^3$ be the sub-Riemannian exponential map at time $t$. The Jacobian $\mathcal{J}^t$ is the is $3$-form on $T^*_0 \R^3$ obtained by pulling back the Lebesgue volume through $\exp_0^t$, i.e.
  \begin{equation}
    \mathcal{J}^{t}(\lambda_0) := \left.\left(\left(\exp_0^t\right)^*\di x\wedge\di y\wedge \di z\right)\right|_{\lambda_0}, \qquad \forall\, \lambda_0 \in T^*_0\R^3.
  \end{equation}
\end{definition}

The expression for $\J^t$ was obtained in \cite{martinetagrachev}, see especially \cite[Sec.\ 3.1]{martinetagrachev} and the computations in \cite[Sec.\ 4.6]{martinetagrachev} leading to the equations (4.19)--(4.22) there.

\begin{proposition}
  \label{JacobianDetMartinet}
  Let $\lambda_0 \in T^*_0\R^3$ with $m\in (0,1)$ and $\omega \neq 0$ (corresponding to $v_0 \neq 0$ and $w_0 \neq 0$). Then for all $t\in [0,1]$ we have
  \begin{equation}
    \label{detMartinet2}
    \J^t(\lambda_0) = -\frac{r^4 t}{\omega^2} \mathrm{sgn}(v_0) \J_R\left(|\omega| t, m\right) r \di r \wedge  \di \theta \wedge \di w|_{\lambda_0},
  \end{equation}
  where $\J_R$ is the \emph{reduced Jacobian}, defined by
  \begin{equation}
    \label{eq:JR}
    \J_R(\xi, m):=  \frac{\xi^2 c_1(\xi, m) + \xi c_2(\xi, m) + c_3(\xi, m)}{\JacobiDN(\xi \mid m)}
  \end{equation}
  where we set
  \begin{align}
    c_1(\xi) & := (1-m)\frac{\JacobiCN(\xi \mid m)}{\JacobiDN(\xi \mid m)},                                                                \\
    c_2(\xi) & := (1 - m)\JacobiSN(\xi \mid m) - 2 (1 - m) \JacobiE(\xi\mid m) \frac{\JacobiCN(\xi \mid m)}{\JacobiDN(\xi \mid m)},        \\
    c_3(\xi) & := \JacobiE(\xi \mid m)^2 \frac{\JacobiCN(\xi \mid m)}{\JacobiDN(\xi \mid m)} - \JacobiE(\xi \mid m) \JacobiSN(\xi \mid m).
  \end{align}
\end{proposition}

The next statement combines the results from \cite[Thms.\ 4.7, 4.10]{martinetagrachev}.

\begin{theorem}
  \label{theorem:cutlocusMartinet}
  Let $\lambda_0 \in T^*_0\R^3$ with $m\in (0,1)$ and $\omega \neq 0$ (corresponding to $v_0 \neq 0$ and $w_0 \neq 0$). The first conjugate time along the curve $\gamma(t)=\exp_0(t\lambda_0)$, $t\in\R$, denoted by $t_{\mathrm{conj}}(\lambda_0)$, satisfies
  \begin{equation}
    2 \EllipticK(m) < |\omega| t_{\mathrm{conj}}(\lambda_0) < 3 \EllipticK(m),
  \end{equation}
  and $\mathcal{J}_R(|\omega| t, m) < 0$ for all $t \in (0, t_{\mathrm{conj}}(\lambda_0))$. Furthermore, the cut time is given by
  \begin{equation}
    t_{\mathrm{cut}}(\lambda_0) = \frac{2 \EllipticK(m)}{|\omega|},
  \end{equation}
  corresponding to the first time $\gamma$ intersects the Martinet surface $\{y = 0\}$.
\end{theorem}


\begin{proposition}
  \label{diffcharMCP}
  Should the Martinet structure satisfies the $\MCP(0, N)$ condition for some $N \in [1,\infty)$, the inequality
    \begin{equation}
        \label{MCPJacobian}
        |\J_R(\omega t, m)| \geq t^{N - 1} |\J_R(\omega, m)|
      \end{equation}
  will hold for all $t \in [0,1]$, all $m \in (0, 1)$ and all $\omega \in (0, 2 \EllipticK(m))$.
\end{proposition}

\begin{proof}
  By \cref{theorem:cutlocusMartinet} if $\lambda_0\in T^*_0\R^3$ is such that $m\in (0,1)$ and $\omega \in (0, 2 \EllipticK(m))$, then $\gamma:[0,1]\to \R^3$ defined by $\exp_0(t\lambda_0)$ is a geodesic, $\lambda_0$ is not a critical point of $\exp_0$, and $\exp_0(\lambda_0)\notin \Cut(0)$. Then, there is a neighborhood $\mathcal{U}$ of $\lambda_0$ such that $\exp_0 : \mathcal{U} \to \exp_0(\mathcal{U})$ is a diffeomorphism, and such that for all $\lambda_0' \in \mathcal{U}$, the geodesic $\gamma(t) := \exp_0(t \lambda_0')$ is the unique geodesic joining $0$ with $\exp_0(\lambda_0')$ and $\exp_0(\lambda_0')$ is not conjugate to $0$ along $\gamma$. Assume that $\varepsilon > 0$ is sufficiently small such that $B_\varepsilon(\exp_0(\lambda_0)) \subseteq \exp_0(\mathcal{U})$ and let $A_{\varepsilon}$ be such that $\exp_0(A_\varepsilon) = B_\varepsilon(\exp_0(\lambda_0))$. In particular by the $\MCP(0, N)$ and \cref{eq:MCPtovolumeineq} in particular, we have that for all $t\in  [0,1]$
  \begin{equation}
    \label{eq:MCPtoJac}
    \left\lvert\frac{\mathcal{J}^t(\lambda_0)}{\mathcal{J}^1(\lambda_0)}\right\rvert = \lim_{\varepsilon\to 0} \frac{\int_{A_{\varepsilon}} \left.\left(\left(\exp_0^t\right)^*\di x \wedge\di y\wedge \di z\right)\right|_{\lambda_0}}{\int_{A_{\varepsilon}} \left.\left(\left(\exp_0^1\right)^*\di x \wedge\di y\wedge \di z\right)\right|_{\lambda_0}}  = \lim_{\varepsilon\to 0}\frac{\mathscr{L}(Z_t(0, B_{\varepsilon}(\exp_0(\lambda_0))))}{\mathscr{L}(B_{\varepsilon}(\exp_0(\lambda_0)))} \geq t^N,
  \end{equation}
  where $\mathscr{L}$ is the Lebesgue measure on $T_0^*\R^3\simeq \R^3$. We deduce \eqref{MCPJacobian} with the expression \eqref{detMartinet2}.
\end{proof}

The failure of the measure contraction property can finally be obtained by showing that \eqref{MCPJacobian}
does not hold in the Martinet structure.

\begin{theorem}[Failure of the $\MCP$ for the Martinet structure]\label{thm:Martinet-noMCP}
  The Martinet structure does not satisfy the $\MCP(K, N)$ for any $K \in \mathbb{R}$ and any $N \in [1,\infty)$.
\end{theorem}
\begin{proof}
  Since the Martinet structure admits dilations, it is enough to disprove $\MCP(0, N)$. We show that we get a contradiction in \eqref{MCPJacobian} by writing an expansion of $\J_R$ around $m = 1$ and letting $\omega$ tend to $+\infty$. The following expansion can be found in \cite[16.15.1--3]{AS-Handbook}:
      \begin{align}
        \JacobiSN(\xi \mid m) & = \tanh(\xi) + \frac{1}{4} (m - 1) \left(\sinh(\xi) \cosh(\xi) - \xi\right) \sech(\xi)^2 + O\big((m - 1)^2\big) ;                        \\
        \JacobiCN(\xi \mid m) & = \operatorname{sech}(\xi) - \frac{1}{4} (m - 1) (\sinh(\xi) \cosh(\xi) - \xi) \tanh(\xi) \sech(\xi) + O\big((m - 1)^2\big) ; \\
        \JacobiDN(\xi \mid m) & = \sech(\xi) + \frac{1}{4} (m - 1) (\sinh(\xi) \cosh(\xi) + \xi) \tanh(\xi) \sech(\xi) + O\big((m - 1)^2\big).
      \end{align}
      Recalling from \cite[17.2.10]{AS-Handbook} that $\JacobiE(\xi \mid m) := \int_0^\xi \JacobiDN(u \mid m)^2 \di u$, we also obtain
      \[
        \JacobiE(\xi \mid m) = \tanh(\xi) - \frac{1}{4} (m - 1) \Big[\xi + \tanh(\xi) (\xi \tanh(\xi) - 1)\Big] + O\big((m - 1)^2\big).
      \]
      Plugging these formulas into \eqref{eq:JR} gives
      \[
        \J_R(\xi, m) := \frac{1}{2} (m - 1) \cosh(\xi) \Big[ \sinh(\xi)^2 + \xi \left( \tanh(\xi) - 2 \xi \right) \Big] + O\big((m - 1)^2\big).
      \]
      Therefore, we get, for all $t\in (0, 1)$, that
      \begin{equation}
        \label{eq:ratioJac0}
        \lim_{\omega \to +\infty} \lim_{m \to 1^{-}} \left\lvert\frac{\J_R(\omega t, m)}{\J_R(\omega, m)}\right\rvert = \lim_{\omega \to \infty} \frac{\cosh(\omega t)}{\cosh(\omega)} \frac{\sinh^2(\omega t) + \omega t \left( \tanh(\omega t) - 2 \omega t  \right) }{ \sinh^2(\omega) + \omega \left( \tanh(\omega) - 2 \omega \right)}
        = 0.
      \end{equation}
      This implies by \cref{diffcharMCP} that the Martinet structure does not satisfy the $\MCP(0, N)$.
\end{proof}
\begin{remark}[Violation of weaker conditions]
  \label{remark:noQCDinMar}
  Suppose that there is a function $f : [0, 1] \to [0, +\infty)$ such that for all $q \in \Mar$ and all Borel set $A \subset \Mar$, with $0<\mm(A)<\infty$ it holds
  \begin{equation}
    \label{eq:moregeneralthanMCP}
    \mm(Z_t(q,A)) \geq f(t) \mm(A),\qquad \forall\, t\in [0, 1].
  \end{equation}
The same argument leading to \eqref{eq:MCPtoJac} yields that for all $t\in [0, 1]$, it holds
  \[
    t |\J_R(\omega t, m)| \geq f(t) |\J_R(\omega, m)|,\qquad \forall\, m \in (0, 1)\, , \ \forall\, \omega \in (0, 2 \EllipticK(m)).
  \]
  The computation \eqref{eq:ratioJac0} shows that $f(t)\equiv 0$ for all $t\in [0,1)$. The Martinet structure is therefore \emph{not} qualitatively non-degenerate in the sense of \cite[Assumption 1]{CavallettiHuesmann2015}.
\end{remark}

%% file: variants.tex
\section{Variants and related curvature conditions}\label{sec:variants}

In this paper, we work with Ohta's $\MCP$, the weakest among the variants of the measure contraction property. A variant in terms of the Rényi entropy was used by Cavalletti and Mondino in \cite{CM-OptMaps} (see also \cite[Def.\ 6.8]{CM-globalization}, where it is denoted by $\MCP_{\varepsilon}$).
One can deduce that $\MCP_{\varepsilon}$ implies Ohta's $\MCP$, see \cite[Lemma 6.13]{CM-globalization}. Since we do not assume the essentially non-branching condition, the converse implication might not be true.

Furthermore, for $Q \in [1,\infty)$, one can define a relaxed version of the $\MCP$ (resp.\ $\MCP_{\varepsilon}$) by multiplying the inequality \eqref{eq:O-MCP} in \cref{def:MCP} (resp.\ \cite[(6.4) in Def.\ 6.8]{CM-globalization}) by a factor $1/Q$. The case $Q = 1$ corresponds to Ohta's $\MCP$ and the $\MCP_{\varepsilon}$ variant.

With minor modifications in the proofs, all the results concerning stability under quotients that we proved for $\MCP$ also hold for the $Q$-relaxed variants and the corresponding $\MCP_{\varepsilon}$. Arguing without loss of generality in the case $K = 0$, any such variant easily implies that
\begin{equation}\label{eq:false}
    \mm(Z_t(x,A))\geq \frac{t^N}{Q} \mm(A),\qquad\forall \,t\in [0,1],
\end{equation}
for the $t$-intermediate set  between $x\in X$ and a Borel set $A$ with $0<\mm(A)<+\infty$. Condition \eqref{eq:false} cannot be satisfied by the Martinet structure; see \cref{remark:noQCDinMar}.

We conclude that any sub-Riemannian metric measure space which, as a consequence of \cref{thm:noMCP-intro}, does not satisfy the $\MCP$, also fails to satisfy the $\MCP_{\varepsilon}$ and the corresponding $Q$-relaxed variants. This includes, of course, all red Carnot groups appearing in \cref{tab:cornucopia}.

Finally, one can verify with an approximation argument that the \emph{Quasi Curvature-Dimension} condition $\mathrm{QCD}(Q,K,N)$ introduced in \cite{M-QCD}, for $Q \in [1,\infty)$, implies the $Q$-relaxed version of the $\MCP_{\varepsilon}(K,N)$ mentioned above. Thus, again under the conditions of \cref{thm:noMCP-intro}, Milman's $\mathrm{QCD}$ condition cannot be satisfied either.

%% file: srgeo.tex
\section{Sub-Riemannian geometry}\label{a:SR}

We collect here basic facts and notations in sub-Riemannian geometry used in this paper. For a comprehensive introduction, we refer to \cite{nostrolibro,riffordbook,montgomerybook}.

\begin{itemize}[left=0em]
\item A sub-Rieman\-nian structure on a smooth, connected $n$-dimensional manifold $M$, where $n\geq 2$, is defined by a set of $L\geq 2$ global smooth vector fields $X_{1},\ldots,X_{L} \in \mathrm{Vec}(M)$, called a \emph{generating family}. We assume the \emph{bracket-generating} condition, i.e., the vector fields $X_{1},\ldots,X_{L}$ and their iterated Lie brackets at $x$ generate the tangent space $T_x M$, for all $x\in M$.
\item The (generalized) \emph{distribution} is the disjoint union $\distr = \bigsqcup_{x\in M} \distr_x$, where 
\begin{equation}
\distr_{x}:=\mathrm{span}\{X_{1}(x),\ldots,X_{L}(x)\}\subseteq T_{x}M,\qquad \forall\, x\in M.
\end{equation}
Notice that $\distr$ is a vector bundle if and only if $\dim\distr_x$, called the \emph{rank at $x$}, does not depend on $x$. In this case we say that the sub-Riemannian structure has constant rank.

\item The generating family induces an inner product $g_{x}$ on $\distr_{x}$ given by:
\begin{equation}
g_{x}(v,v):=\inf\left\{\sum_{i=1}^{L}u_{i}^{2}\,\Big|\,  v=\sum_{i=1}^{L}u_{i}X_{i}(x)\right\},\qquad \forall\, v\in \distr_x.
\end{equation}

\item For $i\geq 1$, we define the \emph{iterated distributions} $\distr^i=\bigsqcup_{x\in M}\distr^i_x$, where:
\[
\distr^{i}_{x}:=\mathrm{span}\{[X_{i_{1}},[\ldots,[X_{i_{j-1}},X_{i_{j}}]]\mid i_{k}\in \{1,\ldots,L\},\; j\leq i\}.
\]
The \emph{step} of the distribution at $x$ is the minimal $s=s(x)$ such that $\distr^{s}_{x}=T_{x}M$. The distribution is \emph{equiregular} at $x$ if $\distr^i_x$ is locally constant in a neighborhood of $x$, for all $j\geq 1$. We also say that $x$ is a \emph{regular point}.

\item A \emph{horizontal curve} is an absolutely continuous (in charts) map $\gamma : [0,1] \to M$ such that there exists $u\in L^{2}([0,1],
\R^{L})$, called \emph{control}, satisfying
\begin{equation}\label{eq:admissible}
\dot\gamma_t =  \sum_{i=1}^L u_i(t) X_i(\gamma_t), \qquad \mathrm{a.e. }\; t \in [0,1].
\end{equation}
The class of horizontal curves depends on the family $\mathscr{F}=\{X_1,\dots,X_L\}$ only through the $C^{\infty}(M)$-module of vector fields generated by $\mathscr{F}$.

\item For any control $u\in L^2([0,1],\R^L)$, let $X_{u(t)}:=\sum_{i=1}^L u_i(t)X_i$ be a time-dependent vector field. Denote by $P_{t_0,t_1}^{u}:M\to M$ the flow of $X_u$ with initial datum at time $t_0$ and final time $t_1$, provided that it is well-defined. Note that it holds
\begin{equation}
P_{t_0,t_0}^{u} = \mathrm{id}_M,\qquad P_{t_0,t_1}^{u}= P_{t_1,t_2}^{u} \circ P_{t_0,t_1}^{u}
\end{equation}
as maps on $M$, for all $t_0,t_1,t_2$ such that the flows exist. Note that if $\gamma$ is a horizontal curve with control $u$, and $\gamma_0=x$, then $\gamma_t = P_{0,t}^{u}(x)$.

\item We define the \emph{length} of a horizontal curve $\gamma:[0,1]\to M$ as follows:
\begin{equation}
\ell(\gamma) := \int_0^1 \sqrt{g(\dot\gamma_t,\dot\gamma_t)}\,\di t.
\end{equation}
The length $\ell$ is invariant by suitable reparametrizations. Every horizontal curve is the reparametrization of one with constant speed.
\item The \emph{sub-Rieman\-nian (or Carnot-Carathéodory) distance} is defined by:
\begin{equation}\label{eq:infimo}
\sfd(x,y) = \inf\{\ell(\gamma)\mid \gamma_0 = x,\, \gamma_1 = y,\quad \gamma \text{ is horizontal} \}.
\end{equation}
The bracket-generating condition implies that $\sfd$ is finite and continuous. If $(M,\sfd)$ is complete as metric space, then for any $x,y \in M$ the infimum in \eqref{eq:infimo} is attained.

\item
On the space of horizontal curves with fixed endpoints, and parametrized with constant speed, the minimizers of the length functional coincide with the minimizers of the \emph{energy functional} 
\begin{equation}
J(\gamma) := \frac{1}{2}\int_0^1 g(\dot\gamma(t),\dot\gamma(t))\, \di t.
\end{equation}

\item 
The \emph{end-point map} of the family $\mathscr{F}=\{X_{1},\ldots,X_{L}\}$ and with base point $x$ is the Fréchet-smooth map $\End^{\mathscr{F}}_{x}: \mathcal{U} \to M$, sending $u$ to the point a time $t=1$ of the solution of
\begin{equation}\label{eq:endpoints}
\dot\gamma_t = \sum_{i=1}^L u_i(t) X_i(\gamma_t), \qquad \gamma_0 = x,
\end{equation}
for every $u\in \mathcal{U} \subset L^2([0,1],\R^L)$ for which the solution is defined on $[0,1]$.

\item \emph{Sub-Riemannian geodesics} are horizontal curves associated with \emph{minimizing controls}, namely the ones that solve the constrained minimum problem
\begin{equation} \label{eq:Ju}
\min\{J(u) \mid  u \in \mathcal{U},\quad \End^{\mathscr{F}}_x(u) = y\},\qquad x,y\in M.
\end{equation}
They are the curves $\gamma:[0,1]\to M$ such that $\sfd(\gamma_t,\gamma_s)=|t-s|\sfd(\gamma_0,\gamma_1)$ for all $t,s\in[0,1]$.

\item If $u$ is a minimizing control with $\End^{\mathscr{F}}_x(u)=y$, then there exists a non-trivial pair $(\lambda_1,\nu) \in T_{y}^*M \times \{0,1\}$, called \emph{Lagrange multiplier}, such that
\begin{equation}\label{eq:multipliers}
\lambda_1 \circ D_u \End^{\mathscr{F}}_x(v)  = \nu  (u,v)_{L^2}, \qquad \forall v\in T_u\mathcal{U} \simeq L^2([0,1],\R^L),
\end{equation}
where $\circ$ denotes the composition of linear maps, $D$ the (Fr\'echet) differential. Non-trivial means that $(\lambda_1,\nu)\neq (0,0)$.

\item The multiplier $(\lambda_1,\nu)$ and the associated geodesic $\gamma$ are called \emph{normal} if $\nu = 1$ and \emph{abnormal} if $\nu = 0$ (we use the same terminology for the covector $\lambda_1$). A minimizing control $u$ may admit different multipliers so that $\gamma$ might be both normal \emph{and} abnormal. In particular we observe that $\gamma$ is abnormal if and only if $u$ is a critical point of $\End^{\mathscr{F}}_x$.

\item To any multiplier $(\lambda_1,\nu)$ we associate an \emph{extremal}, namely a curve $\lambda:[0,1]\to T^*M$, given by
\begin{equation}\label{eq:extremallift}
\lambda_t:=(P_{t,1}^{u})^*\lambda_1,\qquad \forall\,t\in [0,1].
\end{equation}
The covector $\lambda_0$ is called the \emph{initial covector} of the extremal.


\item If $\gamma:[0,1]\to M$ is a normal geodesic with minimizing control $u$, and $\lambda_1\in T_{\gamma_1}^*M$ is a normal multiplier, then the corresponding extremal lift $\lambda :[0,1]\to T^*M$ as defined in \eqref{eq:extremallift} satisfies the differential equation
\begin{equation}
\dot{\lambda}_t = \vec{H}(\lambda_t),
\end{equation}
where $H : T^*M \to \R$ is the \emph{sub-Riemannian Hamiltonian}:
\begin{equation}
H(\lambda) := \frac{1}{2}\sum_{i=1}^{L} \langle \lambda,X_i\rangle^{2}, \qquad \forall\,\lambda \in T^*M,
\end{equation}
and $\vec H$ denotes the corresponding Hamiltonian vector field.

\item If $\gamma :[0,1]\to M$ is an abnormal geodesic with minimizing control $u$, and $\lambda_1\in T_{\gamma_1}^*M$ is an abnormal Lagrange multiplier, then the corresponding extremal lift $\lambda :[0,1]\to T^*M$ as defined in \eqref{eq:extremallift} satisfies the \emph{abnormal condition}
\begin{equation}
\langle \lambda(t),X_i \rangle =0,\qquad \forall\, i=1,\dots,L.
\end{equation}

\item The \emph{exponential map} at $x \in M$ is the map $\exp_x : T_x^*M \to M$, which assigns to $\lambda_{0}\in T_x^*M$ the final point $\pi(\lambda_1)$  of the solution of 
\[
\dot{\lambda}_t = \vec{H}(\lambda_t),
\]
starting at $\lambda_0$. The curve $\gamma_t = \pi(\lambda_t)$, $t\in[0,1]$, has control $u_i(t)=\langle \lambda_t,X_i(\gamma_t)\rangle$ for $i=1,\dots,L$, and satisfies the normal Lagrange multiplier rule with multiplier $\lambda_{1}=e^{\vec H}(\lambda_{0})$.

\item Given a normal geodesic $\gamma_t = \exp_x(t\lambda_{0})$ with initial covector $\lambda_0 \in T_x^*M$
we say that $y=\exp_x(\bar{t}\lambda)$ is a \emph{conjugate point} to $x$ along $\gamma$ if $\bar{t}\lambda$ is a critical point for $\exp_x$. We call $\bar{t}$ a conjugate time. Also we say that $\gamma_s$ and $\gamma_t$ are \emph{conjugate} if $\gamma_t$ is conjugate to $\gamma_s$ along $\gamma|_{[s,t]}$.

\item A normal geodesic $\gamma:[0,1]\to M$ \emph{contains no non-trivial abnormal segments} if for every $s,s'\in [0,1]$ with $0<|s-s'|<1$, the restriction $\gamma|_{[s,s']}$ is not abnormal.
\item If a geodesic $\gamma: [0,1] \to M$ contains no non-trivial abnormal segments, then $\gamma_s$ is not conjugate to $\gamma_{s'}$ for every $s,s'\in [0,1]$ with $0<|s-s'|<1$. 
\item We say that $y\in M$ is a \emph{smooth point}, with respect to $x\in M$, if there exists a unique geodesic joining $x$ with $y$, which is not abnormal, and with non-conjugate endpoints.
\item The \emph{cut locus} $\Cut(x)$ is the complement of the set of smooth points with respect to $x$. The \emph{global cut locus} of $M$ is 
\begin{equation}
\Cut(M) := \{(x,y) \in M \times M \mid y \in \Cut(x)\}.
\end{equation}
The set of smooth points is open and dense in $M$, and the squared sub-Riemannian distance function is smooth on $M\times M \setminus \Cut(M)$ \cite{agrasmoothness,RT-MorseSard}.


\item The \emph{abnormal set} $\mathrm{Abn}(x)$ is the set of points $y$ such that there exists an abnormal geodesic joining $x$ and $y$. It holds $\mathrm{Abn}(x)\subseteq \mathrm{Cut}(x)$. 
\end{itemize}

%% file: goh.tex
\subsection{The Goh-Legendre condition}\label{a:Goh-Legendre}

Consider a sub-Riemannian structure on a smooth manifold $M$ generated by a family of vector fields $X_1,\dots,X_L\in\mathrm{Vec}(M)$. Abnormal geodesics with finite index must satisfy necessary conditions for length-minimality known as Goh-Legendre conditions, see \cite[Thms.\ 12.12, 3.59]{nostrolibro}, that we recall here. 

\begin{definition}[Goh-Legendre]\label{def:Goh-Legendre}
\emph{Goh-Legendre geodesics} are geodesics $\gamma:[0,1]\to M$, with control $u\in L^2([0,1],\R^L)$, such that there exists a never-vanishing extremal lift $\lambda : [0,1]\to M$ such that 
the following holds for almost every $t\in [0,1]$:
\begin{align}
\langle\lambda_t, X_v\rangle & = 0, \qquad\qquad \forall\, v\in \R^L, \tag{abnormal condition} \\
\langle\lambda_t,[X_v,X_w] \rangle & = 0, \qquad\qquad \forall\, v,w\in\R^L, \tag{Goh condition} \\
\langle \lambda_t,[[X_{u(t)},X_v],X_v]\rangle & \geq 0, \qquad\qquad \forall\,v\in\R^L,   \tag{generalized Legendre condition}
\end{align}
where $\langle\cdot,\cdot\rangle$ denotes the duality pairing, and $X_v := \sum_{i=1}^L v_i X_i$ for any $v\in \R^L$. We call \emph{Goh-Legendre curve} a horizontal curve admitting a lift satisfying the above conditions.
\end{definition}

We refer to \cite[Ch.\ 12]{nostrolibro} for details; we only note here that all abnormal geodesics that are not normal must have finite index so that they must be Goh-Legendre geodesics.

We also formulate a strengthening of the Goh-Legendre condition. It was introduced in \cite[(4.16)]{AS-Morse} in relation with rigidity properties of abnormal geodesics.

\begin{definition}[Strong-Goh-Legendre]\label{def:strong-Goh-Legendre}
A \emph{strong}-Goh-Legendre geodesic is a Goh-Legendre geodesic such that the Legendre condition is verified in a strong sense. Namely, in the notation of \cref{def:Goh-Legendre}, there exists $c>0$ such that for almost every $t\in [0,1]$:
\begin{equation}
\langle \lambda_t,[[X_{u(t)},X_v],X_v]\rangle  \geq c\|v\|^2, \qquad \forall\,v\in u(t)^{\perp},\qquad \text{(strong generalized Legendre condition)}
\end{equation}
where $u(t)^\perp = \{v\in \R^L \mid v\cdot u(t)=0\}$, and $\|\cdot\|$ denotes the Euclidean norm of $\R^L$.
\end{definition}

\subsection{Goh-Legendre conditions on a Carnot group}

We rewrite \cref{def:Goh-Legendre} on a Carnot group $\G$. We use the corresponding notation from \cref{sec:Carnothomogeneousspaces}. Denote by $R_g :\G\to\G$ the right multiplication by $g\in \G$. Then, assuming without loss of generality $\gamma_0=e$, for all $g\in \G$ it holds
\begin{equation}
P_{0,t}^{u}(g) = R_{\gamma_t}(g),\qquad \text{where}\qquad \gamma_t = P_{0,t}^{u}(e),
\end{equation}
where, we recall, $P_{0,t}^{u}$ denotes the flow of the non-autonomous vector field $X_{u(t)}=\sum_{i=1}^L u_i(t)X_i$ (a curve in $\g$) with initial datum at time $0$ and final time at time $t$. It follows that for any extremal lift \eqref{eq:extremallift}, letting $\lambda_0\in T_e^*\G \simeq \g^*$ be the corresponding initial covector, it holds
\begin{equation}
\langle\lambda_t,V\rangle = \langle \lambda_0,\Ad_{\gamma_t}(V)\rangle,\qquad\forall\, V\in\g,
\end{equation}
where for $g\in \G$, $\Ad_g:\g\to\g$ is the adjoint endomorphism. Thus, the conditions in \cref{def:Goh-Legendre} are equivalent to the existence of $\lambda_0 \in \g^*\setminus\{0\}$ such that
\begin{align}
\langle\lambda_0,\Ad_{\gamma_t} (X)\rangle & = 0, \qquad\qquad \forall\, X\in\g_1,  \tag{abnormal condition} \\
\langle\lambda_0,\Ad_{\gamma_t} (Y) \rangle & = 0, \qquad\qquad \forall\, Y\in\g_2,  \tag{Goh condition} \\
\langle \lambda_0,\Ad_{\gamma_t}([[X_{u(t)},X],X])\rangle & \geq 0, \qquad\qquad \forall\, X\in\g_1, & \tag{generalized Legendre condition}
\end{align}
for almost every $t\in [0,1$]. Furthermore, for the flow of a non-autonomous vector field, a Volterra series argument shows that, if $\G$ has  step $s$, it holds:
\begin{equation}\label{eq:chrono?}
\Ad_{P_{0,t}^{u}(e)}(W) = W +\sum_{i=1}^{s-1} \int_0^{t} \dots \int_0^{t_{i-1}} [X_{u(t_{i})},[\dots,[X_{u(t_1)},W]]]\,\di t_1\cdots \di t_i,\qquad\forall \, W\in\g.
\end{equation}
Note that \eqref{eq:chrono?} is the non-autonomous version of \eqref{eq:adexp}.

 Elementary arguments using \eqref{eq:chrono?}, show that in step $3$ the above conditions are equivalent to:
\begin{align}
\langle \lambda_0, X\rangle  & =0, \qquad \forall\, X\in\g_1, \tag{abnormal condition} \\
\langle\lambda_0,Y\rangle =\langle \lambda_0, [X_{u(t)},Y]\rangle & =0,\qquad \forall\, Y\in\g_2,\, a.e.\ t\in[0,1],  \tag{Goh condition}  \\
\langle \lambda_0, [[X_{u(t)},X],X]\rangle & \geq 0,\qquad \forall\, X\in\g_1,\, a.e.\ t\in [0,1].  \tag{generalized Legendre condition}
\end{align}
Note that the Goh condition implies that the bilinear map $\g_1\times\g_1\ni (V,W)\mapsto \langle \lambda_0, [[X_{u(t)},V],W]\rangle$ is a quadratic form for almost every $t\in [0,1]$. 

%% file: lemmaproba.tex
\section{Approximations of probability measures}

The following is a standard construction used, in particular, to prove that if $(X,\sfd)$ is separable then also $(\mathcal{P}(X),W_p)$ is separable. We give a proof of the precise formulation we need.

\begin{lemma}\label{lem:approxwithdelta}
    Let $(X,\sfd)$ be a metric space. Let $\mu \in \mathcal{P}(X)$. Assume that for all $m\in \N$ there exist Borel sets $A_1^{m},\dots,A_{L(m)}^{m}$, with $A_i^{m}\subset X$ such that as $m\to \infty$ we have $L(m)\nearrow\infty$ and
    \begin{enumerate}
        \item \label{lem:i:A1} $A_i^{m}\cap A_j^{m} = \emptyset$ if $i\neq j$;
        \item \label{lem:i:A2} $ \mu\left( \cup_{i=1}^{L(m)} A_i^{m}\right) \geq 1-\frac{1}{m}$;
        \item \label{lem:i:A3} $ \diam\left(A_i^{m}\right)\leq \frac{1}{m}$ for all $i=1,\dots,L(m)$.
    \end{enumerate}
    Let $y_i^{m}$ be a choice of a point from each $A_i^{m}$. Define the sequence of probability measures
    \begin{equation}\label{eq:defmun}
        \mu^{m}:= \frac{1}{\mu\left(\cup_{i=1}^{L(m)} A_i^{m}\right)}\sum_{i=1}^{L(m)} \mu\left(A_i^{m}\right) \delta_{y_i^{m}},\qquad \forall\, m\in \N.
    \end{equation}
    Then $\mu^{m} \rightharpoonup \mu$ as $m\to \infty$.
\end{lemma}
\begin{proof}
  Remember that $\mu^{m}\rightharpoonup \mu$ if for any continuous and bounded function $g\in C_b(X)$
    \begin{equation}
        \lim_{m\to\infty}\int_X g(z) \mu^{m}(\di z)= \int_X g(z) \mu(\di z).
    \end{equation}
    Since we are dealing with probability measures, it is sufficient to test the above convergence for $g\in \mathrm{Lip}_1(X,\sfd)\cap C_b(X)$, see \cite[Cor.\ 2.2.6]{Bogachev-Weak}. By our assumptions it holds
    \begin{equation}
        \lim_{m\to \infty} \mu\left(\cup_{i=1}^{L(m)} A_i^{m}\right) =1,
    \end{equation}
    so that we can omit the normalization in \eqref{eq:defmun}. Let $g\in \mathrm{Lip}_1(X,\sfd)$. The we compute
    \begin{align}
        \Big|\int_X g(z)\mu^{m}(\di z) & - \int_X g(z) \mu(\di z)\Big| = \left| \sum_{i=1}^{L(m)} \mu\left(A_i^{m}\right) g(y_i^{m})  - \int_X g(z) \mu(\di z)\right|                                                            \\
                                       & = \left| \int_X \left(\sum_{i=1}^{L(m)}\mathbbm{1}_{A_i^{m}}(z)g(y_i^{m})-g(z) \right) \mu(\di z)\right|                                                                                \\
                                       & = \left| \sum_{i=1}^{L(m)} \int_{A_i^{m}}\left(g(y_i^{m})-g(z) \right)\mu(\di z) - \int_{X\setminus \cup_{i=1}^{L(m)}A_j^{m}} g(z)\mu(\di z)   \right|                                  \\
                                       & \leq \sum_{i=1}^{L(m)} \sup_{z\in A_i^{m}}\sfd(y_i^{m},z)\mu\left(A_i^{m}\right) + \|g\|_{\infty} \mu\left(X\setminus \bigcup_{i=1}^{L(m)} A_i^{m}\right) \\
                                       & \leq  1/m + \|g\|_{\infty}/m,
    \end{align}
    where $\|g\|_{\infty}$ is the supremum of $g$. Letting $m\to\infty$ concludes the proof.
\end{proof}